\documentclass[11pt]{article}  
\usepackage{bm}
\usepackage{amscd}
\usepackage{amsmath}
\usepackage{amssymb}
\usepackage{amsthm}
\usepackage{epsf}
\usepackage{graphicx}
\usepackage{mathtools}
\usepackage{enumerate}
\usepackage{color}
\usepackage{hyperref}
\usepackage{cleveref} 
\usepackage{caption}
\usepackage{subcaption}
\usepackage{xspace}

\newtheorem{definition}{Definition}
\newtheorem{theorem}{Theorem}
\newtheorem{lemma}[theorem]{Lemma}
\newtheorem{corollary}[theorem]{Corollary}
\newtheorem{remark}{Remark}
\newtheorem{proposition}{Proposition}[section]

\newtheorem{example}[theorem]{Example}

\newcommand{\R}{\mathbb{R}}

\newcommand{\Z}{\mathbb{Z}}

\newcommand{\ve}{\varepsilon}
\newcommand{\eps}{\epsilon}

\newcommand{\lo}{\left( 1+ o(1) \right)}

\newcommand{\network}{graph\xspace}
\newcommand{\networks}{graphs\xspace}
\newcommand{\networkss}{graph's\xspace}

\newcommand{\dist}{\mu}

\newcommand\given[1][]{\:#1\vert\:}

\DeclarePairedDelimiter{\defaultDelim}{[}{]}

\DeclareMathOperator{\capPr}{\sf Pr}
\renewcommand{\Pr}[2][]{\capPr_{#1}\defaultDelim*{#2}}
\DeclareMathOperator{\capE}{\sf E}
\newcommand{\E}[2][]{\capE_{#1}\defaultDelim*{#2}}
\DeclareMathOperator{\capVar}{\sf Var}
\newcommand{\Var}[2][]{\capVar_{#1}\defaultDelim*{#2}}

\newcommand{\brac}[1]{\left(#1\right)}

\newcommand{\bfrac}[2]{\left(\frac{#1}{#2}\right)}
\crefformat{equation}{(#2#1#3)}

\DeclareMathOperator*{\argmax}{arg\,max}

\begin{document}

\title{Approximating Sparse Graphs: The Random Overlapping Communities Model}
\author{Samantha Petti\thanks{Harvard University, spetti@fas.harvard.edu. This work was done while the first author was at Georgia Tech. This material is based upon work supported by the National Science Foundation Graduate Research Fellowship under Grant No. DGE-1650044.}  \and Santosh S. Vempala\thanks{Georgia Tech, vempala@gatech.edu. Both authors were supported in part by NSF awards CCF-1563838 and
CCF-1717349.}}

\maketitle

\begin{abstract} 
How can we approximate sparse graphs and sequences of sparse graphs (with  unbounded average degree)? We consider convergence in the first $k$ moments of the graph spectrum (equivalent to the numbers of closed $k$-walks) appropriately normalized.  We introduce a simple random graph model that captures the limiting spectra of many sequences of interest, including the sequence of hypercube graphs. The Random Overlapping Communities (ROC) model is specified by a distribution on pairs $(s,q)$, $s \in \mathbb{Z}_+, q \in (0,1]$. A graph on $n$ vertices with average degree $d$ is generated by repeatedly picking pairs $(s,q)$ from the distribution, adding an Erd\H{os}-R\'{e}nyi random graph of edge density $q$ on a subset of vertices chosen by including each vertex with probability $s/n$, and repeating this process so that the expected degree is $d$. 
We also show that ROC graphs exhibit an inverse relationship between degree and clustering coefficient, a characteristic of many real-world networks.
\end{abstract}


\tableofcontents

\section{Introduction}

What is a good summary of a very large graph? Besides simple statistics like its size and edge density, one would like to know the chance of finding small subgraphs (e.g., triangles), to estimate global properties (e.g., the size of the minimum or maximum cut), and to be able to produce a smaller graph of desired size with similar properties as the original. 
One approach is to define a random graph model, a simple description of a probability distribution over all graphs, such that a graph drawn from the model will likely have similar properties as the graph of interest. Since the introduction of the Erd\H{o}s-R\'enyi random graph $G_{n,p}$ \cite{Erd59, Erd60} developing and analyzing random graphs has become a rich field with many interesting emergent phenomena.

A more powerful approach to graph approximation is via Szemer\'edi's regularity lemma \cite{sze76}, which guarantees the existence of a partition the vertices of a graph into a small number of blocks such that the distribution of edges within a block and between most pairs of blocks resembles a random graph of prescribed density. 
Frieze and Kannan's weak regularity theorem \cite{Fri96} is a simpler version of Szemer\'edi's fundamental theorem which states the number of partition classes (and hence the total size of the description of the approximation) required to produce an approximation with error $\eps>0$ in the cut norm is only a function of $\eps$, independent of the size of the original graph. Applying this version of the regularity lemma produces a much smaller description, leading to to efficient approximation algorithms.
The regularity lemma's consequences are striking --- one can approximate the homomorphism density of any fixed size graph (from the left or right), the size of any cut to within additive error; the partition itself can be constructed algorithmically and is easy to sample. This has lead to effective Stochastic Block Models for large dense graphs. The closely related theory of  graph limits shows that any sequence of dense graphs has a convergent subsequence, whose limit captures the limit of homomorphism densities and normalized cuts of the graphs, and is itself a probability distribution over the unit square (called a {\em graphon}). Moreover, if two graphs are close in the cut metric, then their homomorphism densities are also close. At a qualitative level, these theorems give an essentially complete theory for the approximation of {\em dense} graphs, where the number of edges is $\Omega(n^2)$. 

Such a theory is missing for sparse graphs (with $o(n^2)$ edges). All the approximations for the dense case produce the trivial approximation of the empty graph. While there is an intricately developed theory for bounded-degree graphs that allows one to describe the limits of sequences of such graphs\footnote{One can approximate a bounded-degree graph as a distribution over local neighborhood structures, i.e., the probability that the $r$-neighborhood of a vertex is a particular graph. For any $r$, this is a finite description and it captures homomorphism densities and appears as a limit of a bounded-degree graph sequence.}, it is not algorithmically tractable, and it does not extend to graph sequences when the degree can grow with the size of the graph. Moreover, as we will presently see, the known objects for approximating dense graphs (block models, regularity decompositions, graphons) are inherently unable to approximate sparse graphs. The main motivation for this paper is to understand what properties of sparse graphs (resp. graph sequences) can be succinctly approximated and to provide a model for them. Existing theories are limited in what they can achieve for families of graphs which are neither dense nor bounded-degree. In particular, they seem unable to answer the following representative question \cite{Lov14}:\\

{\em What is a useful limiting representation of the sequence of hypercube graphs? }\\

A full limit theory for sparse graphs would include: (i) a metric that meaningfully describes the closeness of sparse graphs, (ii) a corresponding metric space on unlabelled graphs that provides a way to compare graphs of different sizes and yields a natural notion of convergence, and (iii) limit object that succinctly represents the properties of a convergent sequence and can be used to generate graphs of varying sizes that all have properties similar to the graph sequence. 
We do not achieve this lofty goal. Instead, we explore the extent to which it is possible to succinctly model one important property of sparse graphs--- closed walk and cycle counts. Our goal is to define a random graph model, a simple description of a probability distribution over all graphs, such that a graph drawn from the model will likely have the same normalized closed walk counts as a graph or graph sequence of interest. 

 We focus on approximating the simple cycle and closed walk counts of sparse graphs and graph sequences appropriately normalized. These counts encode information about the local structure of the graph and are related to its spectral properties; the number of closed $k$-walks in a graph is equal to the $k^{th}$ moment of the graph's eigenspectrum. For dense graphs, stochastic block models and graphons approximate both homomorphism densities and cut norm. However the standard cut norm is not useful for sparse graphs as the norm tends to zero. Moreover, natural normalizations do not seem to work either, i.e., they either go to zero or distinguish hypercubes of different sizes. 
 
 Another reason we focus on cycle and walk counts rather than cuts is that approximating local structure is of particular interest in practice. A widely-used technique for inferring the structure and function of a real-world \network is to observe overrepresented motifs, i.e., small subgraphs that appear frequently. Recent work describes the overrepresented motifs of a variety of \networks including transcription regulation \networks, protein-protein interaction \networks, the rat visual cortex, ecological food webs, and the internet (WWW),  \cite{Yeg04, Alo07,Son05,Mil02}. The type of overrepresented motifs has been shown to be correlated with the \networkss function \cite{Mil02}. A model that produces \networks with high motif counts is necessary for approximating \networks whose function depends on the abundance of a particular motif.

\paragraph{Limitations of previous approaches for capturing the cycle and walk counts.}

Previous approaches do not provide a meaningful way to approximate the small cycle counts of sparse graphs. A regularity style partition or stochastic block model inherently cannot approximate the number of triangles unless the rank of the block model grows nearly linearly with the size of the graph as shown in the following simple observation.

\begin{proposition}\label{4bound}
		Let $M$ be an $n \times n$ symmetric matrix with entries in $[0,1]$ such that each row sum is at most $d$. Then the expected number of simple $k$-cycles in a graph on $n$ vertices obtained by uniformly sampling\footnote{Formally, a graph $G$ on $n$ vertices is constructed by adding an edge between each pair of vertices $i$ and $j$ independently with probability $M_{ij}$.} $M$ is at most $d^k rank(M)$.
\end{proposition}
\noindent The proposition follows by observing that the expected number of simple $k$-cycles is at most the trace of $M^k$ (see formal proof in \Cref{sec:lopa}). In particular, any rank $r$ approximation of the $d$-dimensional hypercube where each vertex has degree $O(d)$ has fewer than $O(rd^4)$ simple four-cycles, whereas the hypercube has $2^{d} d^2$ of them. 

The Benjamini-Schramm local neighborhood distribution approach requires graphs' degrees to be bounded so that the resultant metric space of local profiles is compact \cite{ben01}. Other methods designed for the sparse but not bounded degree setting do not produce a satisfactory limit object for the sequence of hypercubes.  

The theory of $L^{p}$ graphons generalizes the graphon to a range of sparse settings \cite{bor14}. While the $L^{p}$ graphon gives approximations for a generalized notion of cut metric for sparse graphs, graphs sampled from the $L^p$ graphon limit of the sequence may have very different normalized subgraph counts than the sequence (i.e. no ``counting lemma" is possible). Recently Backhausz and Szegedy developed a theory of operator convergence and applied it to operators on graphs \cite{backhausz2018action}. They show that dense graphon convergence, $L_p$ graphon convergence, and local-global convergence (which implies Benjamini-Schramm  convergence) are all special cases of operator convergence arising from different graph operators.  With respsect to one notion of convergence in their framework, the limit of the sequence of hypercube graphs is a Cayley graph of $\Z_2^\infty$. In order to model the limit of the hypercube sequence, Frenkel redefined homomorphism density with a different normalization based on the size of the subgraph, but this notion does not help distinguish the limiting number of non-tree structures for sequences of graphs with degree tending to infinity \cite{fre16}.  
The recently developed notion of {\em graphex} \cite{vei16, bor17} is the limit object for sequences of sampling convergent graphs.  However, any sequence of nearly $d$-regular graphs with $d=o(n)$ is sampling convergent since the sampled object according to this notion is a set of isolated edges with high probability. Therefore the graphex cannot distinguish between different graph sequences that are nearly regular.


\paragraph{Normalizing closed walk and cycle counts.} In order to meaningfully compare the closed walk counts and cycle counts between graphs of different sizes, it is necessary to normalize the counts. For dense graphs, homomorphism density of a subgraph $H$ in a graph on $n$ vertices is the number of copies of a subgraph $H$ divided by $n^{|v(H)|}$. This normalization is natural because it gives the probability $H$ is present on a random subset of vertices. For the sparse case, this normalization causes the homomorphism density of all subgraphs tend to zero, and so we must define a different normalization.

When considering a sequence of graphs, we can find a proper normalization of the closed walk counts by looking at the rate of growth of the counts. A graph that locally looks like a $d$-ary tree has approximately $d^{k/2}$ closed $k$-walks at each vertex for $k$ even. Therefore the appropriate normalization of the closed $k$-walk counts for a sequence of such graphs is $nd^{k/2}$. We will see in \Cref{convergence examples} that this normalization is also natural for the sequence of hypercubes. A sequence of sparse graphs in which each vertex's local neighborhood is dense (e.g., a collection of $d$ cliques of size $d$), the appropriate normalization for the walk counts is $n d^{k-1}$. We define the {\em sparsity exponent} of a sequence to measure the rate of growth of the number of closed $k$-walks in the sequence.
Let $W_k(G)$  be the number of closed walks of length $k$ in a graph $G$.
\footnote{Denoting the number of closed $k$-walks by $W_k(G)$, and the eigenvalues of the adjacency matrix of $G$ as $\lambda_1(G) \ge  \lambda_2(G),\dots \ge \lambda_n(G)$, we have $W_k(G) = \sum_{i=1}^n \lambda_i^k.$}

\begin{definition} \label{cwc-defn}
	For $0 < \alpha \le 1$, we define the $\alpha$ normalized closed $k$-walk count as $$W_k(G,\alpha)= \frac{W_k(G)}{nd^{1+\alpha(k-2)}}$$ where $n=|V(G)|$ and $d$ is the average degree of $G$.
\end{definition}

\noindent Throughout we use $(G_i)$ to denote a sequence of finite graphs indexed by the positive integers. 
\begin{definition}[sparsity exponents] \label{sparsity-exponent-defn}
	Let $(G_i)$ be a sequence of graphs. Let	$$\alpha= \inf_{b \in [1/2,1]} \big\{ b \given   \lim_{i \to \infty } W_j(G_i,b)  \text{ exists for all $j$} \big\}$$ be the sparsity exponent of the sequence.
	For $k\geq 3$, let
	$$\alpha_k= \inf_{b \in [1/2,1]} \big\{ b \given   \lim_{i \to \infty } W_j(G_i,b)  \text{ exists for all $j \leq k$} \big\}$$ be the $k$-sparsity exponent of the sequence. 
\end{definition}

 We define the minimum of the sparsity exponent to be $1/2$ because all $d$-regular graphs have at least $\sf{Cat}_{k/2}$$ nd^{k/2}$ closed $k$-walks obtained from tracing trees, where $\sf{Cat_n}$$={ 2n \choose n} /(n+1)$ denotes the $n^{th}$ Catalan number.  For two sequences of graphs with matching degrees, a higher sparsity exponent indicates denser local neighborhoods and therefore more closed walks.

 While the sparsity exponent gives a natural way to normalize the closed walk counts for a sequence of sparse graphs, it does not help determine the appropriate normalization factor for approximating an individual graph (not contextualized in a sequence). When approximating an individual graph, we instead choose to focus on the number of simple $k$-cycles denoted $C_k(G)$ and normalize by the number of edges in the graph.\footnote{ In \Cref{sec:lopa} we show that the naive approach of adding cycles on random vertices does not produce graphs with high cycle-to-edge ratios}  Throughout this paper, for convenience we refer to a simple $k$-cycle as a $k$-cycle. For example, under this convention each triangle is counted 6 times because there are 6 closed walks that traverse a triangle.

 As described in the previous section, approximating $W_k(G)$ and $C_k(G)$ for sparse graphs is already out-of-reach for known methods that work well in the dense and bounded-degree settings. The main contribution of this paper is the following model which can approximate the normalized closed walk and cycle counts for a large class of graphs.

\paragraph{The Random Overlapping Communities Model.}
We introduce a simple generalization of the Erd\H{o}s-R\'enyi model that can approximate the normalized cycle and walk counts of a wide range of sparse graphs. The {\em Random Overlapping Communities (ROC)} model generates \networks that are the union of many relatively dense random communities. A {\em community} is an instance of an Erd\H{o}s-R\'enyi graph $G_{s,q}$ (or a bipartite Erd\H{o}s-R\'enyi graph $G_{s/2,s/2,q}$) on a set of $s$ randomly chosen vertices. A ROC \network is the union of many randomly selected communities that overlap, so  every vertex is a member of multiple communities.  The number, size and density of communities are drawn from a distribution.  \Cref{ROC picture} illustrates this construction.


\begin{figure}[h] 
	\centering
	\includegraphics[scale=.5]{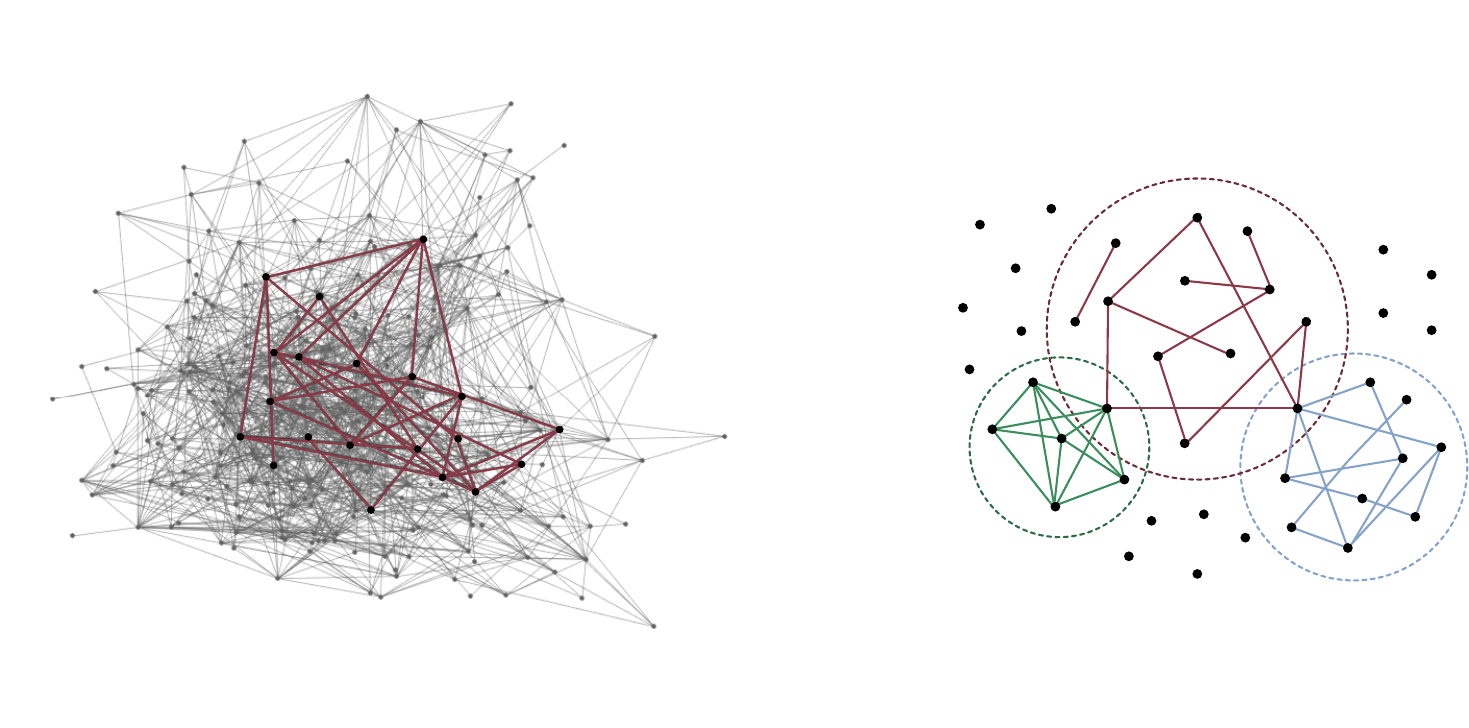}
	\vspace{-10pt}
	\caption{Left: in each step of the construction of a ROC($n,d,s,q$) \network, an instance of $G_{|S|,q}$ is added on a set $S$ of randomly selected vertices. Right: three communities of a ROC graph.}
	\label{ROC picture}
\end{figure}

The distribution $\mathcal{D}$ of a ROC family is specified by a number $a \in [0,1]$ and a distribution $\mu$ on triples $(m_i, q_i,\beta_i)$ with probability $\mu_i$ for the $i^{th}$ triple. Communities are generated by repeatedly picking a triple from the distribution $\mu$. When $\beta_i=0$, the community has expected size $s=m_id^a$ and density $q_i$. If $\beta_i=1$, indicating balanced bipartiteness, the community is defined on a bipartite graph with $m_id^a$ vertices expected in each class.  \\


\begin{centering}
	\fbox{\parbox{0.98\textwidth}{
			{\bf ROC($n,d,\mathcal{D}$)}.  \\
			\textit{Input:} number of vertices $n$, degree $d$, $\mathcal{D}= (a, \mu)$ where $a \in [0,1]$, and $\mu$ a distribution on a finite set of triples $(m_i, q_i, \beta_i)$ where $(m_i, q_i, \beta_i)$ is selected with probability $\mu_i$, $ m_i>0$, $\sum_i \mu_i =1$, $\beta_i \in \{0,1\}$, $0 \leq q_i \leq 1$, and $\max m_i d^a \leq n$. Let $B$ be the set of indices $i$ such that $\beta_i=1$ and $B^c$ be the set of indices $i$ such that $\beta_i=0$. Let  $x= 1/ \brac{\sum_{i\in B^c} \mu_i m_i^2 q_i+2\sum_{i\in B} \mu_i m_i^2 q_i}$.\\
			\\
			\textit{Output:} a \network on $n$ vertices with expected degree $d$ (when $xnd^{1-2a} $ is integer).\\
			\\
			Repeat $\lfloor xnd^{1-2a} \rfloor$ times:	
			\begin{enumerate}
				\item Randomly select a pair $(m_i, q_i, \beta_i)$ from $\dist$ with probability $\mu_i$. 
				\item If $\beta_i=0$
				\begin{enumerate}
					\item Pick a random subset $S$ of vertices (from $\{1,2,\ldots,n\}$) by selecting each vertex independently with probability $m_id^a/n$. 
					
					\item Add the random graph $G_{|S|,q_i}$ on $S$, i.e., for each pair in $S$, add the edge between them independently with probability $q_i$; if the edge already exists, do nothing.   
				\end{enumerate}
				If $\beta_i=1$
				\begin{enumerate}
					\item Pick a random subset $S$ of vertices (from $\{1,2,\ldots,n\}$) by selecting each vertex independently with probability $2m_id^a/n$. For each vertex that is in $S$ randomly assign it to either $S_1$ or $S_2$. 
					
					\item Add the bipartite random graph $G_{|S_1|,|S_2|,q_i}$ on $S$, i.e., for each pair $u \in S_1$ and $v \in S_2$, add the edge between them independently with probability $q_i$; if the edge already exists, do nothing.   
				\end{enumerate}
				
			\end{enumerate}
		
		 	\textit{Note:} For certain applications, we consider a simple version of the model in which one community type is added. For ease of notation, in this case we denote the model $ROC(n,d,s,q)$ where $s$ is the community size and $q$ is the density. Such a parameterization equivalent to $ROC(n,d, \mu, a)$ when $a=0$ and $\mu$ is a probability distribution with supported on the triple $(s,q,0)$. 
		}
	}
\end{centering}
\bigskip

The parameters $(n,d, \mathcal{D})$ are valid ROC parameters if the conditions described under {\em input} in the box hold. A ROC family $\mathcal{D}=(\mu, a)$ refers to the set of ROC models with parameters $\mu$ and $a$ and any valid $n$ and $d$. The sparsity exponent of a sequence determines the parameter $a$ of the ROC family that achieves the limit vector. 
If a vector is achievable with sparsity exponent $\alpha$ then the ROC family that achieves the vector will have parameter $a=\alpha$ unless $\alpha=1/2$ and vector is the Catalan vector ($w_j=0$ for $j$ odd and $w_j=\sf{Cat_{j/2}}$ for $j$ even, where $\sf{Cat_n}$$={ 2n \choose n} /(n+1)$ denotes the $n^{th}$ Catalan number.) In this case any ROC family with $a<1/2$ achieves the vector. 

\begin{remark}
	The ROC model defines a probability distribution over multi-graphs since it is possible for the same edge to be placed multiple times. In a ROC graph with $\Theta(nd)$ edges, the expected number of multi-edges is $\Theta(d^2)$.
\end{remark}

\begin{remark}
In this paper we will study the behavior of the ROC model as $n$ and $d$ tend to infinity. For this reason we ignore the floor function in our asymptotic computations. 
\end{remark}

\paragraph{ROC as a model for real-world networks.} 
We focus on approximating walk and cycle counts with the ROC model in the main text. In \Cref{sec:rw} we demonstrate that the ROC model may be of further use in real-world contexts. The model can be tuned to exhibit high clustering coefficient (the probability two randomly selected neighbor of a random vertex are adjacent). Moreover, ROC graphs exhibit an inverse relationship between degree and clustering coefficient, a phenomenon often observed in real-world network (\Cref{sec:cc}). In addition, we introduce a variant of the model that produces graphs with varied degree distributions (\Cref{sec:droc}). Finally, we compare the ROC model to existing models used in practice and discuss the utility of the model in practice (\Cref{other models,sec:rw dis}).

\paragraph{Comparison to block models.} 
Mixed membership stochastic block models have traditionally been applied in settings with overlapping communities  \cite{Air08}, \cite{Kar11}, \cite{Air06}. The ROC model differs in two key ways. First, unlike low-rank mixed membership stochastic block models, the ROC model can produce sparse \networks with high triangle and four-cycle ratios. As discussed in the introduction, the over-representation of particular motifs in a \network is thought to be fundamental for its function, and therefore modeling this aspect of local structure is important. Second, in both stochastic block models and the community affiliation graphs studied in \cite{yang2012community} (which exhibit overlapping communities), the size and density of each community and the density between communities are all specified by the model. 
As a result, the size of the model must grow with the number of communities, but the ROC model maintains a succinct description even as the number of communities grows. This observation suggests that the ROC model may be better suited for \networks in which there are many communities that are similar in structure, whereas the stochastic block model is better suited for \networks with a small number of communities with fundamentally different structures. In \Cref{other models} we  compare ROC and its extension that allows for varied degree distributions to more random graph models. 

\paragraph{Notation.} Throughout the paper, $C_k(G)$ denotes the number of simple cycles in a graph $G$, and $W_k(G)$ denotes the number of closed walks in a $G$. The lowercase $c_k$ and $w_k$ refer to normalized simple cycle and closed walk counts respectively. The normalization is context dependent. We use $(G_i)$ to denote a sequence of finite graphs indexed by the positive integers. Finally $\sf{Cat_n}$$={ 2n \choose n} /(n+1)$ denotes the $n^{th}$ Catalan number. 

\paragraph{Organization.} The paper has two main objectives: to show that ROC is an effective approximation for individual sparse graphs (\Cref{ratios} and \Cref{sec:rw}) and to develop a notion of convergence for which the ROC model approximates the limit (\Cref{convergence examples,sec:approx,achievable vectors}). We end with a discussion of limitations of the model, possible extensions, and open questions (\Cref{discussion}). In the remainder of this section we summarize the results. 

\subsection{Approximating $k$-cycle-to-edge ratios}\label{ratios}


First we give conditions for when a vector of $k$-cycle-to-edge ratios can be achieved by ROC family and show that almost all pairs of triangle-to-edge and four-cycle-to-edge ratio can be approximated with a ROC graph with one community type. 

The expected simple cycle counts of a ROC graph are the moments of a distribution determined by the ROC parameters. Therefore, determining which vectors of $k$-cycle-to-edge ratios can be achieved by ROC families is closely related to the the following variants of the Stieltjes' moment problem \cite{stieltjes1894recherche}: given a sequence (or truncated sequence) determine whether there exists a discrete distribution with finite positive support with that moment sequence (or truncated moment sequence)?  The solution is characterized by the definition below (see \Cref{stieltjes,stieltjes classic}).

\begin{definition} [Stieltjes conditions] \label{s con}
	The Hankel matrices of a sequence $\mu$ are $$ H_{2s}^{(0)}=
	\begin{pmatrix} 
	\mu_0 & \mu_1 & \dots && \mu_s \\
	\mu_1 &&&&\\
	\vdots && \ddots && \vdots \\
	&&&&\\
	\mu_s && \dots && \mu_{2s} 
	\end{pmatrix}
	\quad
	\text{  and  }
	\quad
	H_{2s+1}^{(1)}=
	\begin{pmatrix} 
	\mu_1 & \mu_2 & \dots && \mu_{s+1} \\
	\mu_2 &&&&\\
	\vdots && \ddots && \vdots \\
	&&&&\\
	\mu_{s+1} && \dots && \mu_{2s+1} 
	\end{pmatrix}.
	$$
	
	\begin{enumerate}
		\item 	The vector $\mu=(\mu_0, \mu_1, \dots \mu_n)$ satisfies the truncated Stieltjes condition if 
		\begin{itemize}
			\item for $n= 2j+1$ odd, $H_{2j}^{(0)}\succeq 0$, $H_{2j+1}^{(1)} \succeq 0$, and $(\mu_{j+1}, \mu_{j+2}, \dots \mu_{2j+1}) \in $ ${\sf ColSpace}$ $H_{2j}^{(0)}$
			\item for $n= 2j$ even, $H_{2j}^{(0)}\succeq 0$, $H_{2j-1}^{(1)} \succeq 0$, and $(\mu_{j+1}, \mu_{j+2}, \dots \mu_{2j}) \in $ ${\sf ColSpace}$ $H_{2j-1}^{(0)}$,
		\end{itemize}
		where ${\sf ColSpace}$ denotes the span of the columns of the matrix. 
		
		\item The infinite vector $\mu=(\mu_0, \mu_1, \dots )$ satisfies the full Stieltjes condition with parameter $k$ (for $k \in \Z^+$) if  
		\begin{equation*}
		det \brac{H_{2s}^{(0)}} > 0  \quad  \text{ and } \quad  
		det \brac{H_{2s+1}^{(1)}} > 0 \text{ for all } s<k,
		\end{equation*}
		and 
		$$det \brac{H_{2s}^{(0)}} = 0  \quad  \text{ and } \quad  
		det \brac{H_{2s+1}^{(1)}} = 0 \text{ for all } s\geq k.$$
	\end{enumerate}

\end{definition}

\begin{lemma}[Cycle Ratio Achievability]\label{thm:kratios}
	There exists a ROC family such that for $G\sim  ROC(n,d,\mathcal{D})$ with $d=o \brac{n^{1/(k-1)}}$  
	$$\lim_{n \to \infty} \frac{2 \E {C_j(G)}}{nd}= c_j \quad \text{ for } 3 \leq j \leq k$$ if and only if there exists $ \gamma \in [0,1]$, $s_0, s_1, s_2, \dots s_k, t_0, t_2, \dots t_{2 \lfloor \frac{k}{2} \rfloor}  \in \R^{+}$, $s_2, t_2 \leq 1$ such that $(s_0, s_1, s_2,\dots s_k)$ and  $(t_0, t_2, \dots t_{2 \lfloor \frac{k}{2} \rfloor} )$ satisfy the truncated Stieltjes condition and for $3 \leq j \leq k$
	$$c_j/2 = \begin{cases} \gamma s_j & j \text{ odd}\\ \gamma s_j +(1-\gamma) t_j & j \text{ even}. \end{cases} $$ 
\end{lemma}

We can use this lemma to show that for almost all pairs of triangle-to-edge and four-cycle-to-edge ratios arising from some \network, there exists a ROC family that produces graphs with the matching cycle-to-edge ratios {\em simultaneously}. Instead, we directly prove a stronger statement (\Cref{real network bounds}), which gives an explicit construction of the ROC family with a single community size $s$ and density $q$ that produces \networks with the pair of ratios. Moreover, we show that the vanishing set of triangle and four-cycle ratio pairs not achievable exactly can be approximated to within a small error.

\begin{theorem}\label{real network bounds}
	Let $H$ be a graph.
	\begin{enumerate} 
		\item  Let $c_i=C_i(H)/|E(H)|$ for $i=3,4$. Then $c_3(c_3/2-1) \le c_4$. 
		\item \label{match} For any $c_3$ and $c_4$ such that $c_3^2\le 2c_4$, and $d= o(n^{1/3})$, the random \network\\
		$G \sim ROC\left(n,d, s,q\right)$ where $s=\frac{2c_4^2}{c_3^3}$ and $q=\frac{c_3^2}{2c_4}$  has $$\lim_{n \to \infty} \frac{2\E{C_3(G)}}{nd}= c_3 \quad \text{ and } \quad  \lim_{n \to \infty} \frac{2\E {C_4(G)}}{nd}= c_4.$$
	\end{enumerate}
\end{theorem}

Here we outline the proof of \Cref{real network bounds} Statement 2 to give further intuition for the ROC model and the proof strategies to come. (See \Cref{4 lim} for the full proof.) The first ingredient in many of our proofs is computing the expected number of closed cycles and walks in ROC graphs (\Cref{sec:approx}). A special case \Cref{exp by class cor} describes the cycle-to-edge ratio as a simple expression of community size $s$ and density $q$. For 
$G \sim ROC(n,d,s,q)$,  \begin{equation} \label{counts}
c_k=\lim_{n \to \infty} \frac{2\E{C_k(G)}}{nd}= 2s^{k-2}q^{k-1} \mbox{ for } d=o(n^{1/(k-1)}).
\end{equation}

Note that by varying $s$ and $q$, we can construct a ROC \network that achieves any ratio of triangles to edges or any ratio of four-cycles to edges. Moreover, it is possible to achieve a given ratio by larger, sparser communities or by smaller, denser communities. For example communities of size 50 with internal density 1 produce the same triangle ratio as communities of size 5000 with internal density 1/10.

When $c_3^2 \leq 2c_4$ we can find $s$ and $q\leq 1$ satisfying \Cref{counts} for $k=3$ and $k=4$ simultaneously.  For every \network with ratios in the narrow range $c_3(c_3/2-1) \le c_4 \le c_3^2/2$, there exists a ROC construction that matches $c_3$ and can approximate $c_4$ by $c_3^2/4$, i.e., up to an additive error $c_3/8$ (or multiplicative error of at most $1/(c_3/2 -1)$ which goes to zero as $c_3$ increases).

\subsection{The Walk Count Achievability Lemma} \label{achievability-lemma}

The walk count achievability lemma characterizes whether a vector can be realized as limiting normalized closed walk counts of ROC graphs. 
To state our results, we first define the convergence of sparse graph sequences and their limits. We consider convergence first for each $k$ and then for all  positive integers $k$, referring to the latter as full convergence. 
Recall the definitions of sparsity exponent and closed walk counts (\Cref{cwc-defn,sparsity-exponent-defn}).

\begin{definition}[$k$-convergent] 
	Let $(G_i)$ be a sequence of graphs with $k$-sparsity exponent $\alpha_k$. The sequence $(G_i)$ is $k$-convergent if $\lim_{i \to \infty } W_j(G_i,\alpha_k)$ exists for all $j\leq k$. We let $w_j= \lim_{i \to \infty } W_j(G_i, \alpha_k)$ and say  $(w_3, w_4, \dots, w_k)$  is the $k$-limit of the graph sequence $(G_i)$.  
\end{definition}

\begin{definition}[fully convergent]
	Let $(G_i)$ be a sequence of graphs with sparsity exponent $\alpha$. We say the sequence is fully convergent if $\lim_{i \to \infty } W_j(G_i,\alpha)$ exists for all $j$. We let $w_j= \lim_{i \to \infty } W_j(G_i,\alpha)$ and say  $(w_3, w_4, \dots)$  is the limit of the graph sequence $(G_i)$.  
\end{definition}

Informally, we say that a ROC family (a distribution on triples) achieves the limit of a convergent sequence of graphs $(G_i)$ if the  normalized expected number of walks in a graph drawn from the ROC family matches the limit of $(G_i)$.  We use {\em achievable} to describe when a ROC family realizes a $k$-limit, {\em fully achievable} to describe when a ROC family realizes a limit, and {\em totally $k$-achievable} to describe the weaker notion that any subsequence of a limit is achievable by a ROC family.

\begin{definition}[$k$-achievable, totally $k$-achievable, fully achievable]\label{achievability defn}\quad
	\begin{enumerate}
		\item The $k$-limit $(w_3, w_4, \dots w_k)$ of a sequence of graphs with sparsity exponent $\alpha$ is {\em $k$-achievable by ROC} if there exists a ROC family $\mathcal{D}$ such that for all $3 \leq j \leq k$, when $d \to \infty$ and $d=o \brac{n^{1/((1-a)k+2a-1)}}$ 
		$$\lim_{n \to \infty} \frac{ \E { W_j\brac{ROC(n,d,\mathcal{D})}}}{nd^{1+\alpha(j-2)}}=w_j.$$ 
	
			\item  The limit of a sequence of graphs {\em totally $k$-achievable by ROC} if every $k$-limit of the sequence is achievable (possibly with a different choice for each $k$).
		
		\item 	The limit $(w_3, w_4, \dots)$ of a sequence of graphs with sparsity exponent $\alpha$ is {\em fully achievable by ROC} if there exists a ROC family $\mathcal{D}$ such that for all $j \geq 3$,
		$$\lim_{n \to \infty} \frac{ \E { W_j \brac{ ROC(n,d,\mathcal{D}) } }}{nd^{1+\alpha(j-2)}}=w_j$$
		whenever (i) $a<1$, $d \to \infty$ and $d=o\brac{n^\ve}$ for all $\ve>0$ or  (ii) $a=1$ and $d= o(n)$.

	\end{enumerate}
\end{definition}

\noindent Roughly speaking, the degree upper bounds ensure that the overwhelming majority of simple cycles are contained entirely in single communities. 
We justify the normalization with respect to  {\em expected} degree in the above definition by showing that the probability that the normalized closed walk counts of $G \sim ROC(n,d,\mathcal{D})$ deviate from the family's limit vanishes as $d \to \infty$ (\Cref{deviate} of \Cref{roc sec sec}). Moreover \Cref{roc seq} gives conditions on $n_i$ and $d_i$ which guarantee that a sequence $(G_i)$ with $G_i \sim ROC(n_i, d_i, \mathcal{D})$ almost surely converges to limit vector achieved by the family.


In ROC families that produce sequences of graphs with sparsity exponent greater than $1/2$, the counts of simple cycles dominate the total closed walk counts. Therefore, a limit is achievable when it is possible to construct a ROC family with normalized simple cycle counts that match the desired normalized closed walk counts. The cycle counts are dominated by the cycles contained entirely in one community. Every community contributes even cycles, but only the non-bipartite communities contribute to the  odd cycle counts. In the following lemma, the parameters $s_i$ and $t_i$ count the number of simple $i$-cycles in non-bipartite and bipartite communities respectively, and the parameter $\gamma$ indicates the expected fraction of communities that are non-bipartite.

Approximating sequences with sparsity exponent $1/2$ is more complicated because the number of simple cycles can be of the same order as the number of closed walks that are not simple cycles in ROC families that produce graph sequences with sparsity exponent $1/2$.
In  \Cref{polynomials} we prove that the polynomial $T$ given in \Cref{relationship} describes the relationship between simple cycle counts and closed walks counts in ROC graphs. Moreover we show that $T$ describes the relationship between simple cycle counts and closed walk counts in {\em locally regular} graphs in which each vertex is in the same number of cycles.  A limit with sparsity exponent 1/2 is achievable when it is possible to construct a ROC family with normalized simple cycle counts that match the inverse  of this  polynomial $T$ applied to the desired normalized closed walk counts.

\begin{lemma}[Walk Count Achievability]\label{ach} \quad
	\begin{enumerate} 
		\item \label{more than half} \emph{(Sparsity exponent $>1/2$)}
		A limit vector $(w_3, w_4, \dots w_k)$ is achievable by ROC with sparsity exponent greater than $1/2$ if and only if there exists $ \gamma \in [0,1]$, $s_0, s_1,  \dots s_k, t_0,$ $ t_2, \dots t_{2 \lfloor \frac{k}{2} \rfloor}  \in \R^{+}$, $s_2, t_2 \leq 1$ such that $(s_0, s_1, s_2,\dots s_k)$ and  $(t_0, t_2, \dots t_{2 \lfloor \frac{k}{2} \rfloor} )$ satisfy the truncated Stieltjes condition and for $3 \leq j \leq k$
		$$w_j = \begin{cases} \gamma s_j & j \text{ odd}\\ \gamma s_j +(1-\gamma) t_j & j \text{ even}. \end{cases} $$ 


		\item \label{half} \emph{(Sparsity exponent $1/2$)}
		Let $T((c_3, c_4, \dots c_k))=(w_3, w_4, \dots w_k)$ be the transformation of a vector given in \Cref{relationship}. The limit vector $(w_3, w_4, \dots w_k)$ is achievable by ROC with sparsity exponent $1/2$ if and only if there exists $ \gamma \in [0,1]$, $s_0, s_1, s_2, \dots s_k, t_0, t_2, \dots t_{2 \lfloor \frac{k}{2} \rfloor}  \in \R^{+}$, $s_2, t_2 \leq 1$ such that $(s_0, s_1, s_2,\dots s_k)$ and  $(t_0, t_2, \dots t_{2 \lfloor \frac{k}{2} \rfloor} )$ satisfy the truncated Stieltjes condition and for $3 \leq j \leq k$
		$$c_j = \begin{cases} \gamma s_j & j \text{ odd}\\ \gamma s_j +(1-\gamma) t_j & j \text{ even}. \end{cases} $$ 
	\end{enumerate}
\end{lemma}


The analogous theorem for the full achievability of limits by ROC require the full Stieltjes condition. See \Cref{full-ach} in \Cref{achieve con}. 
We prove  \Cref{ach,full-ach,thm:kratios} in \Cref{achieve con}

\begin{remark} In the proofs of \Cref{ach,thm:kratios} (given in \Cref{achieve con}), we show that scale of the community sizes needed to match normalized closed walks and cycle-to-edge ratios differ substantially. A ROC family that achieves a normalized closed walk count limit must be parameterized so that the community sizes grow with $d^a$ for some constant $a \in [1/2,1]$ that depends on the sparsity exponent. In contrast, a ROC family that approximates a vector of $k$-cycle-to-edge ratios will have constant community sizes.
\end{remark}

\subsection{Applications of the Walk Count Achievability Lemma to sparse graph sequences.}

We apply the Walk Count Achievability Lemma to show that the ROC model captures the closed walk limits for several interesting sequences of graphs. We begin with the limit of the sequence of hypercube graphs, answering the question raised by \cite{Lov14}.

\begin{theorem}\label{h1}
The limit of the sequence of hypercube graphs is totally $k$-achievable by ROC.
\end{theorem}

To prove theorem, we first compute the closed $k$-walk limit of the hypercube sequence (\Cref{hypercube}). The most challenging part of the proof is establishing that the $k$-cycle limit satisfies the Stieltjes condition described in the Walk Count Achievability Lemma for sparsity exponent $1/2$ (\Cref{ach}). To do so, we establish a simple criterion that describes when a vector can be extended in a way that satisfies the truncated Stieltjes condition (\Cref{char}). (This criterion may be of independent interest.) Then we derive a recursive formula for the limiting $k$-cycle counts of the hypercube sequence (\Cref{cube cycles}) and show that this sequence satisfies our simple criterion (\Cref{s prop}). \Cref{h1} generalizes to sequences of Hamming cubes and Cayley graphs of $(\Z \mod k\Z)^d$ (\Cref{hc generalizations} of \Cref{gen hyper}). These sequences have the same limit as the hypercube sequence and therefore are achieved by the same ROC family. 

The next theorem is about a sequence of strongly regular graphs called rook's graphs (the Cartesian product of two complete graphs, see \Cref{rook}).  We prove \Cref{r1} in \Cref{sec:hc and rook} by giving a parameterization of the ROC model that achieves the limit.

\begin{theorem}\label{r1} 
The limit of the sequence of rook's graphs is fully achievable by ROC. 
\end{theorem}

In \Cref{sec:random seq} we define convergence for sequences of random graphs, and show that the limit of some sequences of  Erd\H{o}s-R\'{e}nyi random graphs (\Cref{er ex}), cannot be achieved exactly by ROC familes, but they can be approximated to arbitrarily small error.

The following theorem establishes that all $4$-limits can be achieved by a ROC model (proved in \Cref{4 lim}). However not all $k$-limits can be achieved. In \Cref{not achievable} we give an example of a graph sequence with a 6-limit that cannot be achieved by a ROC family. 

\begin{theorem}\label{three four}
	 The limit $(w_3, w_4)$ of any convergent sequence of graphs with increasing degree is achieved by a ROC family.
\end{theorem}

\section{Convergent sequences of sparse graphs} \label{convergence examples}

We give sequences that illustrate the notions of convergence. \Cref{random seq} focuses on the convergence of sequences of random graphs. Not all sequences of graphs converge or contain a convergent subsequence according to our defintion; see \Cref{non convergent} for an example. 
We begin with the hypercube sequence, which directly motivates this paper.
\begin{lemma}[hypercube limit]\label{hypercube}  The $d$-dimensional hypercube is a graph on $2^d$ vertices, each labeled with a string in $\{0,1\}^d$. Two vertices are adjacent if the Hamming distance of their labels is $1$. Let $(G_d)$ be the sequence of $d$-dimensional hypercubes.
	The sparsity exponent of the sequence $(G_d)$ is $1/2$ and the sequence is fully convergent with limit  $(w_3, w_4, \dots)$ where\footnote{ Throughout $n!!$ denotes the double factorial, $n!!=n\cdot (n-2) (n-4) \dots 2$ for even $n$ and $n!!=n\cdot (n-2) (n-4) \dots 1$ for odd $n$.} 
	$$w_k = \begin{cases} 
	(k-1)!! & \text{for $k$ even}\\
	0 & \text{for $k$ odd}.
	\end{cases}
	$$ 
	The $k$-limit of the sequence is $(w_3, w_4, \dots w_k)$. 
\end{lemma}

\begin{proof} We claim that for $k$ even $W_k(G_d)= (k-1)!! n d^{k/2} + o\brac{ n d^{k/2}}$ where $n=2^d$. Each hypercube edge $(u,v)$ corresponds to a one coordinate difference between the labels of $u$ and $v$. We think of  $k$-walks on the hypercube as length $k$ strings where the $i^{th}$ character indicates which of the $d$ coordinates is changed on the $i^{th}$ edge of the walk. In closed walks each coordinate that is changed is changed back, so every coordinate appearing in the corresponding string appears an even number of times. Therefore at most $k/2$ coordinates appear in the string.  Let $Y_i$ be the number of length $k$ strings with $i$ distinct characters that correspond to a closed $k$-walk. Since there are $d$ possible coordinates, there are ${d \choose i}$ ways to select the $i$ characters and so $Y_i = \Theta(d^i)$. Therefore $$W_k(G_d)= nY_{k/2} + o\brac{ nd^{k/2}}.$$ There are ${ d \choose k/2}$ ways to select the coordinates to change and $k!/2^{k/2}$ length $k$ strings where $k/2$ characters  appear twice. Thus $$Y_{k/2}= d^{k/2} \frac{k!}{2^{k/2} \bfrac{k}{2}!}+ o\brac{d^{k/2}}= (k-1)!! d^{k/2}+ o\brac{d^{k/2}},$$
	and the claim follows. 
	
	Note that there are no odd closed walks in the hypercube because it is bipartite. Therefore $W_k(G)=0$ for $k$ odd. It follows that the sparsity exponent is $1/2$ and the limit vector is as stated. \end{proof}
Our second example is a strongly regular family with a different sparsity exponent. The lemma follows directly from the characterization of the spectra of strongly regular graphs given in \cite{bro12}, see \Cref{sec:rook}.

\begin{lemma}[rook's graph limit]\label{rook}
	The rook graph $G_k$ on $k^2$ vertices is the Cartesian product \footnote{The Cartesian graph product is defined ``the graph product" on p. 104 of \cite{beineke2004topics}.} of two cliques of size $k$. (Viewing the vertices as the squares of a $k \times k$ chessboard, the edges represent all legal moves of the rook.) Let $(G_k)$ be the sequence on rook graphs. The sparsity exponent of $(G_k)$ is $1$ and the sequence is fully convergent with limit $(w_3, w_4, \dots )$ where $w_j= 2^{2-j}$.
\end{lemma}

\Cref{h1,r1} establish that the closed walk limits of the hypercube and rook graph sequences can be achieved by the ROC model. Proofs of these theorems are given in \Cref{sec:hc and rook}.

\section{Closed walk counts of ROC graphs}\label{sec:approx}
We compute the expected closed walk counts for ROC graphs, which determine the limit vector each ROC family achieves. We use this to prove \Cref{polynomials} which describes the limit vector each ROC family achieves.

We begin by describing a combinatorial relationship between closed walk counts and simple cycle counts that appears in graphs in which each vertex is in approximately the same number of simple cycles (\Cref{relationship}). 

\begin{definition}[cycle-walk transform] \label{relationship}
	Let $$\mathcal{S}_k= \left\{\{ (a_1, t_1), (a_2, t_2), \dots, (a_j, t_j) \} \given \sum_{i=1}^j a_i t_i=k, a_i \not = a_j \text{ for } i \not=j, a_i, t_i \in \Z^+, a_i >1 \right\}.$$ Define $T((c_3, c_4, \dots c_n))=(w_3, w_4, \dots w_n)$ as the transform
	$$w_k=\sum_{S \in \mathcal{S}_k }  \frac{k!}{(\prod t_i!)(k+1- \sum t_i)!}\prod_{i=1}^j (c_{a_i})^{t_i}.$$
	The transform $T$ is analogously defined for infinite count vectors. 
\end{definition}

\begin{remark}
 The first few terms of $T$ are illustrated below:
	\begin{align*}
	w_3&=c_3\\
	w_4&=2+c_4\\
	w_5&=c_5+5c_3\\
	w_6&=c_6+6c_4+3c_3^2+5\\
	w_7&=21c_3+7c_3c_4+c_7+7c_5\\
	w_8&=8c_3c_5+28c_3^2+c_8+8c_6+28c_4+4c_4^2+14\\
	w_9&=9c_3c_6+84c_3+12c_3^3+9c_7+36c_5+9c_4c_5+c_9+72c_3c_4\\
	w_{10}&=42+5c_5^2+180c_3^2+45c_3^2c_4+10c_4c_6+90c_3c_5+10c_8+c_{10} +45c_6+10c_3c_7+120c_4+45c_4^2.
	\end{align*}
		We see that $T$ is invertible by using induction to show that each $c_j$ is completely determined by the vector $(w_3, \dots w_j)$. Note $c_3=w_3$ is completely determined. Assume $c_3, \dots c_j$ have been completely determined. Note that $w_{j+1}=c_{j+1}+ f((c_3, c_4, \dots c_{j}))$ for some function $f$. Since $w_{j+1}$ is given and $f$ is a function of values that are already determined, there is only one choice for $c_{j+1}$.
\end{remark}

In \Cref{basics} we derive the coefficient of $\prod_{i=1}^j (c_{a_i})^{t_i}$ in $T$ by counting the number of walk structures that can be decomposed into $t_1, t_2, \dots t_j$ cycles of lengths $a_1, a_2, \dots a_j$ respectively. In \Cref{locally reg section}, we define class of {\em locally regular} graphs in which each vertex is in the same number of cycles, and then show this class of graphs exhibits the relationship between cycles and closed walks given in \Cref{relationship}.

In \Cref{ROC relation}, we prove \Cref{polynomials}, which describes the limit achieved by a ROC family $\mathcal{D}=(\mu, a)$. The parameter $a$ plays an important role. When $a<1/2$, the closed walks that trace trees dominate the closed walk count, so the limit is the Catalan sequence.  When $a>1/2$ the closed walks that trace simple cycles dominate the closed walk count, so the limit is the normalized number of expected simple cycles.  However, when $a=1/2$, cycles, trees, and other walk structures are all of the same order, and so the relationship given in \Cref{relationship} appears in the limit.

\begin{theorem} \label{polynomials}
	Let $(\dist,a)$ be a ROC family. Let $B$ be the set of all $i$ such that $\beta_i=1$, let $B^c$ be the set of all $i$ such that $\beta_i=0$, and let $x= 1/ \brac{\sum_{i\in B^c} \mu_i m_i^2 q_i+2\sum_{i\in B} \mu_i m_i^2 q_i}$. Define $$ c(k)=\begin{cases} 
	1 & k=2\\
	x \sum_{i \in B^c} \mu_i (m_i q_i)^{k} & k \text{ odd and } k\geq 3\\
	x \sum_{i \in B^c} \mu_i (m_i q_i)^{k} + 2x \sum_{i \in B} \mu_i (m_i q_i)^{k}& k \text{ even and } k\geq 4. \end{cases}$$
	Let $\sf{Cat_n}$ $=\frac{1}{n+1} { 2n \choose n}$ denote the $n^{th}$ Catalan number, and let $T$ be as given in $\Cref{relationship}$.
	\begin{enumerate}
		\item If $a<1/2$,  the ROC family fully achieves the limit $(0, $$\sf{Cat_2}$ $, 0, $ $\sf{Cat_3}$ $, \dots)$ with sparsity exponent $1/2$.
		\item If $a=1/2$, the ROC family fully achieves the limit $T((c(3), c(4), c(5), \dots ))$ with sparsity exponent $1/2$.
		\item If $a>1/2$,  the ROC family fully achieves the limit $(c(3), c(4), c(5), \dots )$ with sparsity exponent $a$.
	\end{enumerate}
	The ROC family achieves the corresponding length $k-2$ prefix as its $k$-limit with the same $k$-sparsity exponent for $k \geq 4$.
\end{theorem}

In \Cref{roc sec sec}, we show that the probability the normalized walk count $W_j(G, \alpha)$ of a ROC graph $G \sim ROC(n,d,\mathcal{D})$ deviates from $w_j$ in the limit  achieved by the family vanishes as $d$ grows (\Cref{deviate}). \Cref{roc seq} gives conditions that guarantee that a sequence of graphs drawn from a common ROC family converges almost surely to the limit achieved by the family.

\subsection{The cycle structure of closed walks } \label{basics}

In order to count the number of closed walks in a multi-graph, we divide the closed walks into classes based on the structure of the cycles appearing in the closed walk and then count the number of closed walks in each class. Each class is defined by a ``cycle permutation" 
 in which each non-zero character represents the first step of a cycle within the walk and each zero represents a step in a cycle that has already begun (\Cref{cycle permutation}). Here both (i) a closed walk from $u$ to $v$ and back along a single edge and (ii) a closed walk from $u$ to $v$ on an edge and back from $v$ to $u$ on a different multi-edge are refered to as 2-cycles.
In \Cref{perm bijection}, we show that the number of cycle permutations corresponding to a walk made of $t_1, t_2, \dots t_j$ cycles of lengths $a_1, a_2, \dots a_j$ respectively is the coefficient of $\prod_{i=1}^j \brac{c_{a_i}}^{t_i}$ in the cycle-walk transform.

\begin{figure}[h]
	\centering
	\includegraphics{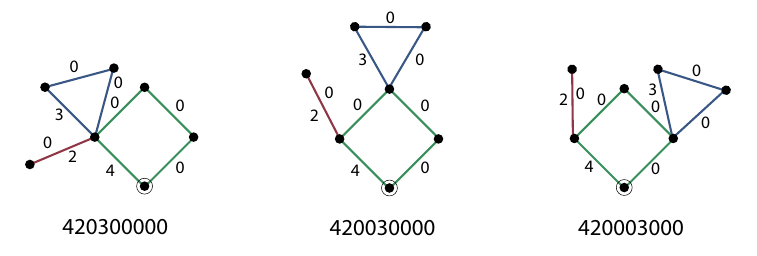}
	\caption{The above walks begin and end at the circled vertex and proceed left to right. Each is labeled with its cycle permutation. }
\end{figure}

\begin{definition}[cycle permutation] \label{cycle permutation}
	Follow the procedure below to label each step of a closed $k$-walk $\mathcal{W}=(r_1, r_2, \dots r_k)$ with a label and define the ``cycle permutation" $P$ of $\mathcal{W}$ as  the labels of the steps in order of traversal. 
	\begin{enumerate}
		\item Repeat until all steps are labeled:\smallskip\\
		Traverse $\mathcal{W}$ skipping a step $r_i$ if it has already been labeled. Let $u$ be the first repeated vertex on this traversal. The modified walk must have traversed a cycle $r_i, r_{i+1}, \dots r_{i+j-1}$ starting at $u$. Label the first step of the cycle (which traverses the edge $r_i$) with the length of the cycle. Label all other steps with zero. 
		\item Traverse $\mathcal{W}$ and let $P$ be the string of labels of the steps as they are traversed. 
	\end{enumerate}
\end{definition}

The following lemma enumerates the cycle permutations using bijection between cycle permutations and generalized Dyck paths (\Cref{dyck d}).

\begin{lemma}\label{perm bijection} Let $\mathcal{S}_k$ be as given in \Cref{relationship}.
	For each $S \in \mathcal{S}_k$, let $M_s$ be the multiset where $a_i$ appears $t_i$ times and there are $k-\sum_i t_i$ zeros. Let $P_S$ be the set of all permutations of $M_s$ such that the following property holds for all $2 \leq i \leq k+1$:
	$$\sum_{\ell \in N(s)} (\ell-1 )\geq z_i$$ where $N(s)$ is the multiset of non-zero labels that appear before the $i^{th}$ label of the permutation and $z_i$ is the number of times zero occurs before the $i^{th}$ label of the permutation.
	The set of cycle permutations is $\bigcup_{s\in S} P_S$ and $$|P_S|=  \frac{k!}{(\prod t_i!)(k+1- \sum t_i)!}.$$ 
\end{lemma}

To compute the size of $P_S$ in the above lemma we use a bijection between permutations in $P_S$ and generalized sub-diagonal Dyck paths, whose cardinality is given in \Cref{dyck}. 

\begin{definition}[generalized Dyck path, see \cite{ruk11}] \label{dyck d}
	A generalized Dyck path $p$ is a sequence of $n$ vertical steps of height one and $k \leq n$ horizontal steps with positive integer lengths $\ell_1, \ell_2, \dots \ell_k$ satisfying $\sum_{i=1}^k \ell_i =n$ on a $n\times n$ grid such that no vertical step is above the diagonal. 
\end{definition}

\begin{lemma}[from \cite{ruk11}]  \label{dyck} 
	Let $D$ the set of generalized Dyck paths on a $(k-\sum {t_i}) \times (k-\sum {t_i})$ grid that are made up of $t_i$ horizontal steps of length $a_i-1$ and $k-\sum t_i$ vertical steps of length $1$. Then $$|D|=\frac{k!}{(\prod t_i!)(k+1- \sum t_i)!}.$$
\end{lemma}

\begin{proof}(of \Cref{perm bijection})
	First we show the set of cycle permutations is $\bigcup_{s\in S} P_S$. Let $P$ be a cycle permutation of some closed walk $\mathcal{W}$ of length $k$. Let $s=\{(a_1, t_1), (a_2,t_2), \dots (a_j,  t_j)\}\in S$ where $\{a_1, a_2, \dots a_j\}$ are the non-zero labels of $P$ and each $a_i$ appears $t_i$ times in $P$ (so $\sum a_i t_i=k)$. To see that $P$ is in $P_S$ we show that  for all $2 \leq i \leq k+1$:
	$$\sum_{\ell \in N(s)} (\ell-1 )\geq z_i$$ where $N(s)$ is the multiset of non-zero labels that appear before the $i^{th}$ label of P and $z_i$ is the number of times zero occurs before the $i^{th}$ label of P. Before the $i^{th}$ step of the walk suppose non-zero labels $\ell_1, \dots \ell_k$ have been traversed. The only steps labeled with a zero that have been traversed must be part of a cycle corresponding to one of the labels $\ell_1, \dots \ell_k$. Since $\ell_i$ labels a cycle of length $\ell_i$ at most $\sum (\ell_i -1)$ zero steps have been traversed. Thus $P \in P_S$.
	
	Next we claim that any $P \in \bigcup P_S$ corresponds to a closed walk $\mathcal{W}$.  Let $T(s)= \sum t_i$ be the number of non-zero values in each permutation in $P_S$. We show that for any $k =\sum a_i t_i$ all permutations in $P_S$ correspond to closed walks by induction on $T(s)$. Note for any $k$ and $T(s)=t_1=1$, there is one permutation in $P_S$, $k=a_1$ followed by $k-1$ zeros. This permutation corresponds to a $k$-cycle. Assume that if $T(s')<T(s)$ then each string in $P_{s'}$ corresponds to a closed walk. We show each $P \in P_S$ corresponds to a closed walk. Consider the last non-zero value of the permutation $P$. Without loss of generality, suppose this value is $a_j$ and that it occurs at the $i^{th}$ coordinate of $P$. Since $P \in P_S$, there must be at least $a_j-1$ zeros to the right of $a_j$. Removing $a_j$ and $a_j-1$ zeros to its right in $P$ yields a valid sequence $P' \in P_{s'}$  where $s'=(a_1, \dots a_j, t_1, \dots t_{j-1}, t_j-1)$ and $k' = k- a_j$. By the inductive hypothesis, $P'$ corresponds to a closed walk $\mathcal{W}'$ of length $\sum a_i t_i- a_j$. Add a cycle of length $a_j$ in between the $(i-1)^{st}$ and $i^{th}$ steps of $\mathcal{W}'$ to obtain a closed walk $\mathcal{W}$ of length $k=\sum a_i t_i$. We have shown that $\bigcup P_S$ is the set of all cycle permutations for closed walks. 
	
	We compute the size of $|P_S|$ by constructing a bijection between permutations in $P_S$ and a set of subdiagonal generalized Dyck paths. Let $s=\{(a_1, t_1), (a_2,t_2), \dots (a_j,  t_j)\}  \in S$.  Let $D$ the set of subdiagonal generalized Dyck paths on a $(k-\sum {t_i}) \times (k-\sum {t_i})$ grid that are made up of $t_i$ horizontal steps of length $a_i-1$ and $k-\sum t_i$ vertical steps of length $1$. Each non-zero label $a_i$ of $P_S$ corresponds to a horizontal step of length $a_i-1$ and each zero label of $P_S$ corresponds to a vertical step of length one. Consider the map between permutations and generalized Dyck paths based on this correspondence. The condition that for all $2 \leq i \leq k+1$ $\sum_{\ell \in N(s)} (\ell-1 )\geq z_i$  translates to the generalized Dyck path not crossing the diagonal. Thus, the correspondence is a bijection between $P_S$ and $D$. \Cref{dyck} implies that $|P_S|=|D|=\frac{k!}{(\prod t_i!)(k+1- \sum t_i)!}$.
\end{proof}

\subsection{Walk and cycle counts in locally regular graphs} \label{locally reg section}
We show that the polynomial relating cycles and closed walks given in \Cref{relationship} governs the relationship between cycles and closed walks in graphs where each vertex is in approximately the same number of cycles. 

\begin{definition}[$k$-locally regular, essentially $k$-locally regular] Let $C_k(G, v)$ denote the the number of $k$-cycles at vertex $v$ in $G$. \begin{enumerate}
		\item 	A multi-graph $G$ is $k$-locally regular if it is regular and $C_j(G, v)= C_j(G, u)$ for all $u, v \in V(G)$ and $j\leq k$.
		\item A sequence of multi-graphs $(G_i)$ with $d_i \to \infty$ and $k$-sparsity exponent $a$ is essentially $k$-locally regular if $C_j(G_i,v)- C_j(G_i,u)= o\brac{d_i^{j/2}}$ for all $u, v \in V(G)$ and $j\leq k$.
	\end{enumerate}
\end{definition}

\begin{theorem}\label{local reg thm}
	Let $G$ be a $k$-locally regular multi-graph on $n$ vertices with degree $d$. Let $c_k=C_k(G)/(nd^{k/2})$ and $w_k=W_k(G)/(nd^{k/2})=W_k(G,1/2)$ where $W_k(G)$ and $C_k(G)$ denote the number of closed $k$-walks and simple $k$-cycles in $G$ respectively. Then 
	$$w_k= T( (c_3, c_4, \dots c_k )  ).$$
\end{theorem}

\begin{proof}
	We count the number of closed walks in $G$ at a vertex $v$ by partitioning the closed walks into sets based on their cycle permutation and computing the size of each partition class. Let $S = \{(a_1,t_1),  (a_2,t_2), \dots ,(a_j, t_j)\}\in \mathcal{S}_k$ and $P \in P_S$ as defined in \Cref{perm bijection}. Define $X_P$ as the number of walks with cycle permutation $P$ at $v$ in $G$. Let $t= \sum t_i$ be the number of non-zero values in $P$ and let $N(P,i)$ denote the $i^{th}$ non-zero value of the string $P$. Let $C_k(G,u)$ denote the number of $k$-cycles at $u$. Since  $G$ is locally regular $C_k(G,u)= c_k d^{k/2}$ for all $k$. It follows
	
	\begin{equation}\label{xp} 
	X_P = \prod_{\ell=1}^t C_{N(P,\ell)}(G,u) = \prod_{i=1}^j  \brac{c_{a_i} d^{a_i/2}}^{t_i}= d^{k/2}\prod_{i=1}^j  \brac{c_{a_i}}^{t_i}.
	\end{equation}
	Summing over all $P \in P_S$ and all vertices we obtain
	$$W_k(G)=n d^{k/2}\sum_{S \in \mathcal{S}} |P_S| \prod_{i=1}^j  \brac{c_{a_i}}^{t_i}, \quad \text{ equivalently } \quad w_k= \sum_{S \in \mathcal{S}} |P_S| \prod_{i=1}^j  \brac{c_{a_i}}^{t_i}.$$
	The statement follows directly from \Cref{perm bijection}. 
\end{proof}

\begin{theorem}\label{locally reg}
	Let $(G_r)$ be a sequence of essentially $k$-locally regular multi-graphs with $n_r$ vertices and degree $d_r \to \infty$, sparsity exponent $1/2$, and $k$-limit $(w_3, w_4, \dots w_k)$. Let $C_j(G_r)$ be the number of $j$-cycles in $G_r$.  
	Then $(w_3, w_4, \dots w_k )=T((c_3, c_4, \dots c_k ) )$ where $$c_j= \lim_{r \to \infty }\frac{C_j(G_r)}{n_r d_r^{j/2}}.$$
\end{theorem}

\begin{proof} 
	We follow the proof of \Cref{local reg thm} until line \Cref{xp}. Since the $G_r$ is approximately locally regular rather than locally regular, we have the weaker guarantee that $C_k(G_r, u)= \frac{C_k(G_r)}{n_i}+ o\brac{d_i^{j/2}}.$ It follows that for $G=G_r$, $n=n_r$, and $d=d_r$,
	
		\begin{equation*} 
		X_P = \prod_{\ell=1}^t C_{N(P,\ell)}(G,u) = \prod_{i=1}^j  \brac{\frac{C_{a_i}(G)}{n }+ o\brac{d^{a_i/2}}}^{t_i}= d^{k/2}\prod_{i=1}^j  \brac{\frac{C_{a_i}(G)}{n d^{a_i/2}} }^{t_i} +o\brac{d^{k/2}}.
		\end{equation*}
		Summing over all $P \in P_S$ and all vertices we obtain
		$$W_k(G_r)=  n_r d_r^{k/2}\sum_{S \in \mathcal{S}} |P_S| \prod_{i=1}^j \brac{\frac{C_{a_i}(G_r)}{n_r d_r^{a_i/2}} }^{t_i} +o\brac{n_rd_r^{k/2}}.$$
		Therefore
			$$w_k=\lim_{r \to \infty} \frac{W_k(G_r)}{n_r d_r^{k/2}}= \lim_{r \to \infty} \sum_{S \in \mathcal{S}} |P_S| \prod_{i=1}^j \brac{\frac{C_{a_i}(G_r)}{n_r d_r^{a_i/2}} }^{t_i} +o_r\brac{1}= \sum_{S \in \mathcal{S}} |P_S| \prod_{i=1}^j \brac{c_{a_i}}^{t_i},$$
		and the statement follows directly from \Cref{perm bijection}. 
\end{proof}

\subsection{Limits achieved by ROC families} \label{ROC relation}
We prove \Cref{polynomials}, which describes the limits of ROC families. The following lemma gives the expected number of closed walks by permutation type.

\begin{lemma}\label{exp by class}
	Let $S  = \{(a_1,t_1),  (a_2,t_2), \dots ,(a_j, t_j)\}\in \mathcal{S}_k$ as defined in \Cref{perm bijection}, and  let $t= \sum_{i=1}^j t_i$. Let $X_S(G)$ be the random variable for the number of walks with a permutation type in $P_S$ in $G \sim ROC(n, d,\mathcal{D})$ where $d=o \brac{n^{1/((1-a)k+2a-1)}}$. Then for the function $c$ as given in \Cref{polynomials}
	\begin{enumerate}
		\item For $a<1$, $\E {X_S(G)}=    |P_S| \brac{\prod_{i=1}^j c(a_i)^{t_i}} nd^{(1-2a)t+ak }+o\brac{nd^{(1-2a)t+ak}}.$
		\item For $a=1$ and $t=1$, $\E {X_S(G)}=   |P_S|\brac{\prod_{i=1}^j c(a_i)^{t_i}} nd^{k-1} +o\brac{nd^{k-1}}.$
		\item For $a=1$ and $t>1$, $\E {X_S(G)}=  \Theta \brac{ nd^{k-t}}.$ 
	\end{enumerate} 
\end{lemma}

Taking $S= \{ (k,1)\}$ in the above lemma gives the number of simple $k$-cycles in a ROC graph. The following corollary describes the cycle counts when the community size is a constant independent of $d$.  The following corollary follows directly from Case 1 of Lemma 17.
\begin{corollary} \label{exp by class cor}
	Let $G \sim ROC(n,d,\mu,0)$ Then for $d= o \brac{ n^\frac{1}{k-1}}$,
	$$\E { C_k(G)}= c(k) nd+ o(nd).$$ 
\end{corollary}

\begin{proof}(of \Cref{exp by class}) Let $P \in P_S$, and let $X_P(G)$ be the random variable for the number of walks in $G$ with permutation type $P$. We show that $\E{X_P(G)}$ is the same for each $P \in P_S$ and so
	\begin{equation} \label{a and b} \E {X_S(G)}= \sum_{P \in P_S} \E {X_P(G)}= |P_S| \E{X_P(G)}. \end{equation}

	To compute the expectation of $X_P(G)$, we apply linearity of expectation to indicator random variables representing each possible walk. We define a possible walk as  (i) an ordered set of vertices $(v_1, \dots v_k)$ such that the walk $v_1, v_2, \dots v_k$ is closed and has cycle permutation $P$ and (ii)  an ordered set of communities $(u_1, \dots u_k)$, $u_i \in [xnd^{1-2a}]$. (Recall that each a ROC graph is the union of Erd\H{o}s R\'enyi graphs. Each such Erd\H{o}s R\'enyi graph is called a community.) The walk exists if for each $1 \leq i \leq k-1$, the vertices $v_i$ and $v_{i+1} $ are adjacent by an edge that was added in the $(u_i)^{th}$ community in the construction of $G$. The probability a possible walk exists in $G$ depends on how often the community labels $(u_1, \dots u_k)$ change between adjacent vertices. 
	
	Let $A$ be the set of possible walks in which each cycle is assigned a distinct community, each edge is labeled with the community assigned to its cycle, and there are $k-t+1$ distinct vertices.
	We write $X_P(G)= A_P(G)+ B_P(G)$ where $A_P(G)$ is the random variable for the number of walks in $A$ that appear in $G$ and $B_P(G)$ is the random variable for the number of walks that appear in $G$ and are not in $A$.  We compute $\E {A_P(G)}$ and show that $\E { B_P(G)}= o \brac{\E { A_P(G)}}$ in cases (1) and (2) and $\E { B_P(G)}= O \brac{\E { A_P(G)}}$ in case (3).
	
	\underline{Claim 1:} $\E {A_P(G)}=\brac{\prod_{i=1}^j c(a_i)^{t_i}} nd^{(1-2a)t+ak }+o\brac{nd^{(1-2a)t+ak}}$. 
	
	We write $A_P(G)$ as the sum of random variables $A_{W}(G)$ that indicate if a walk $W\in A$ is in $G$. We show that $\E{ A_W(G)}$ is the same for all $W$ in $A$ and so 
	$$\E{A_P(G)}= |A|  \E{ A_W(G)}= |A|  \Pr{ A_W(G)}.$$
	
	 We now compute $\Pr{A_W(G)}$. Let $z_1, \dots z_t$ be the non-zero characters of $P$ ordered by first appearance. Let $A_\ell$ be the event that all edges in the cycle corresponding $z_\ell$ were added in the community assigned to $z_\ell$, which we denote $y_\ell$. The probability of $A_\ell$ depends on the community type $(m_i, q_i, \beta_i)$ of $y_\ell$. We compute $$\Pr{ A_\ell}= \sum_{i} \Pr{ \text{ specified cycle appears in community $y_{\ell}$}  \given \text{ $y_{\ell}$ is type $i$} } \Pr{\text{ $y_{\ell}$ is type $i$} }.$$
	It follows that 
	$$
	\Pr{ A_\ell}=
	\begin{cases}
\psi_\ell	\sum_{i \in B} 2\bfrac{m_id^a}{n}^{z_{\ell}} q_i^{z_{\ell}} \mu_i + \sum_{i \in B^c} 
	\bfrac{m_id^a}{n}^{z_{\ell}} q_i^{z_{\ell}} \mu_i & z_{\ell}\geq 3\\
	& \\
	\sum_{i \in B} 2\bfrac{m_id^a}{n}^{2} q_i \mu_i + \sum_{i \in B^c} \bfrac{m_id^a}{n}^{2} q_i \mu_i & z_{\ell}=2.
	\end{cases} 
	$$
	where $\psi_\ell$ is zero when $z_\ell$ is odd and one when $z_\ell$ is even. 
	Equivalently, $\Pr{A_\ell}=d^{az_\ell} c(z_\ell)/ (xn^{z_\ell})$. 
The event that the walk $W$ appears in $G$ is the intersection of the events $A_1, \dots, A_t$. Note that these events are independent because the communities $y_1, y_2, \dots y_t$ are distinct.
	It follows $$\Pr{A_W(G)}= \prod_{\ell=1}^t \Pr{A_\ell}= \frac{ d^{ak} \prod_{i=1}^j c(a_i)^{t_i}}{ n^k x^t}.$$ 

Next we compute the size of $A$.  There are $\frac{(xnd^{1-2a})!}{(xnd^{1-2a}-t) !}= (xnd^{1-2a})^{t}+o\brac{ (xnd^{1-2a})^{t}}$ ways to select $t$ distinct communities and $\frac{n!}{n-(k-t+1)!}=n^{k-t+1}+o \brac{n^{k-t+1}}$ ways to select the vertices for $W \in A$. The claim follows,
\begin{align}
\E{A_P(G)} &=\brac{(xnd^{1-2a})^{t}+o\brac{ (xnd^{1-2a})^{t}}} \brac{n^{k-t+1}+o \brac{n^{k-t+1}} } \brac{ \frac{ d^{ak} \prod_{i=1}^j c(a_i)^{t_i}}{ n^k x^t}} \nonumber \\
&=n d^{(1-2a)t+ak} \brac{\prod_{i=1}^j c(a_i)^{t_i}} + o\brac{nd^{(1-2a)t+ak}}.\label{a walks}
\end{align}

\underline{Claim 2:} In cases (1) and (2), $\E {B_P(G)}= o\brac{nd^{(1-2a)t+ak}}$, and in case (3) $\E {B_P(G)}=$ $O \brac{nd^{(1-2a)t+ak}}$.

Before computing $\E{B_P(G)}$, we introduce notation to describe the features of possible walks that are not in $A$. Let $z_1, \dots, z_t$ be the non-zero characters of $P$, so the walk is composed of cycles of lengths $z_1, \dots z_t$. Let $m_\ell$ be the number of vertices in the cycle corresponding $z_\ell$ that do not appear in the cycles corresponding to $z_1, \dots z_{\ell-1}$. We label each edge with the community it belongs to and call these labels ``community edge labels." Let $\lambda_i$ be the number communities that appear as community edge labels in the cycle corresponding to $z_{\ell}$ and do not appear as community edge labels in any cycle corresponding to $z_1, \dots z_{\ell-1}$. If the $i^{th}$ edge $(v_i,v_{i+1})$ is given the community edge label $u_i$, $v_i$ and $v_{i+1}$ must have both been members of community $u_i$ if the possible walk exists in $G$. We say the $i^{th}$ edge ``assigns" the community $u_i$ to the vertices $v_i$ and $v_{i+1}$. Each vertex receives one such ``vertex-community assignment" per adjacent edge. Vertices may receive the same vertex-community assignment multiple times, i.e. if two consecutive edges have the same community edge label, then the common end is assigned to the same community twice. Let $\Gamma_\ell$ be the number of unique vertex-community assignments from the cycle corresponding to $z_\ell$ that are not repeats of vertex-community assignments given by edges  in cycles corresponding to $z_1, \dots z_{\ell-1}$.  Let $m= \sum_i m_i\leq k-t+1$, $\lambda=\sum_i \lambda \leq k$, $\Gamma= \sum \Gamma_i$, and $j$ be the number of indices $i$ such that $\lambda_i \geq 2$. Let $\mathcal{P}$ denote the parameters $\{\lambda_i, m_i, \Gamma_i\}$, and let   $W_{\mathcal{P}}(G)$ be the number of possible walks with the  parameters $\mathcal{P}$. There are $\Theta \brac{ \brac{nd^{1-2a}}^\lambda}$ ways to select the communities, $\Theta \brac{ n^m}$ ways to select the vertices. The probability a vertex is in an assigned community is $\Theta \bfrac{d^a}{n}$. It follows that 
\begin{equation}\label{lower order}
\E{ W_{\mathcal{P}}(G)}= \Theta \brac{   \brac{nd^{1-2a}}^\lambda  n^m  \bfrac{d^a}{n}^\Gamma  }.
\end{equation}

Next we show that for any set of parameters $\mathcal{P}$, $\Gamma \geq m+ \lambda +j-1$. First we describe relationships between $\lambda_\ell$, $\Gamma_\ell$, $z_\ell$, and $m_\ell$ in different settings.
\begin{enumerate}
\item If there are  precisely $\lambda_\ell$ communities labeling the edges in the cycle corresponding to $z_{\ell}$ then 
\begin{align*}
 \Gamma_\ell &\geq z_\ell+ \lambda_\ell  \quad \quad \quad &\lambda_\ell \geq 2\\
  \Gamma_\ell &= z_\ell  \quad \quad \quad & \lambda_\ell =1.
  \end{align*}
  If there are $\lambda_\ell \geq 2$ distinct communities labeling edges in the cycle, then there are at least $\lambda_\ell$ vertices where the adjacent edges are assigned different communities. These vertices contribute $2\lambda_\ell$ vertex-community assignments and the remaining $z_\ell- \lambda_\ell$ vertices are also assigned a community. If $\lambda_\ell=1$, then each vertex is assigned to the one community.
  
  \item If there are more than $\lambda_\ell$ distinct communities labeling the edges in the cycle corresponding to $z_{\ell}$ then 
\begin{align*}
 \Gamma_\ell &\geq m_\ell+ \lambda_\ell  +1 \quad \quad \quad &\lambda_\ell \geq 1 \\
  \Gamma_\ell &\geq m_\ell \quad \quad \quad &\lambda_\ell =0.
    \end{align*}
   The $m_i$ new vertices must be assigned at least one community. When $\lambda_\ell \not=0$, there must be at least $\lambda_\ell +1$ vertices in which (i) both adjacent edges are labeled with two different first appearing communities or (ii) one adjacent edge is labeled with a first appearing community and one adjacent edge is labeled with a community that has already appeared. If such a vertex is a new vertex then this vertex has a total of two community assignments. If such a vertex has appeared before, it has not been previously assigned to a first appearing community, so this contributes one community assignment.
   \end{enumerate}
Since the first vertex of a cycle corresponding to $z_\ell$ for $\ell \geq 2$ has already been visited, $z_\ell \geq m_\ell+1$ for $\ell \geq 2$. Therefore when $\ell \geq 2$
\begin{align*}
 \Gamma_\ell &\geq m_\ell+ \lambda_\ell +1   &\lambda_\ell \geq 2\\
  \Gamma_\ell &= m_\ell +\lambda_\ell  & \lambda_\ell \leq 1,
  \end{align*}
 and for $\ell=1$, 
 \begin{align*}
 \Gamma_1 &\geq m_1+ \lambda_1  &\lambda_1 \geq 2\\
  \Gamma_1 &= m_1   & \lambda_1=1.
  \end{align*}
  Summing the above inequalities over $\ell$ yields the observation that $\Gamma \geq m+ \lambda +j-1$. Note also that $m \leq k-t+1$. Equation \Cref{lower order} becomes
\begin{equation}\label{lower order simplified}
\E{ W_{\mathcal{P}}(G)}= \Theta \brac{   \brac{nd^{1-2a}}^\lambda  n^m  \bfrac{d^a}{n}^{m+ \lambda +j-1} }=O \brac{n^{1-j} d^{(1-a)\lambda + a(k-t+j)}}.
\end{equation}

Next consider a walk that is not in $A$. There must either be (i) a cycle that has at least two new community labels and so $j \geq 1$ or (ii) fewer than $t$ total community labels, so $\lambda \leq t-1$. If $t=1$ (the walk is a simple cycle), case (ii) does not occur because there must be at least one community label. 
In case (i), $j \geq 1$, $\lambda \leq k-t+j$ and $t\geq 1$, and so Equation \Cref{lower order simplified} becomes 
\begin{equation} \label{type i}
\E{ W_{\mathcal{P}}(G)}=O \brac{n^{1-j} d^{k-t+j}}
=
\begin{cases}
O\brac{d^{k-t+1}}& j =1\\
o\brac{d^{k-t}} & j \geq 2.
\end{cases}
\end{equation}
Therefore
$$\E{ W_{\mathcal{P}}(G)}=nd^{ (1-2a) t+ak} O \brac{n^{-1}d^{(1-a)k+2a-1} }= o \brac{nd^{ (1-2a) t+ak} }.$$
 In case (ii),  $\lambda \leq t-1$ and $j \geq 0$, and so Equation \Cref{lower order simplified} becomes 
\begin{equation}\label{type ii}
\E{ W_{\mathcal{P}}(G)}= O \brac{n^{1-j} d^{(1-a)\lambda + a(k-t+j)}}
= nd^{ (1-2a) t+ak} O \brac{d^{a-1} }.
\end{equation}
Therefore for any $ \mathcal{P}$ that is not in $A$,
$$\E{ W_{\mathcal{P}}(G)}=
\begin{cases}
 o \brac{nd^{ (1-2a) t+ak} } & a<1 \text{ or } a=1 \text{ and } t=1\\
 O \brac{nd^{ (1-2a) t+ak} }& a=1 \text{ and } t>1.
\end{cases}
$$

Note $\E {B_P(G)}= \sum_{\mathcal{P}} \E{ W_{\mathcal{P}}(G)}$. Since the number of sets of valid parameters $\mathcal{P}$ is constant,  Claim 2 follows.

The computation of $\E{ X_P(G)}= \E { A_P(G)}+ \E { B_P(G)}$ did not rely on any information about $P$ besides that $P \in P_S$. Therefore, Equation $\Cref{a and b}$ holds and the statement of the lemma follows directly from Claims 1 and 2. 
\end{proof}

\begin{proof} (of \Cref{polynomials})
For $a <1$,  \Cref{exp by class} implies 
\begin{equation}\label{full}
 \E{W_k(G)}=
 \sum_{S \in \mathcal{S}} |P_s|  \brac{\prod_{i=1}^j c(a_i)^{t_i}} nd^{(1-2a)t+ak }+o\brac{nd^{(1-2a)t+ak}} 
 \end{equation}
We now collect the highest order terms  of \Cref{full} for different values of $a$. Recall $\sum a_i t_i =k$ and $a_i \geq 2$,  so $1 \leq \sum t_i \leq \frac{k}{2}$.

\textbf{Case 1: $a \in (0,1/2)$.}
The highest order term of \Cref{full} is from $S \in \mathcal{S}$ with $a_1=2$ and $t_1= \frac{k}{2}$ for $k$ even and $S \in \mathcal{S}$ with $a_1=3$, $a_2=2$, $t_1=1$, $t_2= \frac{k-3}{2}$ for $k$ odd. For even $k$, Equation \Cref{full} becomes
$$ \E{W_k(G)}=n d^{k/2} \frac{k!}{\brac{\frac{k}{2}}! \brac{\frac{k}{2}+1}!}c(2)^{k/2} +  o\brac{nd^{k/2}}=(Cat_{k/2}) nd^{k/2} +o(nd^{k/2}).$$
For odd $k$, Equation \Cref{full} becomes $$ \E{W_k(G)}= O\brac{nd^{ak +(1-2a) \frac{k-1}{2}} }= O\brac{nd^{\frac{k-1}{2}+a}}=o(nd^{k/2}).$$
It follows that the ROC family $\mathcal{D}=(\mu,a)$ achieves the Catalan vector with sparsity exponent $1/2$, and achieves the $k-2$ length prefix with $k$-sparsity exponent $1/2$ for all $k \geq 4$.

\textbf{Case 2: $a=1/2$.}
Each $S \in \mathcal{S}$ contributes a term of order $nd^{k/2}$ to Equation \Cref{full}. Therefore
$$ \E{W_k(G)}= \brac{ \sum_{S \in \mathcal{S}_k}  \frac{k!}{(\prod t_i!)(k+1- \sum t_i)!}\prod_i  c(a_i)^{t_i} }nd^{k/2}+  o\brac{nd^{k/2}}= w(k) nd^{k/2} + o \brac{nd^{k/2}}.$$
It follows that the ROC family $\mathcal{D}=(\mu,a)$ achieves the limit $(w_3, w_4, \dots )$ with sparsity exponent $1/2$, and achieves the $k-2$ length prefix with $k$-sparsity exponent $1/2$ for all $k \geq 3$.

\textbf{Case 3: $a \in (1/2,1)$.} The highest order term of \Cref{full} is from $S \in \mathcal{S}$ with $a_1=k$ and $t_1=1$. Therefore Equation \Cref{full} becomes $$ \E{W_k(G)}=  c(k) nd^{1+a(k-2)} +o\brac{nd^{1+a(k-2)}}.$$
It follows that the ROC family $\mathcal{D}=(\mu,a)$ achieves the limit $(c_3, c_4, \dots )$ with sparsity exponent $a$, and achieves the $k-2$ length prefix with $k$-sparsity exponent $a$ for all $k \geq 3$.

\textbf{Case 4: $a=1$.}
For $S \in \mathcal{S}$ with $t=\sum t_i$, the number of walks with permutation type in the set $P_S$ is $\Theta \brac{ nd^{(1-2a)t +ak}}$. Therefore the walks contributing the highest order terms correspond to $S \in \mathcal{S}$ with $a_1=k$ and $t_1=1$.  By parts 2 and 3 of \Cref{exp by class}, we have $$ \E{W_k(G)}=  c(k) nd^{k-1} +o\brac{nd^{k-1}}.$$
It follows that the ROC family $\mathcal{D}=(\mu,a)$ achieves the limit $(c_3, c_4, \dots )$ with sparsity exponent $1$, and achieves the $k-2$ length prefix with $k$-sparsity exponent $1$ for all $k \geq 3$.
\end{proof}

\section{ROC achievable limits}\label{achievable vectors}
In this section we give conditions for when a limit is achievable (\Cref{achieve conditions}) and show the limits of the hypercube and rook sequences are achievable (\Cref{sec:hc and rook}).

\subsection{Conditions for ROC achievable limits} \label{achieve conditions}
In this section we address the questions: for which vectors $L$ does there exist  a ROC family that achieves limit or $k$-limit $L$ with sparsity exponent $\alpha$? We first show that all $4$-limits are achievable in \Cref{4 lim}, then  describe necessary and sufficient conditions for a limit vector (of any length) to be achievable in \Cref{achieve con}. Finally \Cref{char} in \Cref{sc shortcut} gives a convenient criterion for determining when the Stieltjes condition is satisfied.

\subsubsection{Achievability of $(w_3,w_4)$ and triangle-to-edge and four-cycle-to-edge ratios.} \label{4 lim}

First we show that triangle-to-edge and four-cycle-to-edge ratios can be achieved simultaneously.
\begin{proof} (of \Cref{real network bounds}.) 
	(1) For clarity of this proof we refer to the number triangle and four-cycle structures (not counted as walks). Under this convention, the number of triangles is $T_3= C_3(H)/6$ and the number of four-cycles is $T_4=C_4(H)/8$. Note $T_3= |E(H)|c_3/6$ and $T_4=c_4 |E(H)|/8$.
	For each edge in $H$, let $t_e$ be the number of triangles containing $e$, so $\sum_{e \in E(H)} t_e = 3 T_3=  c_3 |E(H)|/2$. If triangles $abc$ and $abd$ are present, then so is the four-cycle $acbd$. This four-cycle may also be counted via triangles $cad$ and $cdb$. Therefore $T_4 \geq \frac{1}{2} \sum_{ e\in E(H)} {t_e \choose 2}=\frac{1}{2} \sum_{ e\in E(H)} t_e(t_e-1)/2$ . This expression is minimized when all $t_e$ are equal. We therefore obtain $$\frac{c_4|E(H)|}{8}=T_4\geq \frac{|E(H)|}{2} {  c_3/2 \choose 2}  = \frac{c_3(c_3/2-1) |E(H)|}{8}.$$ It follows that $\frac{c_3(c_3/2-1)}{c_4}\leq1$.  
	
	(2) Since the hypothesis guarantees $q \leq 1$, applying the special case of \Cref{exp by class cor} described in \Cref{counts} to $G \sim ROC\left(n,d, \frac{2c_4^2}{c_3^3}, \frac{c_3^2}{2c_4}\right)$ implies the desired statements.  
\end{proof}
 
Next we prove \Cref{three four}, which states that any $(w_3, w_4)$ that is a limit of a sequence of graphs with increasing degree can be achieved by a ROC family. In fact, the requirement that degree increases is only needed for the case in which the $4$-sparsity exponent is $1/2$.

\begin{lemma} \label{34 relationship}
Let $C_j(G)$ and $W_j(G)$ denote the number of simple $j$-cycles and closed $j$-walks of a graph $G$ respectively. For any graph $G$ on $n$ vertices with average $d$ 
$$W_4(G)  \geq \frac{W_3(G)^2}{nd} \quad \text{ and } \quad C_4(G)  \geq C_3(G) \brac{ \frac{C_3(G)}{nd} -1}.$$
\end{lemma}

\begin{proof}
	For each directed edge $e=(u,v)$ let $t_e$ be the number of walks that traverse a triangle with first edge $(u,v)$. For each edge $(u,v)$ we can construct $t_e^2$ four walks including ${t_e \choose 2}$ four cycles as follows. Select two triangles $(u,v,a)$ and $(u,v,b)$. The closed walk $(u,b, v,a)$ is a closed four walk. When $a \not=b$ the walk is a four cycle. Note $C_3(G)=W_3(G)$.  It follows that 
	$$W_4(G) \geq \sum_{e\in E(G)} t_e^2 \geq nd \bfrac{W_3(G)}{nd}^2 \quad \text{ and } \quad C_4(G) \geq \sum_{e\in E(G)} {t_e \choose 2} \geq nd \bfrac{C_3(G)}{nd} \brac{\frac{C_3(G)}{nd}-1}.$$
\end{proof}

 Holder's inequality implies that for any finite set of $c_i \in \R^{\geq 0}$,  
\[
\brac{\sum_i c_i^2} \brac{\sum_i c_i^4}  \geq \brac{\sum_i c_i^3}^2.
\]
The first part of the lemma also follows from this observation. 

Using these properties, we now prove \Cref{three four}.

\begin{proof} (of \Cref{three four}) Let $(G_i)$ be a graph sequence with increasing degree, and let $\alpha$ be the 4-sparsity exponent of the sequence. 

\textbf{Case 1: } $\alpha>1/2$. By \Cref{34 relationship} for each graph $G_i$ in the sequence satisfies
$$W_4(G_i,\alpha) \geq W_3(G_i,\alpha)^2.$$
It follows that $w_4\geq w_3$. If $w_3 \not=0$, the ROC family $(\mu, \alpha)$ where  $\mu$ is the distribution with support one on $m= w_4^2/w_3^3$ and $q=w_3^2/w_4 $ achieves the limit $(w_3, w_4)$. If $w_3=0$, the ROC family $(\mu, \alpha)$ where $\mu$ is the distribution with support one on $m=w_4$, $q=1$, and $\beta=1$ achieves the limit $(w_3, w_4)$. 

\textbf{Case 2: } $\alpha=1/2$. It suffices to show that the cycle counts $T((w_3, w_4))=(w_3,w_4-2)$ are the moments of some distribution. 
By \Cref{34 relationship}, $C_4(G) \geq C_3(G) \brac{ \frac{C_3(G)}{nd} -1}=W_3(G) \brac{ \frac{W_3(G)}{nd} -1}$. Let $T_4(G)$ be the number of closed four walks that trace a path of length two. The number of two paths is $\sum_v {deg(v) \choose 2} \geq n {d \choose 2}$, and each two path contributes four closed four walks.  Each edge contributes two closed four walks. It follows that $$W_4(G)= C_4(G) + T_4(G) + nd \geq W_3(G) \brac{ \frac{W_3(G)}{nd} -1} + 2 nd(d-1) + nd,$$
and so $$W_4(G_i, 1/2) \geq W_3(G_i, 1/2)^2 + 2-\Theta\brac{1/d}.$$
Since $d \to \infty$, the term $\Theta\brac{1/d} \to 0$. It follows that  $w_4 \geq w_3^2+2$. 
 If $w_3 \not=0$, the ROC family $(\mu, \alpha)$ where $\mu$ is the distribution with support one on $m=(w_4-2)^2/w_3^3$ and $q=w_3^2/(w_4-2)$ achieves the limit $(w_3, w_4)$. If $w_3=0$, the ROC family $(\mu, \alpha)$ where $\mu$ is the distribution with support one on $m=w_4-2$, $q=1$, and $\beta=1$ achieves the limit $(w_3, w_4)$. 
 \end{proof}

\subsubsection{Achievability of limits of general sequences}\label{achieve con}

In this section we prove \Cref{ach}, which characterizes achievable $k$-limits for sparsity exponent greater than half and half. Additionally, we prove the analogous characterization for full achievability, as stated in the following theorem. 

\begin{lemma}[Full Walk Count Achievability Lemma]\label{full-ach} \quad
	\begin{enumerate}
		\item \label{full more than half} \emph{(Sparsity exponent $>1/2$)} A limit vector $(w_3, w_4, \dots )$ is achievable by ROC with sparsity exponent greater than $1/2$ if and only if there exists $ \gamma \in [0,1]$, $s_0, s_1,  \dots,$ $ t_0, t_2, \dots \in \R^{+}$, $s_2, t_2 \leq 1$ such that $(s_0, s_1, s_2,\dots )$ and  $(t_0, t_2, \dots )$ satisfy the full Stieltjes condition and for all $j \geq 3$
		$$w_j = \begin{cases} \gamma s_j & j \text{ odd}\\ \gamma s_j +(1-\gamma) t_j & j \text{ even}. \end{cases} $$ 
		\item\label{full half} \emph{(Sparsity exponent $1/2$)}
		Let $T((c_3, c_4, \dots c_k))=(w_3, w_4, \dots w_k)$ be the transformation of a vector given in \Cref{relationship}. The limit vector $(w_3, w_4, \dots w_k)$ is achievable by ROC with sparsity exponent $1/2$ if and only if there exists $ \gamma \in [0,1]$, $s_0, s_1,  \dots,$ $ t_0, t_2, \dots \in \R^{+}$, $s_2, t_2 \leq 1$ such that $(s_0, s_1, s_2,\dots )$ and  $(t_0, t_2, \dots )$ satisfy the full Stieltjes condition and for all $j \geq 3$
		$$c_j = \begin{cases} \gamma s_j & j \text{ odd}\\ \gamma s_j +(1-\gamma) t_j & j \text{ even}. \end{cases} $$ 
	\end{enumerate}
	
\end{lemma}

The  question underlying achievability is how to determine when a vector is the vector of normalized cycle counts of some ROC family.  Note that the normalized cycle counts $(c(3), c(4), \dots c(k))$ of the family $ROC(n,d,\mathcal{D})$  are the moments of a discrete probability distribution over values determined by $m_i$, $q_i$ and $\beta_i$ scaled by $x$. The question of whether a vector can be realized as the vector of normalized cycle counts for some ROC family is a slight variant of the Stieltjes moment problem, which gives necessary and sufficient conditions for a sequence to be the moment sequence of some distribution with positive support. 

Our question differs in two key ways.
First, the second moment is not directly specified; instead we obtain an upper bound on the second moment from the restriction that $$x\brac{\sum_{i \in B^c} \mu_i m_i^2 p_i+ 2\sum_{i \in B} \mu_i m_i^2 p_i}=1.$$  Second, for achievability of $k$-limits we are interested in when a vector is the prefix of some moment sequence. 
 

The proof of the Full Walk Count Achievability Lemma (\Cref{full-ach}) relies on the classical solution to the Stieltjes moment problem (\Cref{stieltjes classic}), and the proof of the Walk Count Achievability Lemma (\Cref{ach}) uses a variant for truncated moment vectors (\Cref{stieltjes}). We use these lemmas to show \Cref{plain} and \Cref{bipartite}, which together with \Cref{polynomials}, directly imply the necessary and sufficient conditions given in \Cref{ach,full-ach}. 
Finally we prove \Cref{char} which gives a sufficient local condition to guarantee that a sequence can be extended to satisfy the Stieltjes condition. The proof of this lemma establishes the semi-definiteness of Hankel matrices of sequences satisfying a logconcavity condition.

The Stieltjes moment problem was first studied in \cite{stieltjes1894recherche}. We will apply variants as stated in \cite{sho43} and \cite{schmudgen2017moment}.

\begin{lemma}[Stieltjes moment problem, Theorem 1.3 of \cite{sho43}]\label{stieltjes classic}
A sequence $\mu=(\mu_0, \mu_1, \mu_2, \dots)$ is the moment sequence of a distribution with finite positive support of size $k$ if there exists $\{(x_i, t_i)\}_{1 \leq i \leq k}$ with $x_i, t_i >0$ such that $\sum_{i=1}^k x_i t_i^\ell= \mu_\ell$ for all non-negative integers $\ell$. A vector $\mu$ is a moment sequence with positive support  of size $k$ if and only if the full Stieltjes condition with parameter $k$ given in \Cref{s con} is satisfied.
\end{lemma}

\begin{lemma}[truncated Stieltjes moment problem, Theorems 9.35 and 9.36 of \cite{schmudgen2017moment}]\label{stieltjes}
	A vector $\mu=(\mu_0, \mu_1, \mu_2, \dots, \mu_n)$ is the truncated moment sequence of a distribution with finite positive support of size $k$ if there exists $\{(x_i, t_i)\}_{1 \leq i \leq k}$ with $x_i, t_i >0$ such that $\sum_{i=1}^k x_i t_i^\ell= \mu_\ell$ for all $0\leq \ell \leq n$. A vector $\mu$ is a truncated moment sequence with finite support  if and only if the truncated Stieltjes condition given in \Cref{s con} is satisfied.
\end{lemma}

In order to establish that the hypercube vectors satisfy the truncated Stieltjes condition, we will use the following simple implication. 
\begin{lemma}\label{cond-helper}
	If $det\brac{H_{2s}^{(0)}}>0$ for all $0 \leq 2s \leq n$ and  $det\brac{H_{2s+1}^{(1)} }>0$ for all $0 \leq 2s+1 \leq n$ hold with respect to the vector $\mu=(\mu_0, \mu_1, \dots \mu_n)$, then $\mu$ satisfies the truncated Stieltjes condition.
\end{lemma}

\begin{proof}
	In the case of $n=2j+1$ odd, let $H^{(0)}=H_{2j}^{(0)}$ and $H^{(1)}=H_{2j+1}^{(1)}$.  In the case $n=2j$ even, let
	$H^{(0)}=H_{2j}^{(0)}$ and $H^{(1)}=H_{2j-1}^{(1)}$. The
	hypotheses of the lemma and Slyvester's criterion imply that $H^{(0)}\succ 0$ and $H^{(1)} \succ 0$, and so $H^{(0)}$ and $H^{(1)}$ are full rank. 
\end{proof}

\begin{lemma} \label{plain} 
There exists $s_0, s_1,s_2$ with $s_2 \leq 1$ such that $(s_0, s_1, s_2, a_3, \dots a_n)$ satisfies the Stieltjes condition if and only if there exists $x_i, m_i, q_i$ with $x_i, m_i >0$ and $0\leq q_i \leq 1$ satisfying
	\begin{enumerate}
	\item $\sum x_i m_i^2 q_i= 1$
		\item $\sum x_i (m_i q_i)^j= a_j \quad \text{ for all }  3 \leq j \leq n.$
	\end{enumerate} 
Similarly, there exists $s_0, s_1,s_2$ with $s_2 \leq 1$ such that $(s_0, s_1, s_2, a_3, \dots )$ satisfies the full Stieltjes condition if and only if there exists $x_i, m_i, q_i$ with $x_i, m_i >0$ and $0\leq q_i \leq 1$ satisfying (1) and (2) for all $j\geq 3$. 
\end{lemma}

\begin{proof} 
	First assume $\brac{s_0, s_1, s_2, a_3, \dots a_n}$ satisfies the Stieltjes condition (or $\brac{s_0, s_1, s_2, a_3, \dots }$ satisfies the full Stieltjes condition) and $s_2 \leq 1$. By \Cref{stieltjes} (or \Cref{stieltjes classic}) there exists a discrete distribution on $(t_1, t_2, \dots t_k)$ where $t_i$ has mass $x_i$, $t_i >0$, and 
	$$
	\sum x_i t_i^j=
	\begin{cases} 
	a_j & 3 \leq j \leq n \text{ (or $j \geq 3$)} \\
	s_i & 0 \leq j \leq 2.
	\end{cases}
	$$
	Let $q_i=s_2$ for all $i$, and $m_i=t_i/s_2$ for all $i$. 
	Observe
	$$\sum x_i (m_i q_i)^j= \sum x_i \bfrac{t_i s_2}{s_2}^j= a_j
	 \text{ for all } 3 \leq j \leq n \text{ (or for all $j \geq 3$)} $$
	$$\sum x_i m_i^2 q_i= \sum \frac{x_i t_i^2}{s_2}=1.$$
	
	Next assume there exists $x_i, m_i, q_i$ satisfying the given conditions. Let $s_j= \sum x_i t_i^j$ for $j\in \{1,2,3\}$ and $t_i=m_i q_i$. Note
	$$
	\sum x_i t_i^j = \sum x_i(m_i q_i)^j= a_j \text{ for all }  3 \leq j \leq n \text{ (or for all $j \geq 3$)} ,
	$$ and so $\brac{s_0, s_1, s_2, a_3, \dots a_n}$  (or $\brac{s_0, s_1, s_2, a_3, \dots }$) is a moment vector of a finite distribution with positive support. It follows by Lemma \ref{stieltjes} (or \Cref{stieltjes classic}) that the moment vector satisfies the (full) Stieltjes condition. To see that $s_2 \leq 1$, let $q= \max_i q_i$ and observe 
	$$
	s_2= \sum x_i m_i^2 q_i^2 \leq q \sum x_i m_i^2 q_i  = q\leq 1.
	$$
	\end{proof}

\begin{lemma} \label{bipartite}
	There exists $s_0, s_2$ with $s_2 \leq 1$ such that $(s_0, s_2, a_4, \dots a_n)$ satisfies the Stieltjes condition if and only if there exists $x_i, m_i, q_i$ with $x_i, m_i >0$ and $0\leq q_i \leq 1$ satisfying
	\begin{enumerate}
		\item $2 \sum  x_i m_i^2 q_i= 1$
		\item $2 \sum  x_i (m_i q_i)^{2j}= a_{2j} \quad \text{ for all }  2 \leq j \leq n$.
	\end{enumerate} 
	Similarly, there exists $s_0, s_2$ with $s_2 \leq 1$ such that $(s_0, s_2, a_4, \dots a_n)$ satisfies the full Stieltjes condition if and only if there exists $x_i, m_i, q_i$ with $x_i, m_i >0$ and $0\leq q_i \leq 1$ satisfying (1) and (2) for all $j \geq 4$. 
\end{lemma}

\begin{proof} 
	First assume $\brac{s_0, s_2, a_4, a_6, \dots a_n}$ satisfies the Stieltjes condition (or $\brac{s_0, s_2, a_4, a_6, \dots}$ satisfies the full Stieltjes condition) and $s_2 \leq 1$. It follows that $\brac{\frac{s_0}{2}, \frac{s_2}{2}, \frac{a_4}{2}, \frac{a_6}{2}, \dots \frac{a_n}{2}}$ satisfies the Stieltjes condition (or $\brac{\frac{s_0}{2}, \frac{s_2}{2}, \frac{a_4}{2}, \frac{a_6}{2}, \dots} $ satisfies the full Stieltjes condition) because multiplying all entries of a matrix by a positive number does not change the sign of the determinant. By \Cref{stieltjes} (or \Cref{stieltjes classic}) there exists a discrete distribution on $(t_1, t_2, \dots t_k)$ where $t_i$ has mass $x_i$, $t_i >0$, and
	$$
	2\sum x_i t_i^j=
	\begin{cases} 
	a_{2j} & 2 \leq j \leq n \text{ (or $j \geq 4   t$) }\\
	s_{2j} & 0 \leq j \leq 1.
	\end{cases}
	$$
	Let  $q_i=s_2$, and $m_i=\sqrt{t_i}/s_2$ for all $i$. 
	Observe
	$$2 \sum x_i (m_i q_i)^{2j}= \sum x_i  t_i^j= a_{2j} \text{ for all } 2 \leq j \leq n \text{ (or $j \geq 4$) }$$
	$$2\sum x_i m_i^2 q_i=  \frac{2}{s_2} \sum x_i t_i=1.$$
	
	Next assume there exists $x_i, m_i, q_i$ satisfying the given conditions. Let $s_0= 2 \sum x_i$ and $t_i=(m_i q_i)^2$ for all $i$. Note 
	$$
	\sum x_i t_i^j = \sum \frac{x_i}{s_0}(m_i q_i)^{2j}= \frac{a_j}{2} \text{ for all }  2 \leq j \leq n \text{ (or $j \geq 4$)}.
	$$
	Let $q= \max_i q_i$,  $s_1=2 \sum x_i t_i$, and observe
	$$
	s_1=2 \sum x_i t_i= 2\sum x_i m_i^2 q_i^2 \leq 2q \sum x_i m_i^2 q_i  = q \leq 1.
	$$
	It follows that $\brac{\frac{s_0}{2}, \frac{s_2}{2},  \frac{a_4}{2}, \dots \frac{a_n}{2}}$ is a moment vector (or $\brac{\frac{s_0}{2}, \frac{s_2}{2},  \frac{a_4}{2}, \dots }$ is a moment vector), and therefore by \Cref{stieltjes} (or \Cref{stieltjes}) satisfies the (full) Stieltjes condition. It follows that $\brac{s_0, s_1,  a_2, \dots a_n}$ also satisfies the Stieltjes condition (or $\brac{s_0, s_1,  a_2, \dots}$ also satisfies the full Stieltjes condition) because multiplying all entries of a matrix by a positive number does not change the sign of the determinant. 
\end{proof}

\begin{proof}(of \Cref{ach,full-ach})
First assume the vectors of $s_i$ and $t_i$ satisfy the hypotheses. Then by \Cref{plain}, there exists $(x_i, m_i, q_i)$ satisfying $\sum x_i (m_i q_i)^j= s_j$ and $ \sum x_i m_i^2 q_i=1$. For each triple add the triple $(m_i, q_i, \beta_i=0)$ to the distribution. Soon we will specify the corresponding probability $\mu_i$. By \Cref{bipartite}, there exists $(x_i, m_i, q_i)$ satisfying $2\sum x_i (m_i q_i)^{2j}= t_{2j}$ and $ 2\sum x_i m_i^2 q_i=1$. For each triple add the triple $(m_i, q_i, \beta_i=1)$ to the distribution.
Let $B$ be the set of indices $i$ such that $\beta_i=1$ and $B^c$ be the set of indices $i$ such that $\beta_i=0$. Let $z=\sum_{i \in B} x_i \gamma + \sum_{i \in B^c} x_i (1-\gamma)$. We now define $\mu$ by assigning probabilities to triples $(m_i, q_i, \beta_i)$. If $i \in B^c$, let $\mu_i=x_i \gamma/z$. If $i \in B$, let $\mu_i= x_i (1-\gamma)/z$.  Note $\sum \mu_i=1$, and therefore $\mathcal{D}=(\mu, a)$ is a well-defined ROC family with $$x= 1/ \brac{\sum_{i \in B^c} \mu_i m_i^2 q_i +2\sum_{i \in B} \mu_i m_i^2 q_i }=z.$$ \Cref{polynomials} implies that the family achieves the desired limit with sparsity exponent $a$.

Suppose the limit is achievable by some ROC family $\mathcal{D}=(\mu,a)$ where $a$ is the sparsity exponent. Let $\gamma=x \sum_{i \in B^c} \mu_i m_i^2 q_i$, and so $1-\gamma= 2x \sum_{i \in B}\mu_i m_i^2 q_i$. For each $i \in B^c$, let $x_i= x \mu_i /\gamma$. Note $\sum_{i \in B^c} x_i m_i^2 q_i= 1$, and so by \Cref{plain}, the vector with $s_j = \sum_{i\in B^c} \mu_i (m_i q_i)^j$ satisfies the Stieltjes condition. For each $i \in B$, let $x_i= x \mu_i /(1-\gamma)$. Note $2\sum_{i \in B} x_i m_i^2 q_i= 1$, and so by \Cref{bipartite}, the vector with $t_{2j} = \sum_{i\in B} \mu_i (m_i q_i)^{2j}$ satisfies the Stieltjes condition.  \Cref{polynomials}  implies that $c_j$ or $w_j$ is the appropriate combination of $s_j$ and $t_j$. 
\end{proof}

Finally we show that a similar argument proves the condition for when it is possible to match a $k$-cycle-to-edge vector with a ROC family. Instead of using the sparsity exponent of the sequence to determine how the size of communities in the matching ROC family scale, the communities in the matching ROC family are constant size.

\begin{proof}(of \Cref{thm:kratios})	First assume the vectors of $s_i$ and $t_i$ satisfy the hypotheses. Then by \Cref{plain}, there exists $(x_i, m_i, q_i)$ satisfying $\sum x_i (m_i q_i)^j= s_j$ and $ \sum x_i m_i^2 q_i=1$. For each triple add the triple $(m_i, q_i, \beta_i=0)$ to the distribution. Soon we will specify the corresponding probability $\mu_i$. By \Cref{bipartite}, there exists $(x_i, m_i, q_i)$ satisfying $2\sum x_i (m_i q_i)^{2j}= t_{2j}$ and $ 2\sum x_i m_i^2 q_i=1$. For each triple add the triple $(m_i, q_i, \beta_i=1)$ to the distribution.
	Let $B$ be the set of indices $i$ such that $\beta_i=1$ and $B^c$ be the set of indices $i$ such that $\beta_i=0$. Let $z=\sum_{i \in B} x_i \gamma + \sum_{i \in B^c} x_i (1-\gamma)$. We now define $\mu$ by assigning probabilities to triples $(m_i, q_i, \beta_i)$. If $i \in B^c$, let $\mu_i=x_i \gamma/z$. If $i \in B$, let $\mu_i= x_i (1-\gamma)/z$.  Note $\sum \mu_i=1$, and therefore $(\mu, 0)$ is a well-defined ROC family with $$x= 1/ \brac{\sum_{i \in B^c} \mu_i m_i^2 q_i +2\sum_{i \in B} \mu_i m_i^2 q_i }=z.$$ Note $c(j)=c_j/2$ by construction. For $G \sim ROC(n,d, \mu,0)$ and $d= o \brac{ n^\frac{1}{k-1}}$, \Cref{exp by class cor} implies that $\E {C_j(G)}= \frac{c_j}{2} nd +o\brac{nd}$. The statement follows.

	For the other direction, suppose there is a ROC family that achieves the limit, meaning that for $G$ drawn from the family, $\E {C_j(G)}=  \frac{c_j}{2} nd +o\brac{nd}$.  In a graph drawn from a ROC family with  parameters $\mathcal{D}=(\mu,a)$, the number of $j$ cycles is $\Theta(nd^{1+a(j-2)})$. It therefore must be the case that the ROC family achieves the limit has $a=0$.  Let $(\mu,0)$ be the ROC family. Let $\gamma=x \sum_{i \in B^c} \mu_i m_i^2 q_i$, and so $1-\gamma= 2x \sum_{i \in B}\mu_i m_i^2 q_i$. For each $i \in B^c$, let $x_i= x \mu_i /\gamma$. Note $\sum_{i \in B^c} x_i m_i^2 q_i= 1$, and so by \Cref{plain}, the vector with $s_j = \sum_{i\in B^c} \mu_i (m_i q_i)^j$ satisfies the Stieltjes condition. For each $i \in B$, let $x_i= x \mu_i /(1-\gamma)$. Note $2\sum_{i \in B} x_i m_i^2 q_i= 1$, and so by \Cref{bipartite}, the vector with $t_{2j} = \sum_{i\in B} \mu_i (m_i q_i)^{2j}$ satisfies the Stieltjes condition.  \Cref{polynomials}  implies that $c_j/2$  is the appropriate combination of $s_j$ and $t_j$. 
\end{proof}

\subsubsection{Simple criterion for the the Stieltjes condition.} \label{sc shortcut}

The following lemma provides a convenient criterion that implies the truncated Stieltjes condition. In particular, we use this show that the limit of hypercube sequence is totally $k$-achievable (\Cref{h1}).

\begin{lemma} \label{char}
	Let $s_1, s_2, \dots s_k$ be a vector with $s_1 >0$ satisfying $ s_{x} s_{y}< s_{a} s_{b}$ for all $1\leq a< x \leq  y< b$. Then there exists $s_0>0$ such that $(s_0, s_1, s_2, \dots s_k)$ satisfies the truncated Stieltjes condition.
\end{lemma}

The following is the key lemma for proving \Cref{char}.

\begin{lemma} \label{key}
	Let $s_1, s_2, \dots s_k$ be a vector with $s_1 >0$ and $ s_{x} s_{y}< s_{a} s_{b}$ for all $1\leq a< x \leq  y< b$. Let $H$ be the $\lfloor \frac{k+1}{2} \rfloor \times \lfloor \frac{k+1}{2} \rfloor $  with $H_{ij}= s_{i+j-1}$. Then all leading principal minors of $H$ have positive determinant. 
\end{lemma}

\begin{proof} 
	Let $H_k$ denote the $k^{th}$ leading principal minor of $H$ (the square sub-matrix obtained by restricting to the first $k$ rows and first $k$ columns). We show that $det(H_k)>0$ by induction on $k$. Note $det(H_1)=s_1 >0$. Next assume $det(H_{k-1})>0$. Write $H_k=AB$ where $A$ and $B$ are in the form displayed here. 
	$$H_k= 
	\begin{pmatrix}
	s_1 & s_2 & s_3 & \dots & s_k\\
	s_2 & s_3 &     &       &\vdots\\
	s_3 &     &     &       &\vdots \\
	\vdots  &     &     & \ddots      &\vdots \\
	s_k  & \dots    &    \dots &\dots       & s_{2k-1}\\
	\end{pmatrix}
	=	\begin{pmatrix}
	s_1 & s_2 & \dots & s_{k-1} & 0\\
	s_2 &      &       &\vdots& 0\\
	\vdots &          &     \ddots  & \vdots &\vdots\\
	s_{k-1}  &       \dots   & \dots  & s_{2k-3}   & 0 \\
	0  &     0&  \dots         & 0&1\\
	\end{pmatrix}
	\begin{pmatrix}
	1	& 0 & \dots & 0 & x_1\\
	0 & 1 &            &\vdots&x_2\\
	\vdots &     &         \ddots  & \vdots&\vdots\\
	0  &  \dots   &  \dots     & 1   & x_{k-1} \\
	s_k  &s_{k+1}     &  \dots         & s_{2k-2} &s_{2k-1}\\
	\end{pmatrix}.
	$$	
	Note $det(A)= det(H_{k-1})$, which is positive by the inductive hypothesis. It follows there exists a unique solution of real values $x_1, x_2, \dots x_{k-1}$ so that  $H_k=AB$. Since $det(H_k)=det(A) det(B)$ and $det(A)>0$, it suffices to show that $det(B) >0$ to prove the inductive hypothesis. 
	
	Note $det(B)=s_{2k-1} - L $ where 
	$$L= \begin{pmatrix} s_k & s_{k+1} & \dots & s_{2k-2} \end{pmatrix} \begin{pmatrix} x_1& x_2& \dots &x_{k-1}\end{pmatrix}^{T}.$$ 
	By construction of $A$ and $B$, 
	\begin{equation}
	\label{m r}	\begin{pmatrix}
	s_1 & s_2& \dots & s_{k-1}\\
	s_2 &     &       &\vdots\\
	\vdots  &          & \ddots      &\vdots \\
	s_{k-1}  & \dots    &    \dots      & s_{2k-3}\\
	\end{pmatrix}
	\begin{pmatrix} x_1\\x_2\\\vdots\\ x_{k-1}\end{pmatrix}
	= \begin{pmatrix} s_k \\ s_{k+1} \\ \vdots \\ s_{2k-2} \end{pmatrix}.
	\end{equation}
	For $i \in [0,k-2]$, define $$\alpha_i= \frac{ s_{k+i}}{s_{i+1}+ \dots + s_{k-1+i}},$$ and let $\alpha=\max_{i \in [0,k-2]} \alpha_i$.
	Therefore for all $i \in [0,k-2]$
	$$  x_{i+1} s_{k+i}= \alpha_i x_{i+1} (s_{i+1}+ \dots + s_{k-1+i}) \leq \alpha   x_{i+1} (s_{i+1}+ \dots + s_{k-1+i}).$$
	Summing the above equation over all $i \in [0,k-2]$ and applying equality \Cref{m r} yields $$L\leq \alpha (s_k+ s_{k+1} + \dots + s_{2k-2}).$$ To prove $det(B)=s_{2k-1}-L >0$ we show that for all $i \in [0,k-2]$, $\alpha_i (s_k+ s_{k+1} + \dots + s_{2k-2}) < s_{2k-1}$, or equivalently 
	\begin{equation} \label{last eq} s_{k+i}(s_k+ s_{k+1} + \dots + s_{2k-2}) < s_{2k-1} 
	(s_{i+1}+ \dots + s_{k-1+i}).\end{equation} Note by assumption $s_{k+i}s_{k+j}< s_{2k-1} s_{i+j+1}$ for all $i, j \in[ 0, k-2]$. Therefore the $j^{th}$ term on the left side of \Cref{last eq} is less than the $j^{th}$ term on the right side of \Cref{last eq}, and so \Cref{last eq} holds. 
\end{proof}

\begin{proof}(of \Cref{char}.)
	Given the vector $(s_0,s_1, \dots s_k)$ define Hankel matrices as described in \Cref{s con}. In the case of $k=2j+1$ odd, let $H^{(0)}=H_{2j}^{(0)}$ and $H^{(1)}=H_{2j+1}^{(1)}$.  In the case $k=2j$ even, let $H^{(0)}=H_{2j}^{(0)}$ and $H^{(1)}=H_{2j-1}^{(1)}$. 
	\Cref{key} implies that all leading principal minors of $H^{(1)}$ have positive determinant. By \Cref{cond-helper}, it remains to show that there exists $s_0>0$ such that all leading principal minors of $H^{(0)}$ have positive determinant. 
	
	Let $H'$ be  $H^{(0)}$ with the first row and column deleted. The $i^{th}$ leading principal determinant of $H^{(1)}$ has the form $s_0 h_i +b_i$ where $h_i$ is the $(i-1)^{st}$ principal determinant of $H'$. (Taking the determinant via expansion of the first row makes this clear.) Note that \Cref{key} applied to the vector $(s_2, s_3, \dots s_k)$ guarentees that each $h_i>0$. Therefore, it is possible to pick $s_0$ sufficiently large such that all principal determinants $s_0 h_i +b_i$ of $H^{(1)}$ are positive.
\end{proof}

\subsection{Examples of achievable limits: hypercube and rook sequences } \label{sec:hc and rook}

In this section, we prove \Cref{h1} and \Cref{r1} which state that the limit of the sequence of hypercubes is totally $k$-achievable and the limit of the sequence of rook graphs is fully achievable respectively. First we provide the ROC parameters which achieve the $6$-limit of the hypercube.

\begin{remark} 
The ROC family $(\mu, 1/2)$ where $\mu$ is the distribution with support size one on $(8, 1/4,1)$ achieves the $6$-limit of the hypercube sequence. To achieve longer limits, the distribution will have larger support. 
\end{remark}

We now prove \Cref{h1}. Recall from \Cref{hypercube} that the sparsity exponent of the hypercube sequence is $1/2$.  Therefore, to prove the theorem we apply \Cref{half} of \Cref{ach}, which states the vector $(w_3, w_4, \dots w_k)$ can be achieved by ROC if the normalized cycle count vector $(c_3, c_4, \dots c_k)$ corresponding to the transform $T$ can be extended to satisfy the Stieltjes condition. The following lemma gives the cycle vector for the hypercube.

\begin{lemma} \label{cube cycles} Recall from \Cref{hypercube} that the limit of the hypercube sequence $(G_d)$ is
	$(w_3, w_4, \dots)$ where 
	$$w_j = \begin{cases} 
	(j-1)!! & \text{for $j$ even}\\
	0 & \text{for $j$ odd}.
	\end{cases}
	$$ 
	For $T$ the cycle transform given in $\Cref{relationship}$, $T((0, s_2, 0, s_3, 0, \dots))= (0, w_4,0, w_6, 0, \dots)$ where $s_1=1$ and $s_n=(n-1) \sum_{j=1}^{n-1} s_j s_{n-j}$.
\end{lemma}

\begin{remark} 
	The even terms of sequence described in \Cref{cube cycles} form the Online Encyclopedia of Integer Sequences\footnote{\href{https://oeis.org}{https://oeis.org}} (OEIS) entry A000699: $1, 1, 4, 27, 248, 2830, \dots $
	\end{remark}

	\begin{proof}(of \Cref{cube cycles}) Note that the hypercube is vertex transitive (meaning the graph's automorphism group acts transitively on its vertices) so the sequence of hypercubes is essentially $k$-locally regular. Therefore, \Cref{locally reg} implies $C((0, w_4,0, w_6, 0, \dots))= (c_3, c_4, c_5, \dots )$ where $$c_i= \lim_{d \to \infty }\frac{C_i(G_d)}{2^d d^{i/2}}$$ where $C_i(G_d)$ is the normalized number of $i$ cycles at a vertex.
		Instead of applying polynomial operations to the vector $(0,w_4, 0, w_6,0, \dots)$ to obtain $(c_3, c_4, c_5, \dots)$ we directly compute $C_k(G_d)$, the number of cycles in the $d$-dimensional hypercube graph. As in \Cref{hypercube}, we think of $k$-walks on the hypercube as length $k$ strings where the $i^{th}$ character indicates which of the $d$ coordinates is changed on the $i^{th}$ edge of the walk. For closed walks each coordinate that is changed must be changed back, so each coordinate that appears in the string must appear an even number of times.  For $1 \leq i \leq k/2$, let $Z_i$ be the number of such strings of length $k$ that involve $i$ coordinates and correspond to a $k$-cycle on the hypercube graph. Since there are $d$ coordinates $Z_i=\Theta(d^{i})= o(d^{k/2})$ for $i<k/2$ and so $$C_k(G_d)= n Z_{k/2} +o\brac{nd^{k/2}}$$ where $n=2^d$ is the number of vertices.
		
		We compute $Z_{k/2}$ by constructing a correspondence between length $k$ strings with $k/2$ characters each appearing twice that represent cycles and irreducible link diagrams. A link diagram is defined as  $2n$ points in a line with $n$ arcs such that each arc connects precisely two distinct points and each point is in precisely one arc. The arcs define a complete pairing of the integers in the  interval $[1, 2n]$. A link diagram is reducible if there is a subset of $j<n$ arcs that form a complete pairing of the integers in a subinterval of $[1,2n]$ and irreducible otherwise. Let $S$ be the set of length $k$  strings in which the characters $1,2, \dots k/2$ each appear twice and the first appearance of character $i$  occurs before the first appearance of $j$ for all $i <j$.  Let $L$ be the set link diagrams $L$ on $k$ points. Let $\overline{S} \subseteq S$ be the subset of strings that correspond  to cycles on the hypercube and let $\overline{L} \subseteq L$ be the set of irreducible link diagrams.
		
		We construct a bijection $f: S \to L$ and show $f$ restricted to $\overline{S}$ gives a bijection between $\overline{S}$ and $\overline{L}$. Given $s\in S$, we  construct a corresponding link diagram $f(s) \in L$ by labeling $k$ points so that the $i^{th}$ point is labeled with the $i^{th}$ character of $s$ and then drawing an arc between each pair of points with the same label. Note $f$ is a bijection. (To produce $f^{-1}(\ell)$ label the arcs $1, 2, \dots k/2$ by order of their left endpoints. Label each point with the label of its arc and read off the string of the labels.) It remains to show that $f(s) \in \overline{L}$ if and only if $s \in \overline{S}$. We prove the contrapositive. Suppose $s \not \in \overline{S}$. Then the hypercube closed walk corresponding to $s$ is not a cycle. Therefore there exists $j$ and steps $i , i+1, \dots i+j$ of the walk that make a $j/2$ cycle. (Here by convention traversing an edge twice is a 2-cycle.) Since $i , i+1, \dots i+j$ form a cycle, each coordinate that was changed between step $i$ and step $i+j$ must have been changed back. Therefore each character that appears in the interval $[i, i+j]$ appears appears twice. It follows $f(s)$ has a complete pairing of the integers in the subinterval $[i, i+j]$ and therefore is reducible. Next suppose $\ell \not \in \overline{L}$. Then there exists a subinterval $[i, i+j] \not =[1, 2n]$ with a complete pairing of the integer points. Therefore the walk corresponding to $f^{-1}(\ell)$ is not a cycle because the walk visits the same vertex before  step $i$ and after step $i+j$. It follows that $f^{-1}(\ell) \not \in \overline{S}$.

		Thus $|\overline{L}|=|\overline{S}|$. See \cite{ste78} for a proof that $|L|= s_{k/2}$. There are  $d(d-1) \dots (d-k/2+1)= d^{k/2} + o\brac{d^{k/2}}$ ways to select the $k/2$ coordinates in the order they will be changed. Thus $Z_{k/2}= d^{k/2} |S|+ o\brac{d^{k/2}}= s_{k/2} d^{k/2} +o\brac{d^{k/2}}$, and therefore  $$C_k(G_d)=s_{k/2} nd^{k/2} +o\brac{nd^{k/2}}.$$ \end{proof}

\begin{lemma}[from \cite{ste78}]
	\label{stein and everett} Let $s_1=1$ and $s_n=(n-1) \sum_{j=1}^{n-1} s_j s_{n-j}$.
	Let $s_1=1$ and $s_n=(n-1) \sum_{j=1}^{n-1} s_j s_{n-j}$. Then
	$$s_{n+1} > (2n+1) s_n \quad \text{ for } \quad n \geq 4$$
	$$s_{n+1} < (2n+2) s_n\quad \text{ for } \quad n \geq 1.$$
\end{lemma}

\begin{lemma}\label{s prop} 
	Let $s_1=1$ and $s_n=(n-1) \sum_{j=1}^{n-1} s_j s_{n-j}$.
	Let $s_1=1$ and $s_n=(n-1) \sum_{j=1}^{n-1} s_j s_{n-j}$. Then there exists $s_0>0$ such that $(s_0, s_1, s_2, \dots s_k)$ satisfies the Stieltjes condition.
\end{lemma}

\begin{proof}
	 We apply \Cref{char} which says a vector with  $s_{x} s_{y}< s_{a} s_{b}$ for all $1\leq a< x \leq  y< b$ can be can be extended to satisfy the Stieltjes condition. We show this conditions holds for the infinite vector.
	 
	 First we consider the case when $y \geq 4$. By \Cref{stein and everett} $$s_b > \frac{ (2b-1)!!}{(2y-1)!!} s_y \quad \text{ and } \quad s_x< \frac{(2x)!!}{(2a)!!} s_a.$$
	 Note $\frac{(2x)!!}{(2a)!!}<\frac{ (2b-1)!!}{(2y-1)!!}$, and therefore
	 $$s_xs_y < \frac{(2x)!!}{(2a)!!}s_a s_y < \frac{ (2b-1)!!}{(2y-1)!!}s_a s_y < s_a s_b.$$
	 
	 We consider the remaining three cases separately. For $a=1, x=2, y=3$, $b\geq 4$ so  $s_b \geq 27$. Therefore $s_xs_y=4< 27\leq  s_a s_b$. For $a=1, x=3, y=3$, $b\geq 4$ so  $s_b \geq 27$. Therefore $s_x s_y=16< 27\leq  s_a s_b$. For $a=1, x=2, y=2$, $b \geq 3$ so $s_b \geq 4$. Therefore $s_x s_y=1 < 4 \leq s_a s_b$. 
	 
	\end{proof}

\medskip

\begin{proof}(of \Cref{h1})
Follows directly from \Cref{cube cycles}, \Cref{s prop}, and \Cref{half} of \Cref{ach}.
\end{proof}

\begin{remark} 
The limit of the sequence of hypercubes is not fully achievable. 
\end{remark}

\begin{proof}
For a ROC family $( \dist, 1/2)$ with $m= \max_i m_i$, the  $w_k$ coordinate in the limit vector is at most $x(2m)^k$. However the $w_k$ coordinate in the hypercube sequence is $ (k-1)!!  = \Theta \brac{ \bfrac{k}{e}^{k/2}}$.  Therefore there is no $\dist$ which achieves the full hypercube limit vector. 
\end{proof}

In \Cref{gen hyper}, we show that two generalizations of the hypercube have the same limit and therefore are also totally $k$-achievable. We now prove 
\Cref{r1} which states that the limit of the sequence of rook's graphs is fully achievable.

\begin{proof}(of \Cref{r1}.)
Recall from \Cref{rook} that the sequence of $(G_k)$ has sparsity exponent 1 and converges to the vector with $w_i=2^{2-i}$. 
By \Cref{polynomials}, the ROC family with $a=1$ and $\dist$ the distribution that selects $m=1/2$ and $q=1$ with probability $1$ achieves this limit. 
\end{proof}

\section{Discussion}\label{discussion}

\subsection{Limitations of the ROC model}\label{limitations}

We have shown that the limiting closed walk counts of many graph sequences can be approximated by the ROC model; the model is succinct and can be easily sampled to produce graphs with the same normalized walk counts as the sequence up to terms that disappear as the size of the sampled graph grows. \Cref{ach,full-ach} give necessary and sufficient conditions for when a limit vector or sequence can be achieved. The natural next question is whether all graphs sequences converge to a limit that can be achieved by a ROC family. The answer to this question is no. There are both sequences of graphs that are not convergent in any of the senses we have defined, and convergent sequences of graphs with limits that are not achievable by a ROC family. We discuss such examples in this section.

\paragraph{Non-convergent graph sequences.} \label{non convergent}
Not all sequences of graphs converge or have a convergent subsequence. Consider the following sequence of ROC graphs drawn from different ROC families.

\begin{example}
	Let $a \in [1/2,1]$ and let $\mu_i$ be the distribution on one point $m_i=i$ and $q_i=1$. Let $(G_i)$ be a sequence of graphs with $G_i \sim ROC(n_i,d_i,\mathcal{D})$ such that $d_i$ satisfies the degree conditions given in \Cref{achievability defn}. The sequence $(G_i)$ is not $k$-convergent for any $k$ and is not fully convergent.
\end{example}

\noindent By \Cref{polynomials}, $\E {W_3(G_i)}= m_i q_i^2 n_i d_i^{1+a(k-2)} +o\brac{n_i d_i^{1+a(k-2)}}$. Therefore $W_3(G_i,a)=i+\ve(i)$ and for all $\alpha>a$ $W_3(G_i,\alpha)=0+\ve(i)$ where $\ve(i)$ is an error term that vanishes with high probability as $i$ tends to infinity. It follows that almost surely the sparsity exponent and $k$-sparsity exponent are $a$ and the sequence  $W_3(G,a)$ does not converge. Thus, the sequence is almost surely not $k$-convergent or fully convergent.

\paragraph{Limit sequences that are not achievable by any ROC model.} \label{not achievable}

We give a sequence of graphs with increasing degree that converges to a limit that is not achievable by any ROC family and provide a method for producing such sequences. First we give a necessary condition on $(w_3, w_4, w_5, w_6)$ for it to appear as the prefix of a limit vector achievable by a ROC family.

\begin{lemma} \label{3456} 
If $(w_3, w_4, w_5, w_6)$ is a prefix of a $k$-limit that can be achieved by ROC family with sparsity exponent $>1/2$, then 
$$ w_3^2 w_6 \geq w_5^3.$$
\end{lemma}

\begin{proof} 
By \Cref{more than half} of \Cref{ach}, there exists $\gamma \in [0,1]$, $s_0, s_1$, and $s_2 \leq 1$ such that $(s_0, s_1, \dots s_k)$ satisfies the truncated Stieltjes condition, $w_j=s_j \gamma$ for $j$ odd, and $w_j\geq s_j \gamma$ for $j$ even. By \Cref{stieltjes}, this implies that $H_{6}^{(0)}$ and $H_{5}^{(1)}$ are positive semi-definite. 
Since the all principle minors of positive semi-definite matrices have non-negative determinants, it follows that the principal minors
$$
\begin{pmatrix} 
s_4 & s_5\\
 s_5 & s_6
\end{pmatrix}
=\begin{pmatrix} 
s_4 & \frac{w_5}{\gamma}\\
\frac{w_5}{\gamma} & s_6
\end{pmatrix}
\quad
\text{ and }
\quad
\begin{pmatrix} 
s_3 & s_4\\
 s_4 & s_5
\end{pmatrix}
=\begin{pmatrix} 
\frac{w_3}{\gamma}& s_4\\
s_4 & \frac{w_5}{\gamma}
\end{pmatrix}
$$
of $H_{6}^{(0)}$ and $H_{5}^{(1)}$ respectively 
have non-negative determinant. Therefore  $$\gamma^2 s_4 s_6 \geq w_5^2 \quad \text{ and } \quad \gamma^2 s_4^2 \leq w_3 w_5,$$
and so $\gamma s_6 \geq \frac{w_5^{3/2}}{\sqrt{w_3}}$. Since $w_6 \geq \gamma s_6$, the statement follows.
\end{proof}

Next we show how to construct a sequence of graphs with increasing degree that fails this condition. This construction is due to Shyamal Patel.

\begin{lemma} \label{shyamal}
Let $G_0$ be a graph. We construct a sequence $G_i$ as follows. Let $G_i$ be the graph with adjacency matrix $\begin{pmatrix} A_{i-1} & A_{i-1} \\ A_{i-1} & A_{i-1}\end{pmatrix}$ where $A_{i-1}$ is the adjacency matrix of $G_{i-1}$. Let $w_j= W_j(G_0, 1)$. Then for each $i$,
$$W_j(G_i,1)=w_j,$$ and so $(G_i)$ converges to $(w_3, w_4, w_5, \dots )$ with sparsity exponent 1.
\end{lemma}

\begin{proof} Let $G_0$ be a graph on $n$ vertices with average degree $d$. Let $A_0$ be the adjacency matrix of $G_0$ and let $\lambda_1 \geq \lambda_2 \geq \dots \lambda_\ell$ be the non-zero eigenvalues of $A_0$. Note that if $\lambda$ is an eigenvalue of $A_{i-1}$ with eigenvector $v$, then $2 \lambda$ is an eigenvalue of $A_{i}$ with eigenvector $\begin{bmatrix} v \\ v\end{bmatrix}$. Since the adjacency matrix  $A_i$ has the same rank as the adjacency matrix of $A_{i-1}$, the set of non-zero eigenvalues of $A_i$ is precisely the set of non-zero eigenvalues of $A_{i-1}$ in which each is doubled. Therefore $A_i$  has non-zero eigenvalues $2^i \lambda_1, 2^i \lambda_2, \dots 2^i \lambda_\ell$.  Note $G_i$ has $G_0 2^i$ vertices and average degree $2^i d$. Therefore
$$W_j(G_i, 1)= \frac{\sum_{b=1}^\ell (2^i \lambda_b)^j}{2^i n (2^i d)^{j-1}}= \frac{\sum_{b=1}^\ell ( \lambda_b)^j} {n d^{j-1}}=w_j.$$
\end{proof}

The lemma implies that if there is a graph with $W_j(G,1)=w_j$, then there is a sequence of graphs $(G_i)$ with increasing degree that converges to this limit with sparsity exponent 1. 
Taking $G_0$ to be a girth four graph yields a sequence $(G_i)$ with a limit vector that violates the condition of \Cref{3456}. This sequence is dense since $d_i=\Theta( n_i)$.
However, we can construct a sparser sequence $(G_i')$ from $(G_i)$ with the same limit by taking each $G_i'$ to be the union of disjoint copies of $G_i$.

\begin{lemma} \label{disjoint copies}
	Let $(G_i)$ be a convergent sequence of graphs with sparsity exponent $\alpha$. Let $(t_i)$ be a sequence of positive integers, and let $(G_i')$ be a graph sequence in which $G_i'$ consists of $t_i$ disjoint copies of $G_i$. Then $(G_i')$ achieves the same limit as $G_i$ with the same sparsity exponent.
	\end{lemma}

\begin{proof} Note that $G_i$ and $G_i'$ both have average degree $d_i$ and the number of vertices in $G_i'$ is $n_i'= t_i n_i$. Note also that $W_k(G_i')=t_i W_k(G_i)$. It follows that $$W_k(G_i', \alpha) =\frac{ W_j(G_i')}{n_i'd_i'^{1+a(j-2)}}=\frac{ W_j(G_i)}{n_id_i^{1+a(j-2)}}.$$
\end{proof}

We use \Cref{shyamal,3456} to construct a family of sequences with arbitrary sparsity that are not achievable by the ROC model. This example implies that there is no class of densities for which the ROC model can capture all $6$-limits of sequences with the specified density.  

\begin{example} \label{break} Let $G_0$ be the five cycle. Then the sequence $(G_i)$ defined as in \Cref{shyamal} converges to a limit that cannot be achieved by any ROC family. This limit $(0, 3/4, 1/8, 5/8)$ is at constant distance from any achievable limit. Moreover, there exists a sequence $(G_i')$ with the same limit and $d_i'=f(n_i')$ for any function $f(n)=o(n)$. To see this, apply \Cref{disjoint copies} to the sequence $(G_i)$ with $t_i= f^{-1}(n_i)/n_i$.  
\end{example}

\subsection{Extensions of ROC and relation to sparse graph limit theory}

We have seen that the ROC model provides a sampleable approximation of the limits of many sparse graph sequences, in particular the hypercube sequence. Our metric was defined in terms of a vector of closed walk counts of each length appropriately normalized. This vector is a natural choice because closed walk counts are equivalent to the moments of the eigenspectrum, and the normalization factor encodes average density of local neighborhoods.  We end by discussing how the ROC model fits into sparse graph limit theory and mentioning future directions that illustrate the potential of the ROC model.

\paragraph{Distance and convergence.} Our notion of convergence based on normalized closed walk count vectors  differs from other notions of graph convergence in two key ways. First, our theory does not provide an inherent metric for describing the distance between two graphs. The normalization factor used to determine the convergence of a sequence of graphs depends on the rate of growth of the closed walk counts in the sequence. Therefore, it not clear which normalization factor $\alpha$ to use when comparing the closed walk vectors of just two graphs. Second, due to this flexibility in normalization parameter $\alpha$, the space of all vectors of normalized closed walk counts is unbounded, and so it is possible to construct sequences of graphs with no convergent subsequence (as in \Cref{non convergent}). In contrast the set of local profiles and the set of graphons are compact, so every sequence of graphs in these settings has a convergent subsequence. 

\paragraph{Capturing cuts.}
While a graph $H$ drawn from the ROC model may capture the closed walk counts of a graph $G$, there is no guarantee that $H$ and $G$  will have similar cuts. (The in the local profile approach for bounded degree graphs also succeeds at encoding local properties and fails to capture the global property of cuts.) For example, consider a convergent sequence of connected graphs $(G_i)$ and a sequence of graphs $(G_i')$ where $G_i'$ is a collection of disjoint copies of $G_i$. \Cref{disjoint copies} implies $(G_i)$ and $(G_i')$ have the same limit; however the cuts in these sequences greatly differ. Moreover the cuts of a ROC graph drawn from the family that achieves the limit need not have cuts that match either $G_i$ or $G_i'$.

In general, even if the moments of the eigenspectra of two graphs match, their spectral gaps and precise set of eigenvalues may greatly differ.  In Appendix C we discuss a different approximation of the spectrum of the hypercube graph. It is not of constant size (the size of the approximation grows with $d$ for a hypercube of size $2^d$), but it captures the $d$ distinct eigenvalues of the hypercube precisely (and therefore the minimum cut). On the other hand, the approximation does not preserve information about the multiplicities of the eigenvalues, and hence does not capture the walk counts.

\paragraph{An extension of the ROC model.}
We imagine the following extension to the ROC model that has the potential to encode information about the cuts of a graphs and give a finer grained approximation of local structure while also maintaining the approximation of closed walk counts. Begin with a partition of the vertex set, and for each community type specify a distribution over partition classes. Then, when adding a community to a ROC graph,
select vertices for the community based on the corresponding distribution over partition classes. This modification has the potential to better approximate cuts because it is possible to control the number of edges between partition classes.

Moreover, the above modification will likely be a better approximation for graphs that are not close to locally regular. Currently a ROC approximation produces a graph in which each vertex is in approximately the number of closed walks as an average vertex in the target graph. 
An expanded theory, perhaps including the above modification, could create graph approximations that capture the distribution of local closed walk counts vectors at each vertex. 

\paragraph{Achievability of all limits.}
As demonstrated in \Cref{not achievable}, not all limit sequences can be achieved by a ROC model. In particular, the model may not be able to capture the limits of sequences of girth five graphs because the density of the communities need to produce many five cycles will also produce many three and four cycles. This problem could be resolved by generalizing the model so that communities may have structure other than E-R random graphs. Alternately, the aforementioned approach of adding ROCs between partition classes might provide sufficient flexibility to achieve a wider range of limits.

\subsection{Additional open questions.}

\begin{enumerate}

	\item A further generalization involves adding particular sub\networks from a specified set according to some distribution instead of E-R graphs in each step (e.g., perfect matchings or Hamiltonian paths). Such constructions could provide greater flexibility in approximating graphs or could be useful for answering extremal questions.

	\item A fundamental question in the study of graphs is how to identify relatively dense clusters. For example, clustering protein-protein interaction networks is a useful technique for identifying possible cellular functions of proteins whose functions were otherwise unknown \cite{Ste05, Kro06}. An algorithm designed specifically to identify the communities in a \network drawn from the ROC model has potential to become a state-of-the-art algorithm for clustering with overlap.
	
	\item A ROC graph $H$ that approximates a target graph $G$ has similar closed walk counts as $G$. To what extent does this similarly imply that algorithms will behave similarly on $G$ and $H$? For instance, can we analyze the behavior of random walks or percolation of ROC graphs? How does this compare to the behavior of the same process on other graphs with the same closed walk counts?

	\item	The asymptotic thresholds for properties of ROC graphs have yet to be studied. (See \cite{Fri15} for a survey on E-R random graphs.) Which phase transitions appearing in E-R  random graphs also appear in ROC graphs? Does every nontrivial monotone property have a threshold? These questions have been answered in special cases in \cite{anastos2019thresholds}.
	 
\end{enumerate}

\section*{Acknowledgements}
The authors would like to thank Shyamal Patel for his construction of the class of graphs with non-achieveable limit given in \Cref{shyamal} and Winston Li for help with real-world networks.
Lutz Warnke for helpful conversations and pointers to relevant work, and Victor Veitch for his correspondence which clarified our understanding of graphexes. Finally, we thank L\'aszl\'o Lov\'asz for inspiring us to think about the sequence of hypercube graphs. 

\bibliographystyle{plain}
\bibliography{graphapprox}

\appendix

\section{Limitations of previous approaches for achieving high cycle-to-edge ratios}\label{sec:lopa}

A natural approach to constructing a \network with high simple cycle density is to repeatedly add simple cycles on a randomly chosen subset of vertices. However, this process yields low cycle to edge ratios for sparse \networks. For example, a \network on $n$ vertices with average degree less than $\sqrt{n}$ built by randomly adding triangles will have a triangle-to-edge-ratio at most 2/3. (See \Cref{just add tri}.) In \cite{New09} Newman considers a similar approach which produces \networks with varied degree sequences and triangle-to-edge ratio strictly less than 1/3. However, it is not hard to construct \networks with arbitrarily high triangle ratio (growing with the size of the \network).

\begin{theorem} \label{just add tri} Let $G$ be a \network on $n$ vertices obtained by repeatedly adding triangles on sets of three randomly chosen vertices. If the average degree is less than $\sqrt{n}$, the expected ratio of triangles to edges is at most 2/3. \end{theorem}

\begin{proof} 
	Let $t$ be the number of triangles added and $d$ the average degree, so $d = 6t/n$. To ensure that $d < \sqrt{n}$, $t< n^{3/2}/6$. The total expected number of triangles in the \network is $t + (d/n)^3 {n \choose 3}= t +d^3/6= t+36t^3/n^3$. It follows that the expected ratio of triangles to edges is at most $$  \frac{t +36\left(\frac{t}{n}\right)^3}{3t}\leq  \frac{2}{3}.$$ 
\end{proof}

\begin{proof}(of \Cref{4bound})	
	Let $\sigma_1 \dots \sigma_{rank(M)}$ denote the eigenvalues of $M$. Let $G$ be a graph sampled from $M$. Let $C_k(G)$ be the number of simple $k$ cycles in $G$. Observe
	\begin{align*} 
	\E{C_k(G)}&=\sum_{i_1 \not=i_2 \dots \not=i_k} M_{i_1i_2}M_{i_2 i_3}\dots M_{i_k i_1}\\
	&\leq Tr(M^k)\\
	&= \sum_{i=1}^{rank(M)}\sigma_i^k\\
	&\leq rank(M) d^k.
	\end{align*}
\end{proof}

\section{Connectivity of the ROC model}\label{sec:con}
We describe the thresholds for connectivity for ROC graphs with one community type, $ROC(n,d,s,q)$. A vertex is isolated if it is has no adjacent edges. A community is isolated if it does not intersect any other communities. Here we use the abbreviation a.a.s. for asympotically almost surely. An event $A_n$ happens a.a.s. if  $\Pr{A_n} \to 1$ as $n\to \infty$. 

\begin{theorem}\label{no isolated vertices} Let $\ve>0$. For $  d\leq sqe^{sq}(1-\ve)$ and $d =\Omega\brac{sq\ln{n}}$, a \network from $ROC(n,d,s,q)$ a.a.s. has no isolated vertices. \end{theorem}

\begin{proof} We begin by computing the probability a vertex is isolated,
	\begin{align*} 
	\Pr{v \text{ is isolated}  }&=   \sum_{i=0}^{\frac{nd}{s^2q}} \Pr{v \text{ is in $i$ communities} } (1-q)^{si}\\
	&= \lo \sum_{i=1}^{\frac{nd}{s^2q}} { \frac{nd}{s^2q} \choose i} \bfrac{s}{n}^i \brac{1-\frac{s}{n}}^{\frac{nd}{s^2q} -i} e^{-sqi}\\
	&\leq \lo e^{-\frac{d}{sq}}\sum_{i=0}^{\frac{nd}{s^2q}}  \bfrac{de^{-sq+\frac{s}{n}}}{sq}^ i   \\
	&= \lo e^{-\frac{d}{sq}} \sum_{i=1}^{\frac{nd}{s^2q}}  \bfrac{de^{-sq}}{sq}^ i \\
	&=\lo \brac{ e^{-\frac{d}{sq}}}\bfrac{1}{1-\ve}.
	\end{align*}

	Let $X$ be a random variable that represents the number of isolated vertices of a \network drawn from $ROC(n,d,s,q)$. We compute $$\Pr{X>0} \leq \E{X} = \lo n \brac{ e^{-\frac{d}{sq}}}\bfrac{1}{1-\ve}
	=o(1).$$
\end{proof}

\begin{theorem} \label{no isolated communities} 
	A \network from  $ROC(n,d,s,q)$ with $s=o(\sqrt{n})$ has no isolated communities a.a.s. if 
	$$\frac{d}{q}> \log{\frac{nd}{s^2q}}.$$
\end{theorem}

\begin{proof} 
	We construct a ``community graph" and apply the classic result that $G(n,p)$ will a.a.s. have no isolated vertices when $p> (1+\eps)\log{n}/n$ for any $\eps>0$\cite{Erd59}. In the ``community graph" each vertex is a community and there is an edge between two communities if they share at least one vertex; a ROC \network has no isolated communities if and only if the corresponding community graph is connected. The probability two communities don't share a vertex is $\brac{1-\bfrac{s}{n}^{2}}^n$. Since communities are selected independently, the community \network is an instance of $G\brac{\frac{nd}{s(s-1)q}, \brac{1-  \bfrac{s}{n}^{2}}^n   }$. By the classic result, approximating the parameters by $\frac{nd}{s^2q}, 1-e^{s^2/n}$, this \network is connected when $$1-e^{-s^2/n}> \frac{\log{\frac{nd}{s^2q}}}{\frac{nd}{s^2q}}.$$ Since $s= o(\sqrt{n})$ is small, the left side of the inequality is approximately $s^2/n$, yielding the equivalent statement $$\frac{d}{q}>\log\frac{nd}{s^2q}.$$
\end{proof}

Note that the threshold for isolated vertices is higher, meaning that if a ROC \network a.a.s has no isolated vertices, then it a.a.s has no isolated communities. These two properties together imply the \network is connected.

\section{Approximating the hypercube's set of eigenvalues}

The ROC model captures the first $k$ moments of the eigenspectrum of the hypercube. To illustrate the difficultly of capturing the moments with a random model, here we present an approximation method that preserves the set of eigenvalues of the hypercube (without multiplicity) and show this is not enough to also capture the moments. This approximation is a $2^d \times 2^d$ matrix $M$ of rank $d$ with entries in the interval $[0,1]$ that has the same set of eigenvalues as the $d$-dimensional hypercube. A graph is produced by connecting vertices $v_i$ and $v_j$ with probability $M_{i,j}$. 

 The following well-known claim describes the eigenspectrum of the hypercube. 

\begin{proposition} Let $A_d$ be the adjacency matrix of the $d$-dimensional hypercube graph. Then for $i \in \{0,1, \dots d\}$,  $-d+2i$ is an eigenvalue of $A_d$ with multiplicity ${d \choose i}$.\end{proposition}

	We approximate the adjacency matrix $A_d$ of the $d$-dimensional hypercube graph by dividing the hypercube into layers. Layer $i$ consists of the $\ell_i={ d \choose i}$ vertices whose labels have precisely $i$ zeros. We let $p_i$ be the fraction of edges between layer $i$ and layer $i+1$. Each vertex in layer $i$ has $d-i$ neighbors in layer $i+1$, so $p_i=\frac{d-i}{\ell_{i+1}}$. We construct $M_d$, a $2^d \times 2^d$ matrix as follows. We partition the rows and columns into segments such that the $i^{th}$ segment has width $\ell_i={ d \choose i}$. These partitions induce a block structure on the matrix; the $ij$ block contains all of entries the matrix $M_d$ in which the row index is in the $i^{th}$ segment and the column index is in the $j^{th}$ segement. Each entry of $M_d$ on the block $ij$ is equal to the probability that two distinct randomly selected vertices from layer $i$ and layer $j$ are adjacent.

 \begin{figure}[h]
\centering
\includegraphics[scale=.5]{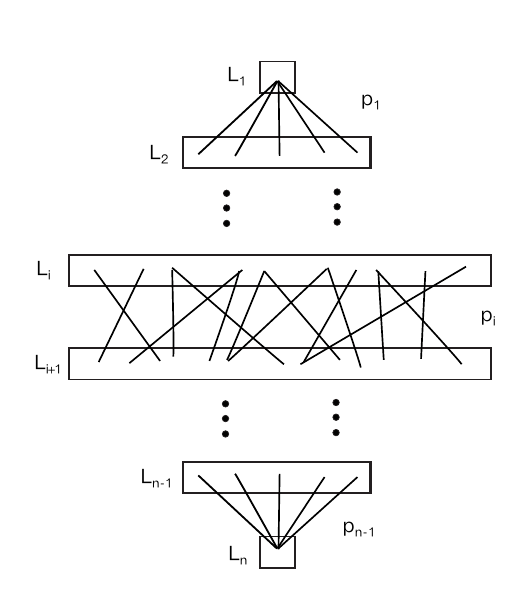}
\caption{ A graphical representation of a graph sampled from $M_d$.}
	\end{figure}
		
\begin{lemma}
The set of eigenvalues of $M_d$ is precisely the set of eigenvalues of the $d$-dimensional hypercube. 
\end{lemma}

\begin{proof}
Let $M^*_d$ be the $d\times d$ matrix obtained from $M_d$ by the following procedure. For each of the $d+1$ blocks, replace the $\ell_i$  rows corresponding to block $i$ with a single  $1 \times 2^d $ row equal to the sum of the $\ell_i$ rows. This yields a $d \times 2^d$ matrix. Delete all duplicate columns to obtain 
$$M^*_d=  \begin{bmatrix} 
		0 & 1 &  && \dots   & \\
		n&0 & 2& & &\\
		&n-1& 0 && &\\
		&&& \ddots&&\\
		\vdots&&&& 0& n\\
		&&&&1&0.
		\end{bmatrix}$$

First we claim that if $\lambda$ is an eigenvalue of $M^*_d$, then $\lambda$ is also an eigenvalue of $M_d$. Suppose $v=(v_0, \dots v_d)$ is an eigenvector of $M^*_d$ with eigenvalue $\lambda$. Let $v^\ast$ be the vector obtained by replacing each entry $v_i$ with $\ell_i$ entries with value $v_i / \ell_i$. Note $v^\ast$ is an eigenvector of $M_d$ with eigenvalue $\lambda$.

Next we show that the eigenvalues of $M^*_d$ are $\{-d, -d+2, \dots , d\}$ by induction on $d$.  	 
		 We use the following identities 
	$$ M^*_d= \begin{bmatrix}  &&&&&&0\\
	&&&  &&&0\\
	&&&  \mathbf{M^*_{d-1}}  &&&\vdots\\
	&&&&  &&0 \\
	&&&&&&n \\
	&0&0&\ldots&0&0&0 \end{bmatrix} \quad+\quad
	\begin{bmatrix} 0 &&&&&\\
	1 & 0&&&&\\
	& 1 & 0 &&&\\
	&& \ddots & \ddots && \\
	&&&&\\
	&&&&1&0 \end{bmatrix}
	 $$
	
		$$M^*_d= \begin{bmatrix}  0&0&0&\dots&0&0\\
	n&&&  &&\\
	0&&&  & &\\
	0&&&\mathbf{M^*_{d-1}}  &&  \\
	\vdots&&&&&\\
	0&&&&&\end{bmatrix} \quad + \quad
	\begin{bmatrix} 0&1 &&\\
	 &0& 1 &\\

	&& \ddots & \ddots && \\
	&&&&&\\
	&&&&0&1\\
	&&&&&0 \end{bmatrix}.
	 $$
	 
	First  note ${M^*_1}= \begin{bmatrix} 0& 1\\1& 0 \end{bmatrix}$ has eigenvalues $-1$ and $1$, establishing the base case. 
	Next we show that if $v$ is an eigenvector of $M^\ast_{d-1}$ with eigenvalue $\lambda$, then $\begin{pmatrix} 0 \\ v\end{pmatrix}- \begin{pmatrix} v \\ 0 \end{pmatrix}$ and $\begin{pmatrix} 0 \\ v\end{pmatrix}+  \begin{pmatrix} v \\ 0 \end{pmatrix}$ are eigenvectors of $M^*_d$ with eigenvalues $ \lambda-1$ and $\lambda +1$ respectively.  Apply the above identities we obtain $$M^*_d \left( \begin{pmatrix} 0 \\ v\end{pmatrix}- \begin{pmatrix} v \\ 0 \end{pmatrix} \right) = (\lambda -1) \begin{pmatrix} 0 \\ v\end{pmatrix} + \left(-\lambda +1 \right)\begin{pmatrix} v \\ 0 \end{pmatrix}= ( \lambda-1) \left( \begin{pmatrix} 0 \\ v\end{pmatrix}- \begin{pmatrix} v \\ 0 \end{pmatrix} \right)$$ 
	
	\begin{align*}M^*_d \left( \begin{pmatrix} 0 \\ v\end{pmatrix}+  \begin{pmatrix} v \\ 0 \end{pmatrix} \right) &= (\lambda +1) \begin{pmatrix} 0 \\ v\end{pmatrix} + \left(\lambda +1 \right)\begin{pmatrix} v \\ 0 \end{pmatrix}
	= ( \lambda+1) \left( \begin{pmatrix} 0 \\ v\end{pmatrix}+ \begin{pmatrix} v \\ 0 \end{pmatrix} \right).\end{align*}
	
	We have shown that $\{-d, -d+2, \dots d\}$ are eigenvalues of $M_d$. Since $M_d$ has rank $d+1$, this is precisely the set of eigenvalues of $M_d$.\end{proof}

Let $S$ be a graph on $n=2^d$ vertices sampled from $M_d$. Note that the expected average degree of $S$ is $d$, and so there approximately $2nd^2$ closed four walks in $S$ that trace trees. In the hypercube, there are $2nd^2$ such walks and an additional $nd^2$ closed four walks that trace simple cycles. However, the expected number of simple four cycles in $S$ is 	$$\E{ C_4(S)}= 8\sum_{i=0}^{n-2} \ell_i {\ell_{i+1} \choose 2} \ell_{i+2} p_i^2p_{i+1}^2 =\Theta\brac{d^5} =o\brac{nd^2}.$$ 
Therefore the expected number of four walks in $S$ is $(2+ o(1)) nd^2$, whereas it is $(3+o(1))nd^2$ for the hypercubes.

\section{ROC as a model for real-world graphs} \label{sec:rw}

\paragraph{Modeling the clustering coefficient of real-world graphs.} 
In \Cref{bounds}, we prove the average clustering coefficient of a ROC \network (with one community type) is approximately $sq^2/d$, meaning that tuning the parameters $s$ and $q$ with $d$ fixed yields wide range of clustering coefficients for a fixed density. Furthermore, \Cref{degree cc} describes the inverse relationship between degree and clustering coefficient in ROC \networks, a phenomena observed in protein-protein interaction \networks, the internet, and various social networks \cite{Ste05, Mah06, Mis07, Yon07}.

\paragraph{Diverse degree distributions and the DROC model.}\label{extension}
We also introduce an extension of our model which produces \networks that match a target degree distribution in expectation. The extension uses the Chung-Lu configuration model: given a degree sequence $d_1, \dots d_n$, an edge is added between each pair of vertices $v_i$ and $v_j$ with probability $\frac{d_i d_j}{\sum_{i=1}^n d_i}$, yielding a \network where the expected degree of vertex $v_i$ is $d_i$ \cite{Chu02}. In the  DROC model,  a modified Chung-Lu random \network is placed instead of an E-R random \network in each iteration. Instead of normalizing the probability an edge is selected in a community by the sum of the degrees in the community, the normalization constant is the expected sum of the degrees in the community.

\subsection{Approximating clustering coefficient}\label{sec:cc}

Closely related to the density of triangles is the clustering coefficient at a vertex $v$, the probability two randomly selected neighbors are adjacent: $$C(v)= \frac{|\{\{a,b\} : a, b \in N(v), a\sim b\}|}{deg(v)(deg(v)-1)/2}.$$ Equivalently the clustering coefficient is twice the ratio of the number of triangles containing $v$ to the degree of $v$ squared. 
Figure \ref{ccbar} illustrates the markedly high clustering coefficients of real-world \networks as compared with Erd\H{o}s-R\'enyi (E-R) \networks of the same density.  We show that the ROC model can be tuned to produce graphs with a variety of clustering coefficients at any density. The proofs in this section are quite technical and left to \Cref{ccproofs}.
\begin{figure}
	\centering
	\includegraphics[scale=.7]{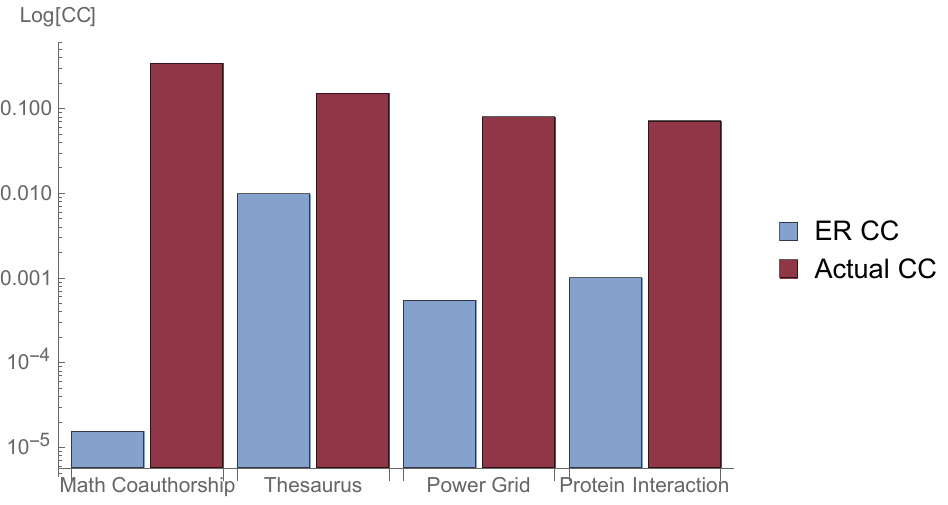}
	\caption{The clustering coefficient in real world \networks is much greater than that of an E-R random graph of the same density. Data from Table 3.1 of \cite{New03}.} \label{ccbar}
\end{figure}

\Cref{bounds} gives an approximation of the expected clustering coefficient when the degree and average number of communities per vertex grow with $n$. The exact statement is given in \Cref{exact}, and bounds in a more general setting are given by \Cref{more general cc}.

\begin{theorem} \label{bounds} Let $C(v)$ denote the clustering coefficient of a vertex $v$ with degree at least 2 in a \network drawn from $ROC(n,d,s,q)$ with $d=o(\sqrt{n})$, $d=\Omega(s)$, $d<s q e^{sq}$, $d= \omega( sq \log\frac{nd}{s})$, $s^2q=\omega(1)$, and $sq =o(d)$. Then 
	$$\E{C(v)}=\left(1+o(1) \right) \frac{s q^2}{d}.$$ 
\end{theorem}

Unlike in E-R \networks in which local clustering coefficient is independent of degree, higher degree vertices in ROC \networks have lower clustering coefficient. High degree vertices tend to be in more communities, and thus the probability two randomly selected neighbors are in the same community is lower.  Figure \ref{compare oc er} illustrates the relationship between degree and clustering coefficient, the degree distribution, and the clustering coefficient for two ROC \networks with different parameters and the E-R random \network of the same density.

\begin{theorem} \label{degree cc}
	Let $C(v)$ denote the clustering coefficient of a vertex $v$ in a \network drawn from $ROC(n,d,s,q)$ with $d=o(\sqrt{n})$, $s=\omega(1)$, $d \leq (s-1)qe^{(s-1)q}$, and $deg(v) \geq 2sq$. Then 
	$$\E{C(v) \given deg(v)=r}= \frac{sq^2}{r} \left(1 +o_r(1) \right) $$
\end{theorem}

\begin{figure}
	\centering
	\includegraphics[width=\textwidth]{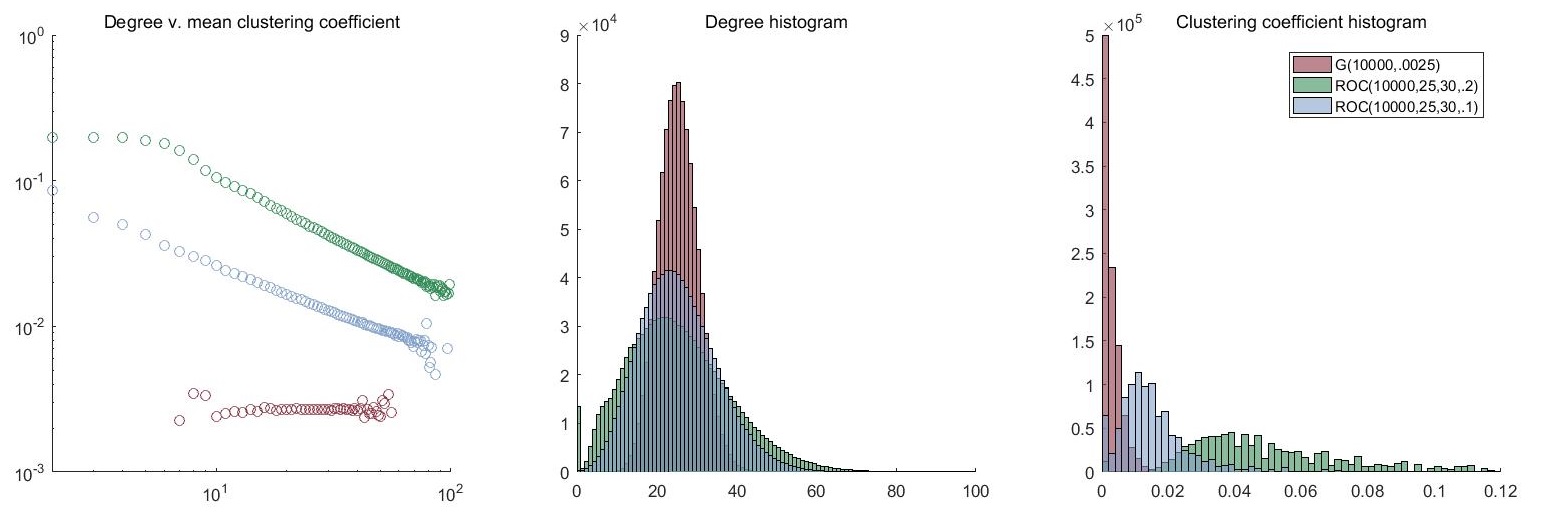}
	\caption{A comparison of the degree distributions and clustering coefficients of 100 \networks with average degree 25 drawn from each $G_{10000,0.0025}$, $ROC(10000,25,30,0.2)$, and $ROC(10000,25,30,0.1)$. The mean clustering coefficients are $0.00270$, $0.06266$, and $0.01595$ respectively.  \label{compare oc er}}
\end{figure}

\paragraph{Clustering coefficient proofs.}\label{ccproofs}

\begin{remark} \label{ignore}
	\Cref{bounds} gives bounds on the expected clustering coefficient up to factors of $(1 + o(1))$.  The clustering coefficient at a vertex is only well-defined if the vertex has degree at least two. Given the assumption in  \Cref{bounds} that $d= \omega(sq \log{\frac{nd}{s}})$, $d= \Omega(s)$, $d< sq e^{sq}$, and $s=\omega(1)$, \Cref{no d1} implies that the fraction of vertices of degree strictly less than two is $o(1)$. Therefore we ignore the contribution of these terms throughout the computations for \Cref{bounds} and supporting \Cref{exact}. In addition we divide by $deg(v)^2$ rather than by $deg(v)(deg(v)-1)$ in the computation of the clustering coefficient since this modification only affects the computations up to a factor of $(1 + o(1))$.
\end{remark}

\begin{lemma}\label{no d1} If $d= \omega( sq\log{\frac{nd}{s}})$, $d= \Omega(s)$, $s=\omega(1)$, $s=o(n)$, and $d<sq e^{sq}$, then a \network from $ROC(n,d,s,q)$ a.a.s. has no vertices of degree less than 2. \end{lemma}

\begin{proof}  
	The assumptions on $d$ imply the hypotheses of \Cref{no isolated vertices} hold. Thus, there are no isolated vertices a.a.s.  
	We begin by computing the probability a vertex has degree one. 
	\begin{align*} \Pr{deg(v)=1 }&=  \sum_{i=1}^{\frac{nd}{s^2q}} \Pr{v \text{ is in $i$ communities} } q (1-q)^{si-1}\\
	&=  \sum_{i=1}^{\frac{nd}{s^2q}} { \frac{nd}{s^2q} \choose i} \bfrac{s}{n}^i \brac{1-\frac{s}{n}}^{\frac{nd}{s^2q} -i} q (1-q)^{si-1}\\
	&\leq \lo \sum_{i=1}^{\frac{nd}{s^2q}}  \bfrac{nd}{s^2q}^ i \bfrac{s}{n}^i e^{-\frac{d}{sq}+\frac{si}{n}} q e^{-qsi+q}  \\
	&= \lo q e^{-\frac{d}{sq}} \sum_{i=1}^{\frac{nd}{s^2q}}  \bfrac{de^{-sq}}{sq}^ i \\
	&=O\brac{ \frac{de^{-sq-\frac{d}{sq}}}{s}}
	\end{align*}
	Let $X$ be a random variable that represents the number of degree one vertices of a \network drawn from $ROC(n,d,s,q)$. When $d= \omega( sq\log{\frac{nd}{s}})$, we obtain $$\Pr{X>0} \leq \E{X} = O\brac{ \frac{nde^{-sq-\frac{d}{sq}}}{s}}=o(1).$$
\end{proof}

\begin{lemma} \label{exact}
	Let $C(v)$ denote the clustering coefficient of a vertex $v$ of degree at least 2 in a \network drawn from $ROC(n,d,s,q)$ with $d=o( \sqrt{n})$ and $d= \omega(sq \log{\frac{nd}{s}})$. Then $$\E{C(v)}=\left(1 + o(1) \right) \left( \sum_{i=1}^{\frac{nd}{s^2q}} { \frac{nd}{s^2q} \choose i} \left( \frac{s}{n}\right)^i \left(1- \frac{s}{n}\right)^{\frac{nd}{s^2q}-i}\frac{s(s-1)q^3k }{\left( sqk + 2-2q\right)^2} \right). $$ \label{cc}  
\end{lemma}

\begin{proof} 
	For ease of notation, we ignore factors of $(1+ o(1))$ throughout as described in \Cref{ignore}.
	First we compute the expected clustering coefficient of a vertex from an $ROC(n,d,s,q)$ \network given $v$ is contained in precisely $k$ communities. Let $X_1, \dots X_k$ be random variables representing the degree of $v$ in each of the communities, $X_i \sim Bin(s,q)$. We have \begin{align} \E{C(v)| \text{ $v$ in $k$ communities }}&=\E{ \frac{\sum_{i=1}^kX_i(X_i-1) q}{\left( \sum_{i=1}^k X_i\right)^2}} \label{given k} \\
	&= qk \E{ \frac{X_1(X_1-1) }{\left( sq(k-1) + X_1\right)^2}} \nonumber \\
	&= qk \E{ \frac{X_1^2 }{\left( sq(k-1) + X_1\right)^2}}- qk \E{ \frac{X_1 }{\left( sq(k-1) + X_1\right)^2}} \nonumber.
	\end{align}
	Write $X_1= \sum_{i=1}^s y_i$ where $y_i \sim Bernoulli(q)$. Using linearity of expectation and the independence of the $y_i's$ we have
	\begin{align*} \E{ \frac{X_1 }{\left( sq(k-1) + X_1\right)^2}}&= s \E{ \frac{y_1 }{\left( sq(k-1) + (s-1)q+y_1\right)^2}}
	&= \frac{sq }{\left( sq(k-1) + (s-1)q+1\right)^2},\\
	\end{align*}
	and 
	\begin{align*} \E{ \frac{X_1^2 }{\left( sq(k-1) + X_1\right)^2}}&= \E{ \frac{\left(\sum_{i=1}^s y_i\right)^2 }{\left( sq(k-1) +\sum_{i=1}^s y_i\right)^2}}\\
	&= s\E{ \frac{ y_1^2 }{\left( sq(k-1) +q(s-1) +y_1\right)^2}}\\
	&\quad +s(s-1)\E{ \frac{\left(y_1y_2\right)^2 }{\left( sq(k-1) + (s-2)q+y_1+y_2\right)^2}}\\
	&= \frac{ sq }{\left( sq(k-1) +q(s-1) +1\right)^2}+\frac{s(s-1)q^2 }{\left( sq(k-1) + (s-2)q+2\right)^2}.
	\end{align*}
	Substituting in these values into \Cref{given k}, we obtain
	\begin{equation}
	\E{C(v)|v \in k \text{ communities }}=qk\left(\frac{s(s-1)q^2 }{\left( sq(k-1) + (s-2)q+2\right)^2}\right)=\frac{s(s-1)q^3k }{\left( sqk + 2-2q\right)^2}.\label{conditional}
	\end{equation}
	
	Let $M$ be the number of communities a vertex is in, so $M  \sim Bin\left(\frac{nd}{s^2q},\frac{s}{n}\right).$ It follows 
	\begin{align*}
	\E{C(v)}&=\sum_{i=1}^{\frac{nd}{s^2q}}\Pr{\text{ $v$ in $k$ communities }} \E{C(v)|\text{ $v$ in $k$ communities }}\\
	&=\sum_{i=1}^{\frac{nd}{s^2q}} { \frac{nd}{s^2q} \choose i} \left( \frac{s}{n}\right)^i \left(1- \frac{s}{n}\right)^{\frac{nd}{s^2q}-i}\frac{s(s-1)q^3k }{\left( sqk + 2-2q\right)^2}.\\
	\end{align*}
\end{proof}

The proof of \Cref{bounds}, relies on the follow two lemmas
regarding expectation of binomial random variables. 

\begin{lemma} \label{ex of rec p1} 
	Let $X \sim Bin(n,p)$. Then 
	\begin{enumerate}
		\item $\E{ \frac{1}{X+1} \given X \geq 1 }= \frac{1- \left( 1-p\right)^{n+1}-(n+1)p(1-p)^n}{p(n+1)}$ and
		\item $\E{ \frac{1}{X+1}}= \frac{1- \left( 1-p\right)^{n+1}}{p(n+1)}$.
	\end{enumerate}
\end{lemma}

\begin{proof} 
	Observe 
	\begin{align*} 
	\E{ \frac{1}{X+1} \given X\geq 1} &= \sum_{i=1}^n {n \choose i}  \frac{p^i (1-p)^{n-i}}{i+1}\\
	&= \frac{1}{p(n+1)} \sum_{i=1}^n { n+1 \choose i+1} p^{i+1} (1-p)^{n-i}\\
	&= \frac{1- \left( 1-p\right)^{n+1}-(n+1)p(1-p)^n}{p(n+1)}.
	\end{align*}
	Similarly 
	\begin{align*} 
	\E{ \frac{1}{X+1}} &= \sum_{i=0}^n {n \choose i}  \frac{p^i (1-p)^{n-i}}{i+1}
	= \frac{1}{p(n+1)} \sum_{i=0}^n { n+1 \choose i+1} p^{i+1} (1-p)^{n-i}
	= \frac{1- \left( 1-p\right)^{n+1}}{p(n+1)}.
	\end{align*}
\end{proof}

\begin{lemma}\label{ex of rec}
	Let $X \sim Bin(n,p)$. Then $$\E{ \frac{1}{X} \given X \geq 1} \leq \frac{1}{p(n+1)}\left(1 +\frac{3}{p(n+2)} \right).$$
\end{lemma}

\begin{proof} Note that when $X \geq 1$, \begin{equation} \label{true} \frac{1}{X} \leq \frac{1}{X+1}+\frac{3}{(X+1)(X+2)}.\end{equation} By Lemma \ref{ex of rec p1}, 
	\begin{equation*} \E{\frac{1}{X+1}\given X \geq 1} \leq \frac{1}{p(n+1)}.
	\end{equation*}
	We compute 
	\begin{align*}
	\E{\frac{1}{(X+1)(X+2)}\given X\geq 1}&= \sum_{i=1}^n \frac{{n \choose i} p^i (1-p)^{n-i}}{(i+1)(i+2)}\\
	&= \frac{1}{p^2 (n+2)(n+1)} \sum_{i=1}^n {n +2 \choose i+2} p^{i+2}(1-p)^{n-i}\\
	&\leq \frac{1}{p^2 (n+2)(n+1)}.
	\end{align*}
	Taking expectation of \Cref{true} gives $$\E{\frac{1}{X}\given X\geq 1} \leq \frac{1}{p(n+1)}\left(1 +\frac{3}{p(n+2)} \right).$$

\end{proof}

\begin{proof} (of \Cref{bounds}.) For ease of notation, we ignore factors of $(1 + o(1))$, as described in \Cref{ignore}.
	It follows from \Cref{conditional} in the proof of \Cref{exact} that
	$$\frac{q}{k+1}\leq \E{C(v)|v \in k \text{ communities }}\leq \frac{q}{k}, $$ where the left inequality holds when $q(s-1)\geq 5$.
	
	We now compute upper and lower bounds on $\E{C(v)}$, assuming $v$ is in some community. Let $M$ be the random variable indicating the number of communities containing $v$, $M \sim Bin\left(\frac{nd}{s(s-1) q}, \frac{s}{n}\right)$. It follows
	\begin{align*}\E{C(v)}=\sum_{k=1}^{\frac{nd}{s^2q}}\Pr{M=k }\E{C(v)|M=k} \end{align*}
	\begin{align*}q\E{\frac{1}{M+1} \given M \geq 1} \leq \E{C(v)}\leq  q\E{\frac{1}{M}\given M \geq 1}. \end{align*}
	Applying Lemmas \ref{ex of rec p1} and \ref{ex of rec} to the lower and upper bounds respectively, we obtain 
	$$\frac{q\left(1-\left(1-\frac{s}{n}\right)^{\frac{nd}{s(s-1)q}+1}- \brac{\frac{nd}{s(s-1)q}+1}\left(1-\frac{s}{n}\right)^{\frac{nd}{s(s-1)q}}\right)}{\frac{d}{(s-1)q}+\frac{s}{n}}\leq\E{C(v)}\leq \frac{q}{\frac{d}{(s-1)q} +\frac{s}{n}} \left( 1 + \frac{3}{\frac{d}{(s-1)q} +\frac{2s}{n}} \right) $$
	which for $s=o(n)$ simplifies to
	\begin{equation} \label{more general cc}
	\lo \frac{(s-1)q^2}{d} \left( 1-\frac{nd}{s (s-1) q}e^{-d/((s-1)q)} \right) \leq \E{C(v) } \leq  \frac{(s-1)q^2}{d} \left( 1+ \frac{(s-1)q}{d}\right) \lo.
	\end{equation}
	Under the assumptions $s^2q =\omega(1)$ and $sq = o(d)$, we obtain our desired result $$\E{C(v)}=\left(1+o(1) \right)\left( \frac{sq^2}{d}\right).$$
\end{proof}

The following lemma will be used in the proof of \Cref{degree cc}.

\begin{lemma} \label{helper} Let $a \ge 1, r \ge a+1$.
	Let $X$ be a nonnegative integer drawn from the discrete distribution with probability proportional to $f(x)=x^{r-x}e^{-ax}$. Let $z=\argmax f(x)$. Then for $t\geq 1$,  \[
	\Pr{|x-z| \ge t\sqrt{z}} \le e^{-t+1}.
	\]
\end{lemma}

\begin{proof}
	We consider $f$ as a function over the positive reals (rather than over the positive integers) and observe that $f$ is logconcave:
	\[
	\frac{d^2}{dx^2}\ln f(x) = \frac{d}{dx}(-a + \frac{r}{x}-1 - \ln x) = -\frac{r}{x^2}-\frac{1}{x},
	\]
	which is nonpositive for all $x > 0$. We will next bound the standard deviation of this density, so that we can use an exponential tail bound for logconcave densities.
	 Setting its derivative to zero, we see that at the maximum, we have 
	\begin{equation}\label{eq:max-cond}
	a+1= \frac{r}{x}-\ln x.
	\end{equation}  
In other words the maximum $z$ satisfies: $(a+1)z+z\ln z = r$.
Next we claim that $z \le r/(a+1)$. To see this, note that at $x = r/(a+1)$ we have $(a+1)x + x\ln x \ge r$ by the assumption that $r \ge a+1$; moreover, the derivative of $(a+1)x + x\ln x$ is positive for any $x \ge r/(a+1)$. 

	Observe  that when $\delta  \le r-z$,
	\begin{align*}
	\frac{f(z+\delta)}{f(z)} = 	\frac{(z+\delta)^{r-z-\delta}e^{-az-a\delta}}{z^{r-z}e^{-az}} 
	&=\left(1+\frac{\delta}{z}\right)^{r-z-\delta}z^{-\delta}e^{-a\delta}  \\
	&\leq e^{\delta(\frac{r}{z}-1-a-\ln z)}e^{-\frac{\delta^2}{z}}\\
	& = e^{-\frac{\delta^2}{z}},
	\end{align*}
	where in the final equality we used the optimality condition (\ref{eq:max-cond}). 
	
	By logconcavity (which says that for any $x,y$ and any $\lambda \in [0,1]$, we have $f(\lambda x +(1-\lambda)y)\ge f(x)^\lambda f(y)^{1-\lambda}$) we have 
	$$f(x+ \delta)= f\brac{ \brac{ 1-\frac{1}{t} } x + \frac{1}{t} ( x +t\delta) } \geq f(x)^{1-1/t} f(x+t\delta)^{1/t}$$
	for any $t \geq 1$. 
	It follows that
	\begin{equation} \label{f(x)}
	f(z+t\delta) \le f(z)e^{\frac{-t\delta^2}{z}}
	\end{equation}
	for all $|t|\geq 1$ (since we can apply the same argument for $z-\delta$). 
Taking $\delta =  ar/(a+1) \leq r -z$, we obtain
\[
e^{-t\frac{\delta^2}{z}} = e^{-t\frac{(ar)^2}{(a+1)^2 z}} \le e^{-t\frac{a^2r}{a+1}} \le e^{-t\frac{r}{2}}. 
\]
	
	Using  the observation $\sum_{x \in \Z^{+} } f(x) \geq f(z)$,  it follows that $$\Pr{ x \geq z+ t \sqrt{z}}\leq  e^{-t} \quad \text{ and } \quad \Pr{ x \leq  z- t \sqrt{z}}\leq e^{-t},$$  and so $$\Pr{|x-z| \ge t\sqrt{z}} \le 2 e^{-t}\le e^{-t+1}.$$

\end{proof}

\begin{proof} (of \Cref{degree cc}). 
	Let $M$ denote the number of communities a vertex $v$ is selected to participate in. We can write
	\begin{align*}
	\E{C(v)|deg(v)=r} &= \sum_{k=\frac{r}{s}}^r \E{C(v)|deg(v)=r, M=k}\Pr{M=k|deg(v)=r}\\
	&=\sum_{k=\frac{r}{s}}^r \E{C(v)|deg(v)=r, M=k}\Pr{deg(v)=r|M=k}\frac{\Pr{M=k}}{\Pr{deg(v)=r}}.
	\end{align*}
	First we compute the expected clustering coefficient of a degree $r$ vertex given that it is $k$ communities: 
	\begin{align*}
	\E{C(v)|deg(v)=r \text{ and  }M=k}=\frac{ \sum_{i \not= j, i,j \in N(v)} q\left(\Pr{i, j \text{ part of same community} } \right)}{deg(v)\left(deg(v)-1\right)} =\frac{q}{k}.
	\end{align*}
	Next we note that $M$ is a drawn from a binomial distribution, and the degree of $v$ is drawn from a sum of $k$ binomials, each being $Bin(s,q)$.	Therefore,
	\begin{align*}
	\Pr{M=k}\Pr{deg(v)=r|M=k} &= {\frac{nd}{s(s-1)q} \choose k} \left(\frac{s}{n}\right)^k\left(1-\frac{s}{n}\right)^{\frac{nd}{s(s-1)q}-k} {sk \choose r} q^r(1-q)^{sk-r}.
	\end{align*}
	Using this we obtain
	\begin{align}
	\E{C(v)|deg(v)=r} &= \frac{\sum_{k=\frac{r}{s}}^r \frac{q}{k}\Pr{M=k}\Pr{deg(v)=r|M=k}}{\sum_{k=\frac{r}{s}}^r\Pr{M=k}\Pr{deg(v)=r|M=k}}\nonumber \\
	&=\lo q \frac{\sum_{k=\frac{r}{s}}^r \frac{1}{k}\cdot \left(\frac{d}{(s-1)qk}\right)^ke^{-\frac{d}{(s-1)q}+\frac{sk}{n}} \left(\frac{skq}{r}\right)^re^{-qsk+qr}}{\sum_{k=\frac{r}{s}}^r\left(\frac{d}{(s-1)qk}\right)^ke^{-\frac{d}{(s-1)q}+\frac{sk}{n}} \left(\frac{skq}{r}\right)^re^{-qsk+qr}}\nonumber\\
	&=\lo 	q \frac{\sum_{k=\frac{r}{s}}^r \frac{1}{k}\cdot \left(\frac{d}{(s-1)q}\right)^k k^{r-k}e^{-qsk}}{\sum_{k=\frac{r}{s}}^r
		\left(\frac{d}{(s-1)q}\right)^k k^{r-k}e^{-qsk}}. \label{expression}
	\end{align}		
	Writing $a=qs-\ln(d/(s-1)q)$, this is 
	\[
	q \frac{\sum_{k=\frac{r}{s}}^r  \frac{1}{k}\cdot k^{r-k}e^{-ak}}{\sum_{k=\frac{r}{s}}^r k^{r-k}e^{-ak}}.
	\]
	Therefore \Cref{expression} is the same as $q\E{1/x}$ when $x$ is a nonnegative integer drawn from the discrete distribution with density proportional to $f(x)=x^{r-x}e^{-ax}$. 
	We approximate the argmax of $f$ with $z \approx \frac{r}{sq}$ and
	use \Cref{helper} 
	to bound
	\begin{align*}
	\E{\left| \frac{1}{x}-\frac{1}{z} \right|} 
	&\leq  \sum_{t=1}^\infty  \left(\frac{1}{z}-\frac{1}{z+ t \sqrt{z} }\right)e^{-t} +\sum_{t=1}^{\sqrt{z}-1} \left(\frac{1}{z- t \sqrt{z} }- \frac{1}{z}\right)e^{-t}  \\
	&=  \sum_{t=1}^\infty \frac{t\sqrt{z}e^{-t}}{z(z+t\sqrt{z})}+\sum_{t=1}^{\sqrt{z}-1} \frac{t\sqrt{z}e^{-t}}{z(z-t\sqrt{z})}\\
	&\leq  \frac{1}{z} \sum_{t=1}^\infty \frac{te^{-t}}{\sqrt{z}+1}+\frac{\sqrt{z}}{z} \brac{ \sum_{t=1}^{\sqrt{z}/3} \frac{3te^{-t}}{2z} +\sum_{t=\sqrt{z}/3}^{\sqrt{z}-1} te^{-t}  }\\
	&=  \frac{O(1)}{z\sqrt{z} } + \frac{O(1)}{ z\sqrt{z}} + O\brac{ \frac{\sqrt{z}}{3} e^{-\frac{\sqrt{z}}{3} } } = \frac{O(1)}{z \sqrt{z}}.
	\end{align*}
	Using this and approximating $z$ by $\frac{r}{sq}$, the expectation of $x$ with respect to the density proportional to $f$ can be estimated:
	\[
	q \E{\frac{1}{x}}=\frac{q}{z}\left(1+O\left(\frac{1}{\sqrt{z}}\right)\right)=\lo\frac{sq^2}{r}\left(1+O\left(\sqrt{\frac{sq}{r}}\right)\right)=(1+o_r(1))\frac{sq^2}{r}
	\]
	as claimed.

	
\end{proof}

\subsection{Varied degree distributions: the DROC extension}\label{sec:droc}

In this section we introduce an extension of our model which produces \networks that match a target degree distribution in expectation. In each iteration a modified Chung-Lu random \network is placed instead of an E-R random \network.  \bigskip

\begin{centering}
	\fbox{\parbox{0.95\textwidth}{
			{\bf DROC($n,D,s,q$)}.  
			
			\textit{Input: } number of vertices $n$, target degree sequence $D= t(v_1), \dots t(v_n)$ with mean $d$.
			
			\textit{Output: }a \network on $n$ vertices where vertex $v_i$ has expected degree $t(v_i)$.\\
			
			Repeat $n/((s-1)q)$ times:
			\begin{enumerate}
				\item Pick a random subset $S$ of vertices (from $\{1,2,\ldots,n\}$) by selecting each vertex with probability $s/n$.
				\item Add a modified C-L random \network on $S$, i.e., for each pair in $S$, add the edge between them independently with probability 
				$ \frac{qt(v_i) t(v_j)}{sd} $; if the edge already exists, do nothing.   
			\end{enumerate}
		}
	}
\end{centering}

\begin{theorem}\label{E1} Given a degree distribution $D$ with mean $d$ and $\max_i t(v_i)^2 \leq \frac{sd}{q}$,  DROC($n,D,s,q$) yields a \network where vertex $v_i$ has expected degree $t(v_i)$. \end{theorem}
\noindent We require  $\max_i t(v_i)^2 \leq \frac{sd}{q}$ to ensure that the probability each edge is chosen is at most 1. In the DROC model the number of communities a vertex belongs to is independent of target degree $t(v)$. When $t(v) > \frac{sd}{q}$, if $v$ participates in the average number of communities and is connected to all vertices in each of its communities, it likely will not reach degree $t(v)$. Therefore when $s$ is low and $q$ is high, the DROC model is less able to capture degree distributions with long upper tails. Moreover, when $s$ is low and $q$ is high, there will be more isolated vertices in a DROC \network since the expected fraction of isolated vertices is at least $(1-s/n)^{n/(q(s-1))}$. In \Cref{ex cc} we show that when $s$ is low and $q$ is high the clustering coefficient is largest. In this regard the DROC model is somewhat limited; it may not be possible to achieve some very high clustering coefficients while simultaneously capturing the upper tail of the degree distribution and avoiding isolated vertices.


The following corollary shows that it is possible to achieve a power law degree distribution with the DROC model for power law parameter $\gamma>2$. We use $\zeta(\gamma)=\sum_{n=1}^\infty n^{-\gamma}$ to denote the Riemann zeta function. 
\begin{corollary} \label{power law}
	Let $D\sim \mathcal{D_\gamma}$ be the power law degree distribution defined as follows: $$ \Pr{ t(v_i)=k}= \frac{k^{-\gamma}}{\zeta(\gamma)},$$ for all $1 \leq i \leq n$. If $\gamma>2$ and 
	$$\frac{s}{q} =\omega(1)\frac{\zeta(\gamma)}{\zeta(\gamma-1)} n^{\frac{1}{\gamma-1}},$$ then with high probability $D$ satisfies the conditions of \Cref{E1}, and therefore can be used to produce a DROC \network. 
\end{corollary}
Taking the distribution $D_d$ with $t(v)=d$ for all $v$ in the DROC model does not yield $ROC(n,d,s,q)$. The model $DROC(n, D_d, s,q)$ is equivalent to $ROC(n,d,s, \frac{qd}{s})$.

By varying $s$ and $q$ we can control the clustering coefficient of a $DROC$ graph. 
\begin{theorem}\label{ex cc} Let $C(v)$ denote the clustering coefficient of a vertex $v$  in \network drawn from $DROC(n,D,s,q)$ with $\max t(v_i)^2 \leq \frac{sd}{q}$, $s= \omega(1)$, $s/n=o(q)$, and $t=t(v)$. Then  $$\E{C(v)} = \lo  \frac{\brac{\sum_{u \in V} t(u)^2}^2}{d^3n^2s} \brac{ (1- e^{-t})^2q^2 +c_t q^3},$$ where $c_t \in [0,6.2)$ is a constant depending on $t$.  \end{theorem}

\noindent Equation \cref{cc droc} in the proof of the theorem gives a precise statement of the expected clustering coefficient conditioned on community membership.



\paragraph{DROC proofs.}\label{ex proofs}

\begin{proof} (of \Cref{ex cc}.) Let $v$ be a vertex with target degree $t= t(v)$, and let $k$ denote the number communities containing $v$.  First we claim $deg(v)\sim Bin \left((s-1)k, \frac{t q}{s} \right)$.  Let $s$ be an arbitrary vertex of a community $S$ containing $v$. 
	$$\Pr{ s \sim v \text{ in } S} = \sum_{u \in V} \Pr{s=u} \Pr{v \sim u \text{ in } S} = \sum_{u \in V} \frac{1}{n} \frac{ t(u) t q}{ds}= \frac{tq}{s}.$$
	A vertex in $k$ communities has the potential to be adjacent to $(s-1)k$ other vertices, and each adjacency occurs with probability $t q /s$.
	
	Next, let $N_u$ be the event that a randomly selected neighbor of vertex $v$ is vertex $u$. We compute
	\begin{align}
	\Pr{N_u}&= \sum_{r} \frac{\Pr{ u \sim v \given deg(v)=r}\Pr{deg(v)=r}}{r}\nonumber \\
	&=\sum_r \frac{ \Pr{u \sim v} \Pr{deg(v)=r \given u \sim v}}{r}\nonumber\\
	&=\Pr{u \sim v} \E { \frac{1}{deg(v)} \given u \sim v}\nonumber\\
	&= \lo  \bfrac{s}{n}^2\frac{n}{(s-1)q} \frac{ t(u)t q}{sd} \left(\frac{1-e^{-tqk}}{tk q}\right)\label{deg ex} \\
	&= \lo  \frac{t(u)\brac{1-e^{-tqk}}}{qkdn}\nonumber.
	\end{align}
	To see \Cref{deg ex}, note that by the first claim $\E { \frac{1}{deg(v)} \given u \sim v}= \E{\frac{1}{X+1}}$ where $X \sim Bin\left((s-1)k -1, \frac{tq }{s}\right)$. Applying \Cref{ex of rec p1} and assuming $s=\omega(1)$, we obtain 
	$$\E { \frac{1}{deg(v)} \given u \sim v}=\frac{1-(1-\frac{tq }{s})^{(s-1)k }}{((s-1)k )\frac{tq }{s}}= \lo \frac{1-e^{-tqk}}{tk q}.$$
	
	Now we compute the expected clustering coefficient conditioned on the number of communities the vertex is part of under the assumption that $s/n=o(q)$. Observe \begin{align}
	\E {C(v) \given \text{$v$ in $k$ communities}}&=\sum_{u,w} N_u N_w\Pr{ u \sim w \given u \sim v \text{ and } w \sim v }\nonumber \\
	&=\sum_{u,w} \frac{t(u)t(w)\brac{1-e^{-tqk}}^2}{(qkdn)^2}\brac{ \frac{1}{k} + \bfrac{s}{n}^2\frac{n}{(s-1)q} } \frac{t(u)t(w)q}{sd}\nonumber \\
	&= \lo \frac{\brac{1-e^{-tqk}}^2 \brac{\sum_{u \in V} t(u)^2}^2}{qd^3k^3n^2s}. \label{cc droc}
	\end{align}

	Next compute the expected clustering coefficient without conditioning on the number of communities. To do so we need to compute the expected value of the function $f(k)=\frac{(1- e^{-kqt})^2}{k^3}$. We first use Taylor's theorem to give bounds on $f(k)$. For all $k$, there exists some $z\in [1/q, k]$ such that
	\begin{equation*}
	f(k)= f\bfrac{1}{q} + f'\bfrac{1}{q} \brac{k-\frac{1}{q}} + \frac{ f''(z)}{2}\brac{k-\frac{1}{q}}^2.
	\end{equation*}
	Note that for $z \in [1/q, k]$ 
	\begin{align*} f''(z)&= \frac{12 (1 - e^{-k q t})^2}{k^5} - \frac{
		12 e^{-k q t}(1 - e^{-k q t}) q t}{k^4} + \frac{
		2 e^{-2 k q t} q^2 t^2}{k^3} - \frac{
		2 e^{-k q t} (1 - e^{-k q t}) q^2 t^2}{k^3}\\
	&\leq \frac{12 (1 - e^{-k q t})^2}{k^5} + \frac{
		2 e^{-2 k q t} q^2 t^2}{k^3}\\
	&\leq q^5\brac{ 12+ 2t^2e^{-2t}},  
	\end{align*}
	and $$f''(z) \geq 0.$$
	It follows that 
	\begin{equation} \label{taylor} 
	f\bfrac{1}{q} + f'\bfrac{1}{q} \brac{k-\frac{1}{q}}  \leq f(k)\leq  f\bfrac{1}{q} + f'\bfrac{1}{q} \brac{k-\frac{1}{q}} + q^5\brac{ 6+ t^2e^{-2t}}\brac{k-\frac{1}{q}}^2.
	\end{equation}

	Let $M \sim Bin(n/(sq), s/n)$ be the random variable for the number of communities a vertex $v$ is part of. (Since $s= \omega(1)$ replacing the number of communities by $n/(sq)$ changes the result by a factor of $\lo$.) We use \Cref{taylor} to give bounds on the expectation of $f(M)$, 
	
	\begin{align*}
	\E{f(M)} &\leq  \E{ f\bfrac{1}{q} + f'\bfrac{1}{q} \brac{M-\frac{1}{q}} + q^5\brac{ 12+ 2t^2e^{-2t}}\brac{M-\frac{1}{q}}^2}\\
	&=(1- e^{-t})^2q^3 + \frac{1}{q}\brac{1-\frac{s}{n}} q^5 \brac{ 6+ t^2e^{-2t}}\\
	&\leq (1- e^{-t})^2q^3 + q^4 \brac{ 6+ t^2e^{-2t}}
	\end{align*}
	and 
	\begin{align*}
	\E{f(M)} \geq  \E{ f\bfrac{1}{q} + f'\bfrac{1}{q} \brac{M-\frac{1}{q}} }
	&=(1- e^{-t})^2q^3. \end{align*} 
	Therefore $ \E{f(M)}  = (1- e^{-t})^2q^3 + c_t q^4 $ for some constant $c_t \in [0, 6.2)$.
	
	Finally, we compute
	\begin{align*} 
	\E {C(v)}&= \sum_k \Pr{M=k}\frac{\brac{1-e^{-tqk}}^2 \brac{\sum_{u \in V} t(u)^2}^2}{qd^3k^3n^2s}\\
	&= \frac{\brac{\sum_{u \in V} t(u)^2}^2}{qd^3n^2s} \E {f(M)}\\
	&= \lo \frac{\brac{\sum_{u \in V} t(u)^2}^2}{d^3n^2s} \brac{ (1- e^{-t})^2q^2 +c_tq^3}.
	\end{align*}

	%
\end{proof}

\begin{proof} (of \Cref{power law}.) Let $d= mean(D)$. We compute $$ \E {d} = \sum_{k=1}^\infty \frac{k^{-\gamma+1}}{\zeta(\gamma)} = \frac{\zeta(\gamma-1)}{\zeta(\gamma)}.$$
	
	Next we claim that with high probability the maximum target degree of a vertex is at most $t_0=n^{2/(\gamma-1)}$. Let $X$ be the random variable for the number of indices $i$ with $t(v_i)>k_0$. 
	
	\begin{align*}
	\Pr{  \max_i t(v_i) > t_0 } &\leq \E {X }= n \Pr{ t(v_1) > t_0} \leq n \sum_{i=t_0+1}^{\infty} \frac{i^{-\gamma}}{\zeta(\gamma)}\\
	&\leq n\int_{i=t_0}^{\infty}  \frac{i^{-\gamma}}{\zeta(\gamma)}= \bfrac{1}{\zeta(\gamma)(\gamma-1)} n t_0^{1-\gamma}=o(1).
	\end{align*}
	It follows that $ \max_i t(v_i)^2 \leq n^{\frac{1}{\gamma-1}}$, and so $\max_i t(v_i)^2 \leq \frac{sd}{q}$.
\end{proof}

\subsection{Comparison of ROC and DROC to other random graph models} \label{other models}
The ROC model captures any pair of triangle-to-edge and four-cycle-to edge ratios simultaneously, and the DROC model can exhibit a wide range of degree distributions with high clustering coefficient. Previous work \cite{Hol02}, \cite{Ost13}, and \cite{Rav02} provides models that produce power law \networks with high clustering coefficients. Their results are limited in that the resulting \networks are restricted to a limited range of power-law parameters, and are either deterministic or only analyzable empirically. In contrast, the DROC model is a fully random model designed for a variety of degree distributions (including power law with parameter $\gamma >2$) and can  provably produce \networks with a range of clustering coefficient. The algorithm presented in \cite{Vol04} produces \networks with tunable degree distribution and clustering, but  unlike ROC \networks, there is no underlying community structure and the resultant \networks do not exhibit the commonly observed inverse relationship between degree and clustering coefficient. 

The Block Two-Level Erd\H{o}s and R\'enyi (BTER) model produces graphs with scale-free degree distributions and random dense communities \cite{Ses12}. However, the communities in the BTER model do not overlap; all vertices are in precisely one E-R community and all other edges are added during a subsequent configuration model phase of construction. Moreover, in the BTER model community membership is determined by degree, which ensures that all vertices in a BTER community have similar degree. In contrast, the degree distribution within a DROC community is a random sample of the entire degree distribution.

\subsection{ROC as a model for real-world \networks.}\label{sec:rw dis}
Modeling a \network as the union of relatively dense communities has explanatory value for many real-world settings, in particular for social and biological networks. Social networks can naturally be thought of as the union of communities where each community represents a shared interest or experience  (e.g. school, work, or a particular hobby); the conceptualization of social networks as overlapping communities has been studied  in  \cite{Pal07}, \cite{Xie11}.
Protein-protein interaction networks can also be modeled by overlapping communities, each representing a group of proteins that interact with each other in order to perform a specific cellular process. Analyses of such networks show  proteins are involved in multiple cellular processes, and therefore overlapping communities define the structure of the underlying \network  \cite{Ahn10}, \cite{Kro06}, \cite{Bad03}.

Our model therefore may be a useful tool for approximating large \networks.  It is often not possible to test algorithms on \networks with billions of vertices (such as the brain, social \networks, and the internet). Instead, one could use the DROC model to generate a smaller \network with same clustering coefficient and degree distribution as the large \network, and then optimize the algorithm in this testable setting.  Further study of such a small \network approximation could provide insight into the structure of the large \network of interest. 

Moreover, the ROC model could be used as a null hypothesis for testing properties of a real-world networks known to have community structure. It is established practice to compare real-world \networks to various random \network models to understand the non-random aspects of its structure (\cite{Cir08, Son05, New05, New01}). The ROC model is particularly well-suited to be the null hypothesis \network for \networks with known community structure.  
Comparing such a network to a ROC network would differentiate between properties of the network that are artifacts of community structure and those that are unique to the \network.

\section{Sequences of random graphs}\label{sec:random seq}
\subsection{Almost sure convergence for sequences of random graphs} \label{random seq}
By an abuse of notation, we say that  a sequence of random graphs converges to a limit vector $L$ if a sequence of graphs drawn from the sequence of random graph models almost surely converges $L$. \Cref{rg con con} gives a method for showing that a sequence of random graphs converges, which we apply to describe the limits of sequence of E-R graphs (\Cref{er ex}). We will again apply \Cref{rg con con} when we discuss the convergence of sequences of ROC graphs (\Cref{deviate} and \Cref{roc seq}).

\begin{definition}[convergence of random graph sequences] Let $M=(M_i)$ be a sequence of random graph models. Let $S$ be a sequence of graphs $(S_i)$ where $S_i \sim M_i$. We say the sequence of random graphs $M$ converges to $L$ if a sequence $S$ drawn from $M$ almost surely converges to $L$.
\end{definition}

\begin{lemma} \label{rg con con}
	Let $M=(M_i)$ be a sequence of random graph models, and let $(S_i)$ be a sequence of graphs where $S_i \sim M_i$. Let $\ve>0$ and $A_{i,\ve, \alpha}(w_j)$ be the event that $|W_j(S_i,\alpha) - w_j| \geq \ve$. 
	\begin{enumerate}
		\item 	If for all $j$ and $\ve>0$ $$\sum_{i=1}^\infty \Pr{A_{i,\ve, \alpha}(w_j)}< \infty,$$ then $M$ converges to $L= (w_3, w_4, \dots  )$ with sparsity exponent $\alpha$.
		\item If the above hypothesis holds for all $j \leq k$, then $M$ $k$-converges to $L=(w_3, w_4, \dots w_k)$ with $k$-sparsity exponent $\alpha$.
		\item  Let $D(S_i)$ be the random variable for the average degree of a vertex in $S_i$, let $d_i = \E {D(S_i)}$, and let $n_i$ be the number of vertices of $S_i$. 
		If $\lim_{i \to \infty} \frac{\E {W_j(S_i)} }{n_i d_i^{ 1+ \alpha(j-2)}}=w_j$ then there exists an index $i_0$ and a constant $C$ such that  
		$$\sum_{i=1}^\infty \Pr{A_{i,\ve, \alpha}(w_j)} \leq C+  \sum_{i=i_0}^\infty  \frac{\Var[]{D(S_i)}}{ d_i^2} + \frac{\Var[]{W_j(S_i)}}{ \brac{n_id_i^{ 1+ \alpha(j-2)}}^2 }.$$ 
	\end{enumerate}
	
\end{lemma}

\begin{proof} We begin with (1) and (2). Fix $j$. To show that $W_j(S_n, \alpha) \to w_j$ almost surely, it suffices to show that for all $\ve>0$, $\Pr{ A_{i,\ve, \alpha}(w_j) \text{ occurs infinitely often}}=0$. By the Borel Cantelli Lemma $\sum_{n=1}^\infty \Pr{A_{i,\ve, \alpha}(w_j)}< \infty$ implies $\Pr{ A_{i,\ve, \alpha}(w_j) \text{ occurs infinitely often}}=0$. Statements (1) and (2) follow from the fact that a countable intersection of almost sure events occurs almost surely.
	
	For (3), we apply \Cref{deviation bound} which bounds the probability $W_j(S_i,\alpha)$ deviates from expectation by separately bounding the probabilities that the number of edges and the number of closed $j$-walks in $S_i$ deviate from expectation.  Let $g_i=w_j-  \frac{\E {W_j(G_i)} }{n_i d_i^{ 1+ \alpha(j-2)}}$, and so $g_i=o(1)$. Let $i_0$ be such that for all $i\geq i_0$, $|g_i|< \ve/4$, and let $c= \min \{ \delta, \ve/4 \}$. By \Cref{rg con con}(3) for all $i \geq i_0$ $$\Pr{A_{i,\ve, \alpha}(w_j)} \leq \frac{1}{c^2} \brac{ \frac{ \Var{D(G_i)}}{d_i^2}+ \frac{\Var{W_j(G_i)}}{\brac{n_i d_i^{1+\alpha(j-2)}}^2}}.$$
	The claim follows from the observation that
	$$\sum_{i=1}^\infty \Pr{A_{i,\ve, \alpha}(w_j)} \leq i_0 +  \frac{1}{c^2}\sum_{i=i_0}^\infty  \frac{\Var[]{D(S_i)}}{ d_i^2} + \frac{\Var[]{W_j(S_i)}}{ \brac{n_id_i^{ 1+ \alpha(j-2)}}^2 }.$$ 
\end{proof}

\begin{lemma}\label{deviation bound}
	Let $S$ be a random graph on $n$ vertices.  Let $\ve>0$ and $A_{\ve, \alpha}(w_j)$ be the event that $|W_j(S,\alpha) - w_j| \geq \ve$.  Let $D(S)$ be the random variable for the average degree of a vertex in $S$, and let $d = \E {D(S)}$. Let $g=w_j-  \frac{\E {W_j(S)} }{n d^{ 1+ \alpha(j-2)}}$. For $|g|< \ve/2$, $\delta=\min\left\{ \frac{\ve}{2j(w_j+\ve)}, \frac{1}{2(j-1)^2} \right\}$ and $\lambda=\ve/2- |g|$, $$\Pr{A_{\ve, \alpha}(w_j)} \leq  \frac{\Var[]{D(S)}}{ \delta^2 d^2} + \frac{\Var[]{W_j(S)}}{\lambda^2 \brac{nd^{ 1+ \alpha(j-2)}}^2 }.$$
\end{lemma}

\begin{proof}
	Observe that if $A_{i,\ve, \alpha}(w_j)$ holds, then  for any $\delta>0$ at least one of the following events hold: 
	\begin{enumerate}[ \quad  (a)]
		\item $|D(S)-d|> \delta d$
		\item $W_j(S)\geq (w_j+ \ve)\brac{n \brac{d (1-\delta)}^{1+\alpha(j-2)}}$
		\item $W_j(S)\leq (w_j- \ve)\brac{n \brac{d (1+\delta)}^{1+\alpha(j-2)}}$.
	\end{enumerate}
	When (a) does not hold $$\frac{W_j(S)}{n \brac{d (1+\delta)}^{1+\alpha(j-2)}}\leq W_j(S,\alpha) \leq\frac{W_j(S)}{n \brac{d (1-\delta)}^{1+\alpha(j-2)}} .$$ Assume (a) does not hold and $A_{\ve, \alpha}(w_j)$.
	If $W_j(S, \alpha) \geq w_j + \ve$ then (b) holds. If $W_j(S, \alpha) \leq w_j - \ve$ then (c) holds. The observation follows. 
	
	We now give a bound on the probability of (b) or (c). Let $\gamma^{-}= 1- (1-\delta)^{1+ \alpha(j-2)}$ and $\gamma^{+}=  (1+\delta)^{1+ \alpha(j-2)}-1$. We write $w_j= \frac{\E {W_j(S)}}{n d^{1+\alpha(j-2)}}+g$. Statement (b) becomes 
	$$W_j(S)- \E {W_j(S)} \geq \brac{ \ve +g - \gamma^{-}\brac{w_j+\ve}}n d^{1+\alpha(j-2)},$$
	and statement (c) becomes 
	$$W_j(S)- \E {W_j(S)} \leq \brac{  \gamma^{+}\brac{w_j-\ve}-\ve +g }n d^{1+\alpha(j-2)}.$$
	Under the assumptions that $\delta=\min\left\{ \frac{\ve}{2j(w_j+\ve)}, \frac{1}{2(j-1)^2} \right\}$ and $\alpha\leq 1$, 
	$$\gamma^{-} =1- (1-\delta)^{1+ \alpha(j-2)} \leq \delta (1+ \alpha(j-2)) < \delta j \leq \frac{\ve}{2(w_j +\ve)} $$
	$$\gamma^{+}  = (1+\delta)^{1+ \alpha(j-2)} -1 \leq  \delta(j-1) + \sum_{i=2}^{j-1} \delta^i {j-1 \choose i}\leq \delta(j-1) +2 \delta^2 (j-1)^2 < \delta j \leq \frac{\ve}{2(w_j -\ve)}.$$
	Let $\lambda= \ve/2- |g|$ and note 
	$$\ve +g - \gamma^{-}\brac{w_j+\ve} \geq \lambda\quad \text{ and } \quad  \ve -g -\gamma^{+}\brac{w_j-\ve} \geq \lambda. $$
	It follows from Chebyshev's inequality that
	$$\Pr{ (b) \text{ or } (c) } \leq \Pr{ |W_j(S)-\E {W_j(S)}| \geq c(\delta) n d^{1+\alpha(j-2) }} \leq \frac{ \Var[]{W_j(S)}}{\brac{\lambda nd^{1+\alpha(j-2)}}^2}.$$
	
	Finally, we apply Chebyshev's inequality to bound the probability of (a), apply a union bound for the event $A_{i,\ve, \alpha}(w_j)$, and obtain
	$$\Pr{A_{\ve, \alpha}(w_j)}\leq \Pr{(a)}+ \Pr{ (b) \text{ or } (c) }  \leq  \frac{\Var[]{D(S)}}{ \delta^2 d^2} + \frac{\Var[]{W_j(S)}}{\lambda^2 \brac{nd^{ 1+ \alpha(j-2)}}^2 }.$$
\end{proof}

Our final example is sequences of Erd\H{o}s-R\'enyi random graphs. These demonstrate some of the subtler issues with defining limits.

\begin{lemma}[Erd\H{o}s-R\'enyi sequence] \label{er ex}
	Let $(G_n) \sim G\brac{n^{2\ell}, n^{2- 2\ell}}$ for $\ell >1$. We denote the $j^{th}$ Catalan number $Cat_j=\frac{1}{j+1} { 2j \choose j}$.
	\begin{enumerate}
		\item For $k< 2\ell$, the $k$-sparsity exponent of  $(G_n)$ is $1/2$ and the $k$-limit is $(w_3, w_4, \dots w_k)$ where $w_j=0$ for $i$ odd and $w_j = Cat_{i/2}$ for $i$ even.
		\item For $k= 2\ell$, the $k$-sparsity exponent of  $(G_n)$ is $1/2$ and the $k$-limit is $(w_3, w_4, \dots , w_{k-1}, \overline{w}_k)$ where $w_j=0$  for odd $j$, $w_j = Cat_{j/2}$ for even $j$, and $\overline{w}_k = w_k +1$.  
		\item For $k> 2\ell$, the sparsity exponent of  $(G_n)$ is $\frac{k-\ell-1}{k-2}$ and the $k$-limit is $(w_3, w_4, \dots , w_k)$ where $w_j=0$  for $j<k$, $w_k=1$.
		\item The sparsity exponent of $(G_n)$ is $1$ and the limit is $(0,0, \dots)$. 
	\end{enumerate}
\end{lemma}

First we compute the expectation and variance of the number of closed $i$-walks in a E-R random graph.
\begin{lemma} \label{mean and var}
	Let $G \sim G(n, d/n)$. Let $W_j(G)$ be the random variable for the number of closed $j$ walks in $G$. Then 
	$$\E{W_j(G)}= d^j + Cat_{j/2} nd^{\lfloor j/2\rfloor} + \Theta\brac{d^{j-1}+ nd^{\lfloor j/2\rfloor-1}}$$
	$$\Var{W_j(G)}= \Theta\brac{d^{2j-1}+ n^2d^{2\lfloor j/2\rfloor-1}+ nd^{\lfloor j/2 \rfloor+j-1}}.$$
\end{lemma}

\begin{proof}
	Let $W_j^{a,b}(G_n)$ be the number of closed $j$ walks involving  $a$ vertices and $b$ edges, so $b \leq j$ and either $a\leq b$ (the walk contains a cycle) or $a=b+1$ and $b \leq j/2$ (the walk traces a tree). Let $f(a,b,j)$ be the number of closed $j$-walks with $b$ total edges on $a$ labeled vertices $1,2, \dots a$ such that the order in which the vertices are first visited is $1,2, \dots a$.  Note $f(j,j,j)=1$ and $f( j/2 +1,j/2,j)=Cat_{j/2  }$ for $j$ even (because there are $Cat_{b}$ ordered trees on $b$ edges, see \cite{sta97}).  Let $\zeta(j)$ be one if $j$ is even and zero otherwise. We split the sum based on whether the walk contains a cycle or traces a tree and compute
	\begin{align*}
	\E{W_j(G)}&= \sum_{b=1}^j \sum_{a=1}^{b+1} \E{W_j^{a,b}(G)}= \sum_{b=1}^j \sum_{a=1}^{b+1}f(a,b,j) \frac{n!}{(n-a)!} \bfrac{d}{n}^b \\
	& = \sum_{b=3}^j \sum_{a=1}^{b}f(a,b,j)  \frac{n!}{(n-a)!} \bfrac{d}{n}^b + \sum_{b=1}^{\lfloor j/2 \rfloor}  f(b+1,b,j)  \frac{n!}{(n-(b+1))!} \bfrac{d}{n}^b \\
	&= d^j + \zeta(j) Cat_{j/2} nd^{ j/2} + \Theta\brac{d^{j-1}+ nd^{\lfloor j/2\rfloor-1}}.
	\end{align*}
	
	To find the variance of $W_j(G)$ we compute the expectation squared. Let $P_j^{a,b}(G)$ be the number of pairs of closed $j$ walks involving a total of $a$ vertices and $b$ edges. Let $g(a,b,j)$ be the number of pairs of closed $j$-walks with $b$ total edges on $a$ labeled vertices $1,2, \dots a$ such that the order in which the vertices are first visited is $1,2, \dots a$ when the first walk is traversed then the second walk. Note $g(2j, 2j, j)= 1$ and $g( j +2 , j, j)= \brac{Cat_{ j/2}}^2$ (since there are $\brac{Cat_{b}}^2$ ways to pick two disjoint ordered trees on $b$ edges).
	
	We split the sum  based on whether both walks contain a cycle, or both trace trees, or one traces a tree and one traces a cycle and compute
	\begin{align*}
	\E{ W_j(G)^2}&= \sum _{b=1}^{2j} \sum_{a=1}^{b+2} \E{P_j^{a,b}(G)} = \sum _{b=1}^{2j} \sum_{a=1}^{b+2} g(a,b,j) \frac{n!}{(n-a)!} \bfrac{d}{n}^b\\
	&= \sum _{b=1}^{2j} \sum_{a=1}^{b} g(a,b,j) \frac{n!}{(n-a)!} \bfrac{d}{n}^b +\sum _{b=1}^{2\lfloor j/2 \rfloor}  g(b+2,b,j) \frac{n!}{(n-(b+2))!} \bfrac{d}{n}^b\\
	&\quad \quad +\sum _{b=1}^{2j} \sum_{a=1}^{b+1} g(a,b,j) \frac{n!}{(n-a)!} \bfrac{d}{n}^b\\
	&= d^{2j} + \zeta(j) \brac{Cat_{j/2}}^2 n^2d^{j} + \Theta\brac{d^{2j-1}+ n^2d^{2\lfloor j/2\rfloor-1}+ nd^{\lfloor j/2 \rfloor+j-1}}.
	\end{align*}
	It follows 
	$$\Var{W_j(G)}=  \E{ W_j(G)^2}-\E{W_j(G)}^2= \Theta\brac{d^{2j-1}+ n^2d^{2\lfloor j/2\rfloor-1}+ nd^{\lfloor j/2 \rfloor+j-1}}.$$
\end{proof}

Now we use these computations and apply \Cref{rg con con}  to prove  \Cref{er ex}.
\begin{proof}(of \Cref{er ex}).
	By \Cref{mean and var}
	$$\E{W_j(G_n)}=
	\begin{cases}
	n^{2j}+ Cat_{j/2 } n^{2\ell + j}+ o\brac{n^{2j}+  n^{2\ell + 2 \lfloor j/2 \rfloor} } & \text{ $j$ is even}\\
	n^{2j}+  o\brac{n^{2j}+  n^{2\ell + 2 \lfloor j/2 \rfloor} } & \text{ $j$ is odd}.
	\end{cases}
	$$
	We compute the  $k$-sparsity exponent $$\alpha_k= \inf_{a \in [1/2,1]}\left\{ a \given \E{ W_j(G_n)}=  O\brac{n^{ 2 \ell +2 + 2\alpha (j-2)} }\text{ for all } j\leq k\right\}= \max \left\{\frac{1}{2},  \frac{k-\ell-1}{k-2}\right\},$$ and the sparsity exponent 
	$$\alpha= \inf_{a \in [1/2,1]}\left\{ a \given \E{ W_j(G_n)}=  O\brac{n^{ 2 \ell +2 + 2\alpha (j-2)} }\text{ for all } j \right\}=1.$$
	
	Note for each of the cases outlined in the statement, $w_j = \lim_{i \to \infty} \frac{\E {W_j(G_n)} }{n^{ 2 \ell +2 + 2\alpha (j-2)}}$ where $\alpha$ is the corresponding sparsity exponent. To prove convergence for cases 1-3, we apply \Cref{rg con con}(2) and for case 4 we apply \Cref{rg con con}(1). By \Cref{rg con con}(3), it remains to show that 
	\begin{equation} \label{sum of var}
	\sum_{n= n_0}^\infty  \brac{ \frac{ \Var{D(G_n)}}{d_n^2}+ \frac{\Var{W_j(G_n)}}{\brac{n^{ 2 \ell +2 + 2\alpha (j-2)}}^2}}< \infty.
	\end{equation}
	
	Note $D(G_n) \sim Bin\brac{ { n^{2\ell}\choose 2} , n^{2- 2\ell}}$, and so $\Var{D(G_n)}= { n^{2\ell} \choose 2 }n^{2- 2\ell} \brac{ 1- n^{2- 2\ell}}$ and $d_n = {n^{2\ell}\choose 2}n^{2- 2\ell}$. It follows 
	$$\sum_{n= 2}^\infty  \frac{ \Var{D(G_n)}}{d_n^2}= \sum_{n= 2}^\infty \frac{{n^{2\ell}\choose 2}n^{2- 2\ell}\brac{ 1- n^{2- 2\ell}} }{\brac{{n^{2\ell}\choose 2}n^{2- 2\ell}}^2}\leq \sum_{n= 2}^\infty 2n^{-2\ell-2}<\infty.$$
	
	We show the sum of the variance term is finite by considering the cases separately. By \Cref{mean and var}, 
	$$\Var{W_j(G_n)}= \Theta\brac{n^{4j-2}+ n^{4\ell + 4\lfloor j/2\rfloor-2}+ n^{2\ell + 2\lfloor j/2 \rfloor+2j-2}},$$
	and so 
	$$X:= \frac{\Var{W_j(G_n)} }{\brac{n^{ 2 \ell +2 + 2\alpha (j-2)}}^2} =\Theta\brac{ n^{4j-6-4\ell-4\alpha\brac{j-2}}+ n^{ -2 + (j-2) (2-4\alpha)} +n^{-2 \ell +(j-2)(3-4 \alpha)}}.$$
	
	For (1) and (2),  $\alpha= 1/2$,  $j\leq k \leq 2\ell$, and so $X = \Theta \brac{ n^{2j-4\ell-2}+ n^{-2}+ n^{-2\ell +j-2}} =O \brac{n^{-2}}.$
	For (3), $\alpha=\frac{k-\ell -1}{k-2}$, $k >2 \ell$, and $j \leq k$. 
	Since $\alpha \geq \frac{j- \ell -1}{j-2}$, $\Theta \brac{ n^{4j-6-4\ell-4\alpha\brac{j-2}}} = O \brac{n^{-2}}$. Since  $\alpha>1/2$, $\Theta \brac{n^{ -2 + (j-2) (2-4\alpha)}}= O\brac{n^{-2}}$. Since $\alpha>1/2$ and $k> 2\ell$,  $\Theta \brac{n^{-2 \ell +(j-2)(3-4 \alpha)}}=O\brac{n^{-2 \ell +k-2}}= O\brac{n^{-2}}$. It follows that $X = O \brac{n^{-2}}.$
	For (4) $\alpha=1$ and $j \geq 3$, and so $X =\Theta\brac{ n^{2-4\ell}+ n^{ 2 -2j } +n^{-2 \ell -j+2}}=O \brac{n^{-2}}.$
	Therefore in all cases 
	$$\sum_{n=2}^\infty  \frac{\Var{W_j(G_n)}}{\brac{n^{ 2 \ell +2 + 2\alpha (j-2)}}^2}= 
	\sum_{n=2}^\infty O\brac{n^{-2}}< \infty,$$ and the statement follows from \Cref{sum of var}.
\end{proof}

\subsection{The convergence of sequences of ROC graphs} \label{roc sec sec}

\Cref{achievability defn} states that the vector achieved by a ROC family is the expected walk count of a ROC graph from that family normalized with respect to {\em expected} degree. We now justify this definition by showing that for $G\sim ROC(n,d,\mathcal{D})$, the probability that  the normalized closed walk count $W_j(G, \alpha)$ deviates from the limit $w_j$ achieved by the family tends to zero as $d$ grows (\Cref{deviate}). Moreover, we show that the sequence $G_i \sim ROC(n_i, d_i, \mathcal{D})$ almost surely converges to the limit achieved by the family when $n_i$ and $d_i$ grow sufficiently fast (\Cref{roc seq}). 

\begin{theorem} \label{deviate}
	Let $(w_3, w_4, \dots )$ be the limit achieved by the ROC family $\mathcal{D}=(\mu, a)$. Let $G \sim ROC(n,d, \mathcal{D})$ where $d= o\brac{n^{1/((1-a)k+2a-1)}}$ and $|w_j-  \frac{\E {W_j(G)} }{n d^{ 1+ \alpha(j-2)}}|< \ve/2$.
	Then for $\alpha=\max\{a, 1/2\}$, $$\Pr{|W_j(G, \alpha)- w_j|>\ve}=  f(d,a)$$ where 
	$$f(d,a)= 
	\begin{cases} 
	O \brac{  d^{-1+2a}+ \frac{d^{k/2-1}}{n}+\frac{d^{(k-1)(2a-1)}}{n}}& a<1/2\\
	O \brac{\frac{d^{k/2-1}}{n}+ d^{-1/2}}  & a=1/2\\
	O \brac{ d^{1-2a} + \frac{d^{(1-a)(k-2)}}{n} }& 1/2<a<1\\
	O \brac{d^{-1} +   \frac{ d}{n}}& a=1.
	\end{cases}
	$$
\end{theorem}

\begin{corollary}\label{roc seq}
	Let $(w_3, w_4, \dots )$ be the limit achieved by the ROC family $\mathcal{D}=(\mu, a)$. Let $G_i \sim ROC(n_i,d_i, \mathcal{D})$ where $d_i= o\brac{n_i^{1/((1-a)k+1-2a)}}$ and $f(d_i,a)$ is  defined for $G_i$ as in \Cref{deviate}. If $\sum_{i=1}^\infty f(d_i,a) <\infty$ the sequence of graphs $(G_i)$ converges to the limit $(w_3, w_4, \dots )$ with sparsity exponent $\alpha= \max\{a, 1/2\}$.
\end{corollary}

\paragraph{Achieving normalized and unnormalized closed walk counts.} A ROC family $(\dist, a)$ that achieves the limit of a sequence of graphs $(G_i)$ with the appropriate sparsity exponent is a sampleable model that produces graphs in which the normalized closed walk counts match the limit up to an error term that tends to zero as the size of the sampled graph grows. The following remark describes when a sequence of graphs drawn from the ROC model also matches the unnormalized closed walk counts of the sequence $(G_i)$ term by term. The remark is stated for $k$-convergence and $k$-limits, but an analogous statement holds for full convergence and limits. 

\begin{remark} 
	Let $(G_i)$ be a sequence of graphs each with $n_i$ vertices and average degree $d_i$ such that $(G_i)$ is $k$-convergent with $k$-limit $L$ and $k$-sparsity exponent $\alpha$. Suppose the ROC family $\mathcal{D}=(\dist, a)$ achieves the limit $L$ with sparsity exponent $\alpha$. 
	\begin{enumerate}
		\item If $d_i =o\brac{n_i^{1/((1-a)k+2a-1)}}$ then the sequence $(H_i)$ with $H_i \sim ROC(n_i,d_i, \dist, a)$ has the property that for sufficiently large $i$ and $j \leq k$, in expectation $G_i$ and $H_i$ have the same average degree and number of closed $j$-walks up to lower order terms with respect to $d_i$.
		\item It is possible to construct other sequences $(H_i)$ with  $H_i\sim ROC( n_i , f(n_i), \mathcal{D})$ such that for sufficiently large $i$, in expectation $H_i$ and $G_i$  have different edge densities, but have the same normalized number of closed walks up to lower order terms.
	\end{enumerate} 
\end{remark}

To prove \Cref{deviate}, we will apply \Cref{deviation bound} , which bounds the probability that the normalized walk count deviates from expectation in terms of the probability that the number of edges deviates and the probability that the walk count deviates.  \Cref{var edge,var walk} compute these quantities. 

\begin{lemma}\label{var edge} 
	Let $G\sim ROC(n,d, \mathcal{D})$, $\mathcal{D}= (\mu, a)$. Let $D(G)$ be the random variable for the average degree of $G$. Then $$\E{ D(G)}= d \quad \text{ and } \quad \Var{D(G)}= \Theta\bfrac{d^{1+a}}{n}.$$
\end{lemma}

\begin{proof} Let $D(G)= \frac{1}{n}\sum_{v,w, u} X_{u,v,w}$ where $X_{u,v,w}$ is an indicator random variable for the event that the edge $w,v$ is added in the $u^{th}$ community. Note $$\E{X_{u,v,w}}=\Pr{ X_{u,v,w}}= \sum_{i \in B^c} \mu_i \bfrac{m_i d^a}{n}^2 q_i +\sum_{i \in B} \mu_i 2\bfrac{m_i d^a}{n}^2 q_i = \frac{d^{2a}}{xn^2}.$$ There are $n(n-1)$ pairs $w,v$ and $xnd^{1-2a}$ communities $u$. Therefore $$\E{D(G)}= \frac{1}{n}\sum_{v,w, u}\E{ X_{u,v,w}}=  \frac{1}{n}n(n-1) xnd^{1-2a} \frac{d^{2a}}{xn^2}=d-\frac{d}{n}. $$
	
	Next we compute the expected pairs of edges, $\E{ D(G)^2}$. We fix the potential edge defined by vertices $a$ and $b$ and community $x$ and sum over all other potential edges.
	\begin{align*}
	\E{D(G)^2}&= \bfrac{1}{n^2}n(n-1)(xnd^{1-2a}) \sum_{u,v,w} \Pr{ X_{u,v,w} \text{ and } X_{a,b,x}}\\
	&= x(n-1)d^{1-2a} \bigg(\sum_{u,v,w\not=x} \Pr{ X_{u,v,w}} \Pr{X_{a,b,x}} +\sum_{u,v\not \in \{a,b\}, x} \Pr{ X_{u,v,w}} \Pr{X_{a,b,x}}\\
	& \quad +2\sum_{u=a, v \not=b,x} \Pr{ X_{u,v,w} \text{ and } X_{a,b,x}} + \Pr{X_{a,b,x}}   \bigg)\\
	&=  x(n-1)d^{1-2a} \bfrac{d^{2a}}{xn^2} \bigg( n(n-1) \brac{xnd^{1-2a}-1} \bfrac{d^{2a}}{xn^2}+ (n-2)(n-3) \bfrac{d^{2a}}{xn^2}\\
	& \quad + (n-3) \Theta \bfrac{d^a}{n}  + 1\bigg)\\
	&=\brac{ d-\frac{d}{n}}^2+\Theta\bfrac{d^{1+a}}{n}.
	\end{align*}
	It follows that $$\Var{ D(G)}= \E{D(G)^2} - \E{D(G)}^2=\Theta\bfrac{d^{1+a}}{n}.$$
\end{proof}

\begin{lemma}\label{var walk}
	Let $G\sim ROC(n,d, \mathcal{D})$, $\mathcal{D}= (\mu, a)$ with $d=o \brac{n^{1/((1-a)k+2a-1)}}$.  Let $W_k(G)$ be the random variable for the number of closed $k$-walks in $G$. Then $$\Var{W_k(G)}= 
	\begin{cases}
	O\brac{ \brac{n d^{k/2}}^2 \brac{  d^{-1+2a}+ \frac{d^{k/2-1}}{n}+\frac{d^{(k-1)(2a-1)}}{n}}}& a<1/2\\
	O \brac{ \brac{n d^{k/2}}^2 \brac{\frac{d^{k/2-1}}{n}+ d^{-1/2}}  }& a=1/2\\
	O \brac{ \brac{n d^{1+a(k-2)}}^2 \brac{ d^{1-2a} + \frac{d^{(1-a)(k-2)}}{n} }}& 1/2<a<1\\
	O \brac{ \brac{n d^{k-1}}^{2} \brac{d^{-1} +   \frac{ d}{n}}}& a=1.\\
	\end{cases}
	$$
\end{lemma}

\begin{proof} 
	We give an upper bound on $\E{ W_k(G)^2}$ by counting the expected number of pairs of walks. Let $P'_k(G)$ be the random variable for the number of pairs of $k$-walks in $G$ that do not intersect, and let $P''_k(G)$ be the random variable for the number of pairs of $k$-walks in $G$ that do intersect. Note that two $k$-walks that intersect can be thought of as a $2k$ walk. The expected number of $2k$ walks in $G$ is $\Theta\brac{ nd^{1+a(2k-2)} }$ (see \Cref{polynomials}), and so $\E{P''_k(G)}=\Theta\brac{ nd^{1+a(2k-2)} }$. 
	
	To compute $P'_k(G)$ we recall the partition of possible walks with permutation type $P_S$ into sets $A, B(i)$ and $B(ii)$ as described in the proof of \Cref{exp by class}. 
	Let $S(t)$ be the set of $S \in \mathcal{S}_k$ such that $\sum t_i=t$. For $S\in S(t)$, the expected number of walks with type $A$ is at most $n d^{(1-2a)t+ak} \brac{\prod_{i=1}^j c(a_i)^{t_i}}$ (see \Cref{a walks}). 
	The expected number of walks with type $B(i)$ is $ \Theta \brac{d^{k-t}}$ (see \Cref{type i}), and the expected number of walks with type $B(ii)$ is
	$\Theta \brac{nd^{(1-2a)t+ak+a-1} }$ when $t\not=1$ and $0$ when $t=1$ (see \Cref{type ii}).  
	Therefore 
	$$
	\E{P'_k(G) }\leq
	\brac{\sum_{t=1}^{\lfloor k/2\rfloor }\sum_{S \in S(t)} |P_S| n d^{(1-2a)t+ak} \brac{\prod_{i=1}^j c(a_i)^{t_i}} +\Theta\brac{ d^{k-t}+ \zeta_t nd^{(1-2a)t+ak+a-1}} }^2 
	$$
	where $\zeta_t=0$ if $t=1$ and $\zeta_t=1$ otherwise. 
	We simplify and obtain
	$$
	\E{P'_k(G)} =
	\begin{cases}
	\E{W_k(G)}^2+ O\brac{ n d^{k/2} \brac{  n d^{k/2- 1+2a}+ d^{k-1}+ nd^{k/2+a-1}}} & a<1/2\\
	\E{W_k(G)}^2+ O \brac{ n d^{k/2} \brac{d^{k-1}+ nd^{ak+a-1}}} & a=1/2\\
	\E{W_k(G)}^2+ O \brac{ n d^{1+a(k-2)} \brac{n d^{2+a(k-4)} + d^{k-1}+   nd^{1+a(k-3)}  }} & 1/2<a<1\\
	\E{W_k(G)}^2+O \brac{ n d^{k-1} \brac{n d^{k-2} + d^{k-1}+   nd^{k-2}  }} & a=1.\\
	\end{cases}
	$$
	
	Finally we compute \begin{align*} \Var{W_k(G)}&= \E{W_k(G)^2}+ \E{W_k(G)}^2= \E{P'_k(G)} + \E{P''_k(G)}- \E{W_k(G)}^2\\
	&= \begin{cases}
	O\brac{ n d^{k/2} \brac{  n d^{k/2- 1+2a}+ d^{k-1}+ nd^{k/2+a-1}}}  + O\brac{ nd^{1+a(2k-2)} }& a<1/2\\
	O \brac{ n d^{k/2} \brac{d^{k-1}+ nd^{ak+a-1}}} +O\brac{ nd^{k} }& a=1/2\\
	O \brac{ n d^{1+a(k-2)} \brac{n d^{2+a(k-4)} + d^{k-1}+   nd^{1+a(k-3)}  }}+O\brac{ nd^{1+a(2k-2)} } & 1/2<a<1\\
	O \brac{ n d^{k-1} \brac{n d^{k-2} + d^{k-1}+   nd^{k-2}  }} +O\brac{ nd^{2k-1} }& a=1,\\
	\end{cases}
	\end{align*}
	and the statement follows by simplifying the above expressions. 
\end{proof}

We now prove \Cref{deviate} by applying \Cref{deviation bound}.
\begin{proof} (of \Cref{deviate})
	Let $g=w_j-  \frac{\E {W_j(G)} }{n d^{ 1+ \alpha(j-2)}}$, $\delta=\min\left\{ \frac{\ve}{2j(w_j+\ve)}, \frac{1}{2(j-1)^2} \right\}$ and $\lambda=\ve/2- |g|$. By \Cref{deviation bound,var edge,var walk}, \begin{align*}
	\Pr{|W_j(G,\alpha)-w_j|>\ve} \leq  \frac{\Var[]{D(G)}}{ \delta^2 d^2} + \frac{\Var[]{W_j(G)}}{\lambda^2 \brac{nd^{ 1+ \alpha(j-2)}}^2 }
	\leq O \brac{ \frac{d^{a-1}}{n} }+f(d,a) =f(d,a).
	\end{align*}
\end{proof}

\Cref{roc seq} follows directly from \Cref{deviate} and part 3 of \Cref{rg con con}.

\subsection{Erd\H{o}s-R\'enyi sequences}
We consider ROC approximations of the sequences of Erd\H{o}s-R\'enyi graphs given in \Cref{er ex}.

\begin{theorem}
	Let $\ell >1$. Let $(G_n) \sim G(n^{2 \ell}, n^{2\ell-2})$.
	\begin{enumerate}
		\item For $k< 2\ell$, the $k$-limit of $(G_n)$ is achieved by any ROC family with $a< 1/2$.
		\item For $k\geq  2\ell$, the $k$-limit  of $(G_n)$  is not $k$-achievable by any ROC family. However, for any $\ve>0$, there exists a ROC $k$-achievable vector that is $L_\infty$ distance at most $\ve$  from the $k$-limit. 
		\item The sparsity exponent of $(G_n)$ is $1$ and the limit is $(0,0, \dots)$. This limit is not ROC fully achievable. However, for any $\ve>0$, there exists a ROC fully achievable vector that is $L^\infty$ distance at most $\ve$  from $(0,0, \dots)$. 
	\end{enumerate}
\end{theorem}

\begin{proof}
	For $k< 2\ell$ the sparsity exponent of  $(G_n)$ is $1/2$ and the $k$-limit is $(w_3, w_4, \dots w_k)$ where $w_i=0$ for $i$ odd and $w_i = Cat_{i/2}$ for $i$ even. By \Cref{polynomials}, this is the limit for any ROC family with $a<1/2$.
	
	For $k= 2\ell$, the $k$-sparsity exponent of  $(G_n)$ is $1/2$ and the $k$-limit is $(w_3, w_4, \dots , w_{k-1}, \overline{w}_k)$ where $w_i=0$  for odd $i$, $w_i = Cat_{i/2}$ for even $i$, and $\overline{w}_k = w_k +1$. By \Cref{polynomials}, in order to approximate the vector it is necessary to have 
	$\mu$ be such that $x\sum \mu_i (m_i q_i)^j=c_j$ where $(c_3, c_4, \dots c_k)=T(w_3, w_4, \dots , w_{k-1}, \overline{w}_k)$ is the cycle transform, so $c_j=0$ for $j<k$ and $c_k=1$. Since $\mu_i, q_i, m_i >0$, $c_k=1$ implies $c_j \not =0$ for all $j<k$. Therefore, the vector cannot be achieved exactly by ROC. 
	
	We now show that it is possible to achieve a vector that is arbitrarily close to the desired vector with respect to the $L_\infty$ metric.
	Note that for $\mu$ the distribution on one point $m= \delta^\frac{1-k}{k-2}$ and $q= \delta$, the resulting ROC family $(\mu, 1/2)$ has $$c_j=m^{j-2}q^{j-1}= \delta^{\frac{(j-2)(1-k)}{k-2}+j-1}.$$ Therefore $c_k=1$ and $c_j$ for $j<k$ can be made  arbitrarily small by decreasing $\delta$. To achieve $L_\infty$ distance $\ve$, choose $\delta$ small enough so that $\max_{j<k} w_j=max_j T(c_3, c_4, \dots c_k)_j < \ve$.

	For $k> 2\ell$, the sparsity exponent of  $(G_n)$ is $\frac{k-\ell-1}{k-2}$ and the $k$-limit is $(w_3, w_4, \dots , w_k)$ where $w_i=0$  for $i<k$, $w_k=1$. Therefore, by \Cref{polynomials}, to approximate the vector we likewise need $\mu$ be such that $x\sum \mu_i (m_i q_i)^j=w_j$ where $c_j=w_j=0$ for $j<k$ and $w_k=c_k=1$, and the result is as in the previous case.
	
	Similarly, we can approximate the vector $(0,0, \dots)$ with sparsity exponent $1$ up to arbitrarily small error with respect to the $L_\infty$ distance. Note that for $\mu$ the distribution on one point $m$ and $q$, the ROC family $(\mu, 1)$ has $w_k= m^{k-2} q^{k-1}$. Therefore it is possible to achieve error $\ve$ by selecting $m$ and $q$ such that $\max_k m^{k-2} q^{k-1}< \ve$.
\end{proof}

\section{Other families of graphs}
\subsection{Rook graphs}\label{sec:rook}
\begin{proof}(of \Cref{rook})
	The rook's graph is the strongly regular graph on $n=k^2$ vertices with degree $d=2k-2$ such that each pair of adjacent vertices have $\lambda=k-2$ common neighbors and each pair of non-adjacent vertices have $\mu=2$ common neighbors.  The classical result \cite{bro12} states that the eigenspectrum of a strongly regular graph is 
	\begin{align*}
	&d \text{ with multiplicity $1$,}\\
	&\frac{1}{2} \brac{(\lambda- \mu)+ \sqrt{ (\lambda- \mu)^2 + 4(d- \mu)}} \text{ with multiplicity $\frac{1}{2} \brac{ (n-1) - \frac{2d+(v-1)( \lambda- \mu)}{\sqrt{ (\lambda- \mu)^2+4(d -\mu)}}}$, and }\\
	&\frac{1}{2} \brac{(\lambda- \mu)- \sqrt{ (\lambda- \mu)^2 + 4(d- \mu)}} \text{ with multiplicity $\frac{1}{2} \brac{ (n-1) + \frac{2d+(v-1)( \lambda- \mu)}{\sqrt{ (\lambda- \mu)^2+4(d -\mu)}}}$. }
	\end{align*}
	Therefore the eigenspectrum of the rook graph $G_k$ is 
	\begin{align*}
	2k-2 \text{ with multiplicity $1$,  }
	-2 \text{ with multiplicity $(k-1)^2$, and }
	k-2 \text{ with multiplicity $2k-2$.}
	\end{align*} 
	We compute $$\lim_{k \to \infty} W_j(G_k, 1) =\lim_{k \to \infty} \frac{ (2k-2)^j+ (k-1)^2(-2)^j + (2k-2) (k-2)^j }{ k^2(2k-2)^{j-1}}= 2^{2-j}.$$
\end{proof}

\subsection{Generalized hypercubes} \label{gen hyper}
Two generalizations of the hypercube have the same limit and therefore are also totally $k$-achievable. 

\begin{corollary}[Hypercube generalizations]\label{hc generalizations}
	The following sequences of graphs $(G_d)$ converge with sparsity exponent 1/2 to the same limit as the hypercube sequence. \begin{enumerate}
		\item (Hamming generalization) Let $G_d$ be the graph on vertex set $\{0,1, \dots , k-1\}^d$ where two vertices are adjacent if the Hamming distance between their labels is one.
		\item (Cayley generalization) Let $G_d$ be the graph on vertex set $\{0,1, \dots , k-1\}^d$ where two vertices are adjacent if their labels differ by a standard basis vector. 
	\end{enumerate}
\end{corollary}

\begin{proof}
	Since the Hamming and Cayley sequences are locally regular, it suffices to show that the sequences have sparsity exponent $1/2$ and the same vector of normalized cycle counts as the hypercube. Let $D$ denote the degree of the graph $G_d$, so for the Hamming graph $D=d(k-1)$ and for the Cayely graph $D=2d$. 
	
	First we show that both sequences have sparsity exponent $1/2$ by showing that each vertex is in $O\brac{D^{i/2}}$ $i$-cycles (locally regularity guarantees the walk counts are of the same order). We count $i$-cycles by grouping them according to the number of coordinate positions changed during the cycle, as in the proof of \Cref{cube cycles}. The highest order term comes from $i$ cycles in which $i/2$ coordinates are changed. Therefore the number of $i$-cycles at each vertex is $O\brac{d^{i/2}}=O\brac{D^{i/2}}$ and it follows that the sparsity exponent is $1/2$.
	
	Next we compute the cycle vector $c_i=\lim_{d \to \infty} \frac{C(G_d)}{nD^{i/2}}$. The number of $i$-cycles at a vertex that involve changing fewer than $i/2$ coordinates is $o(D^{i/2})$, so such cycles do not contribute to $c_i$. Therefore, while there are odd cycles in the Hamming and Cayley graphs, $c_i=0$ for $i$ odd. We now count the number of $i$-cycles at a vertex that involve changing $i/2$ coordinates. As described in \Cref{cube cycles} there are $s_{i/2} d^{i/2}$ ways to select $i/2$ coordinates and change them in a manner that corresponds to a cycle. In the hypercube, there is only one way to change a single coordinate, so the total number of cycles at a vertex is $s_{i/2} d^{i/2}$.
	
	In the Hamming graph there are $k-1$ ways to change a coordinate since there are $k$ possibilities for each coordinate. Therefore, for the Hamming sequence and $i$ even 
	$$c_i=\lim_{d \to \infty} \frac{s_{i/2} (k-1)^{i/2}}{nD^{i/2}}=\lim_{d \to \infty} =s_{i/2}.$$
	
	In the Cayley graph there are two ways to change a single coordinate (either add one or subtract one). 
	Therefore, for the Cayley sequence and $i$ even 
	$$c_i=\lim_{d \to \infty} \frac{s_{i/2} 2^{i/2}}{nD^{i/2}}=\lim_{d \to \infty} =s_{i/2}.$$
\end{proof}

\begin{remark}
	The above corollary shows that same ROC family $\mathcal{D}=(\mu,a )$ achieves the $k$-limit of the sequence of hypercubes, and the closely related Hamming and Cayley generalizations. 
	This ROC family can produce sequences of ROC graphs $(G_d)$, $G_d \sim ROC(n_d, d_d,\mathcal{D} )$, unique to each of these settings by varying relationship between $n_d$ and $d_d$. A sequence with $n_d= 2^d$ and $d_d=d$ will match the edge density and unnormalized walk counts of the hypercube, whereas a sequence with $n_d=k^d$ and $d_d=d (k-1)$ or $d_d= 2d$ will match the edge density and unnormalized walk counts of the Hamming or Cayley generalization respectively.
\end{remark}

\end{document}